%% file: main.tex
\documentclass[final,a4paper, 11pt]{article}
\RequirePackage{etex} 
\usepackage{authblk}
\usepackage{blindtext}
\usepackage{comment}
\usepackage[utf8]{inputenc}
\usepackage{amsmath}
\usepackage{amsthm}
\usepackage{amsfonts}
\usepackage{lipsum}
\usepackage{amssymb}
\usepackage{mathrsfs}
\usepackage{empheq}
\usepackage{authblk}
\usepackage{mathtools}
\usepackage{siunitx}
\usepackage{soul} 
\usepackage{amsbsy}
\usepackage{braket}
\usepackage[font=footnotesize]{caption}
\usepackage[export]{adjustbox}
\usepackage{xcolor}
\usepackage{multicol}
\usepackage{multirow}
\usepackage{graphicx}
\graphicspath{{IMG/}}
\usepackage{tikz}
\usepackage{pgfplots}
\pgfplotsset{compat=1.6}
\usepackage{placeins}
\usepackage[english]{babel}
\usepackage[margin=0.8in]{geometry}

\usepackage{float}
\usepackage{longtable}
\usepackage{colortbl}
\usepackage{subcaption}
\usepackage{listings}
\usepackage{hyperref}
\usepackage{cleveref}
\usepackage{autonum}
\usepackage{realboxes}
\theoremstyle{plain}
\newtheorem{definition}{Definition}
\numberwithin{definition}{section}

\newtheorem{assumption}{Assumption}
\numberwithin{assumption}{section}

\newtheorem{theorem}{Theorem}
\numberwithin{theorem}{section}

\numberwithin{proposition}{section}
\theoremstyle{plain}
\newtheorem{remark}
{Remark}
\theoremstyle{plain}
\newtheorem{lemma}{Lemma}[section]
\usepackage{tikz}

\usepackage{enumitem}
\usepackage[makeroom]{cancel}
\usepackage{algorithm}
\usepackage{algpseudocode}

\usepackage[square, numbers, sort&compress]{natbib} 
\newcommand\T{\rule{0pt}{2.6ex}}
\newcommand\B{\rule[-1.2ex]{0pt}{0pt}}
\newcommand{\jump}[1]{\left[\mkern-1.5mu \left[#1\right] \mkern-1.5mu\right]}
\newcommand{\avg}[1]{\left\{ \mkern-5mu \left\{#1 \right\} \mkern-5mu \right\}}
\newcommand{\wavg}[1]{\left\{ \mkern-5mu \left\{#1 \right\} \mkern-5mu \right\}_{\omega}}
\newcommand{\wTavg}[1]{\left\{ \mkern-5mu \left\{#1 \right\} \mkern-5mu \right\}_{\omega_{\boldsymbol{\Theta}}}}
\newcommand{\wPavg}[1]{\left\{ \mkern-5mu \left\{#1 \right\} \mkern-5mu \right\}_{\omega_{\mathbf{K}}}}
\newcommand{\wUavg}[1]{\left\{ \mkern-5mu \left\{#1 \right\} \mkern-5mu \right\}_{\omega_{\mu}}}

\renewcommand{\div}[1]{\nabla \mkern-2.5mu \cdot \mkern-2.5mu {#1}}
\newcommand{\divh}[1]{\nabla_h \mkern-2.5mu \cdot \mkern-2.5mu {#1}}

\newcommand\nnfootnote[1]{%
  \begin{NoHyper}
  \renewcommand\thefootnote{}\footnote{#1}%
  \addtocounter{footnote}{-1}%
  \end{NoHyper}
}

\definecolor{myred}{rgb}{0.9, 0.0, 0.0}
\definecolor{myblue}{rgb}{0.0, 0.28, 0.67}
\definecolor{mygreen}{rgb}{0.0, 0.7, 0.0}
\definecolor{myyellow}{rgb}{1.0, 0.55, 0.0}

\begin{document}

\title{\textbf{Splitting strategies for the fully-coupled nonlinear thermo-hydro-mechanical problem}}
\author[a]{Stefano Bonetti\textsuperscript{*,}}
\author[a]{Michele Botti}
\author[a]{Paola F. Antonietti}
\affil[a]{\small{\textit{MOX, Dipartimento di Matematica, Politecnico di Milano, P.zza Leondardo da Vinci 32, 20133 Milano, Italy.}}}
\date{}

\maketitle

\begin{center}
\begin{minipage}[c]{0.8\textwidth}
We propose novel semi-decoupled and fully-decoupled iterative algorithms for efficiently solving the fully-coupled nonlinear four-field thermo-poroelastic model discretized in space by discontinuous Galerkin method on polytopal grids. We present the model problem, its four-field formulation, and the arbitrary-order weighted
symmetric interior penalty scheme exploited for its spatial discretization. Such a scheme is robust with respect to strong heterogeneities in the model coefficients. Then, we present the two solution strategies and prove that under suitable conditions both schemes are convergent. A wide set of numerical simulations is presented to assess the convergence and robustness properties of the proposed method. Moreover, we test the scheme with literature and physically sound test cases for \textit{proof-of-concept} applications in the geophysical context.
\medskip

\medskip

\end{minipage}
\end{center}

\nnfootnote{* Corresponding author}
\nnfootnote{\textit{E-mail addresses:} \texttt{stefano.bonetti@polimi.it} (Stefano Bonetti), \texttt{michele.botti@polimi.it} (Michele Botti), \texttt{paola.antonietti@polimi.it} (Paola F. Antonietti)}

\section{Introduction}
\label{sec:Introduction}
Modeling the thermo-poroelastic (TPE) coupling -- that refers to the interactions between temperature, fluid flow, and mechanical deformations -- is of crucial importance for modeling many natural and engineered systems. Some notable applications in which it turns out to be necessary to take into account the effects of the temperature to the classical poroelastic equations (cf. \cite{Biot1941, Terzaghi1943}) belong to the fields of environmental sustainability, civil engineering, and materials science; with direct impact also on public safety and economy. Relevant examples include, e.g., $CO_2$ sequestration, geothermal energy production, and seismicity/induced seismicity studies.

In the context of geophysical applications, as the ones previously mentioned, the subsoil is often modeled as a fully-saturated poroelastic material; then, the TPE problem is often formulated starting from the poroelasticity theory, which describes how the fluid flow and the elastic deformations interact within a porous medium. The temperature variation plays a key role in the description of the physical phenomena and for understanding its evolution. Thence, the temperature field is added to the physical model via an energy conservation equation that leads to a fully-coupled TPE system of equations \cite{Brun2019, Brun2020, AntoniettiBonetti2022}. The introduction of the energy conservation equation means that, within the classic equations of mass and momentum conservation, a reaction term and a contribution to the volumetric component of stress are added, respectively. Last, the energy conservation for the temperature field is similar to the mass conservation equation for the pressure, but with the presence of a nonlinear convective term, which considers the product between the temperature gradient and the filtration velocity. To ensure inf-sup stability of the problem, we consider the four-field formulation of the problem \cite{AntoniettiBonetti2022, Bonetti2024} obtained by introducing one additional scalar equation that is solved in the so-called total pressure auxiliary variable \cite{Oyarzua2016}. Moreover, this formulation gives the robustness of the scheme for quasi-incompressible media, i.e. with respect to volumetric locking. Alternative formulations of the TPE problem are discussed, e.g., in \cite{VanDuijn2019}  under different assumptions on the deformations rates and in \cite{Gatmiri1997,Lee1997} where the nonlinear convective term is neglected as it is assumed that the energy is balanced only by conduction. Depending on the different applications, in the literature the TPE problem can be found in the quasi-static regime, where the small deformations regime is assumed \cite{Brun2019, Brun2020, AntoniettiBonetti2022}, or in its fully-dynamic form \cite{Bonetti2023}, where seismic effects are the main focus. In this  work, we consider the problem in its steady formulation \cite{Bonetti2024}, that can be seen as one step of an implicit time advancing scheme (e.g. backward Euler method) applied to the quasi-static problem. We remark that, in this case, for establishing a connection between the steady problem and the quasi-static one, the physical parameters must be scaled by the time-step. 

For the spatial discretization of the problem, we propose a discontinuous Galerkin finite element method on polytopal grids (PolyDG \cite{Cangiani2014}). PolyDG schemes are appealing in this context because their geometrical flexibility facilitates local mesh refinement and coarsening and allow to easily handle highly heterogeneous media by better representing inner discontinuities. PolyDG schemes are also suited for arbitrary-order approximations and $p$-adaptivity. To further enhance robustness with respect to heterogeneities we consider a Weighted Symmetric Interior Penalty (WSIP) version of the PolyDG scheme, in which weighted averages \cite{Heinrich2002, Heinrich2003, Stenberg1998} are introduced in place of the standard averages operators of the discontinuous Galerkin methods \cite{Ern2009}. Examples of PolyDG schemes can be found in  \cite{Antonietti2013, Bassi2012, Antonietti2019} for elliptic problems,  in \cite{Cangiani2017} for parabolic problems, and in   \cite{Antonietti.Botti.ea:21} for poroelasticity. In \cite{AntoniettiBonetti2022, Bonetti2023, Bonetti2024} and in \cite{Bonetti2025} a PolyDG method for thermo-hydro-mechanics and thermo/poro-viscoelasticity is analyzed, respectively. 
In the literature, other discretization strategies for the TPE and closely related models in geophysics include, e.g., Finite Volume methods \cite{Berge2020, Stefansson2020, Schneider2017, Stefansson2021}, Hybrid Finite Volume method \cite{Droniou2024},  Hybrid High-Order \cite{BottiDiPietro2020}, Hybridizable Discontinuous Galerkin \cite{Fu2019}, lowest order Raviart-Thomas coupled with Lagrangian finite element methods \cite{Brun2020}, and eXtended finite elements \cite{Jafari2022}.
\smallskip

When solving a multiphysics differential system, it is very important to design the right strategy for its algebraic solution. In fact, these problems are often associated with large-scale issues, especially when considering realistic simulations in three-dimensional settings. Considering the monolithic resolution of the system in these contexts can lead to very large, ill-conditioned, broadband linear systems and, in three-dimensional (3D) cases, could also lead to memory problems. For this reason, a valid candidate for the numerical solution consists in a iterative coupling scheme, in which one (or more) problems are solved separately, in parallel or sequentially. These problems are solved faster, as they are smaller in size and generally have better properties (e.g. symmetric positive definite) than the monolithic system, which usually exhibits a perturbed saddle-point structure. Given the simpler structure of the different subproblems, it is also simpler to design effective (possibly robust) preconditioning techniques for them rather then for the whole system. The drawback of such a solution strategy is that an iterative procedure -- along with a suitable stopping criterion -- is introduced. However, since the TPE problem considered in this paper is nonlinear, an iterative strategy has to be introduced even when a monolithic approach is used. In previous works \cite{Brun2019, AntoniettiBonetti2022} the nonlinear transport term has been addressed with a fixed-point iterative strategy, and in \cite{Bonetti2024} a robust linearization of the advection term has been proposed. 
\smallskip

In this work, we propose and analyze two novel splitting schemes (one partially-decoupled and one fully-decoupled) where the pressure equation is always decoupled from the energy equation. In this way, at each iteration of the splitting scheme, we need to solve three linear problems, without the need to introduce an additional linearization procedure. In the context of geophysics, iterative decoupling strategies have been extensively studied in the literature for the poroelasticity problem \cite{Ahmed2020, Both2017, Castelletto2015, Iliev2016, Kim2009, Kolesov2017, Mikelic2013, Mikelic2014, Tran2005}; moveover, in \cite{Kolesov2014} a strategy for decoupling the thermoelasticity problem is investigated. For what concerns the TPE problem, in \cite{Brun2020} different splitting strategies have been studied for the nonlinear problem in its five-fields formulation, while in \cite{Cai2025, Li2025} partially-decoupled algorithms have been studied for the four-field formulation of the linear problem discretized by stable pairs of Lagrangian finite elements. The major highlights of this paper are: \textit{(a)} novel partially-decoupled and fully-decoupled solution strategies for the nonlinear TPE prolem \textit{(b)} a convergence analysis of the corresponding numerical strategies based on top of a PolyDG-WSIP discretization of the problem \textit{(c)} a numerical verification of the convergence and robustness properties of the proposed methods both in a two- and three-dimensional setting; and \textit{(d)} a verification test where we apply the schemes in a physically-sound test case (inspired by geothermal energy production procedures \cite{AntoniettiBonetti2022}) in a three-dimensional setting, showing the applicability of these schemes for real-life applications.
\smallskip

The rest of the paper is organized as follows: the model problem, the assumptions on the data, and the four-field formulation are presented in Section~\ref{sec:ModelProblem}. In Section~\ref{sec:Discretization}, we design the PolyDG-WSIP spatial discretization. In Section~\ref{sec:SolutionStrategy}, we introduce the solution strategies for the four-field TPE problem and in Section~\ref{sec:ConvMainResults} we report the main results regarding the stability of the discrete problems and the convergence of the method. Then, in Section~\ref{sec:NumericalResults} we present numerical results assessing the convergence and robustness properties of the splitting algorithms, while in Section~\ref{sec:geo3D} we present the numerical results related to the application of the proposed schemes to a physically-inspired geothermal model problem in three-dimensions. Last, in Section~\ref{sec:ConvProof}, we report the preliminaries, the auxiliary results, and the proof of the theoretical results presented in Section~\ref{sec:ConvMainResults}.

\section{Model Problem}
\label{sec:ModelProblem}
Let $\Omega \subset \mathbb{R}^d$, $d \in \{2,3\}$, be an open bounded Lipschitz polygonal/polyhedral domain. We consider the coupled thermo-hydro-mechanical problem reading: \textit{find $(\mathbf{u}, p, T)$ such that it holds} 
\begin{subequations}
	\label{eq:TPE_system}
	\begin{empheq}[left=\empheqlbrace]{align}
		& a_0 T - b_0 p + \beta \div{\mathbf{u}} - c_f \nabla T \cdot (\mathbf{K} \nabla p) - \div{(\boldsymbol{\Theta} \nabla T )} = H && \text{in} \ \Omega, \label{eq:TPE_energy_cons} \\
		& c_0 p - b_0 T + \alpha \div{\mathbf{u}} - \div{(\mathbf{K} \nabla p)} = g && \text{in} \ \Omega, \label{eq:TPE_mass_cons} \\
		& -\div{\boldsymbol{\sigma}(\mathbf{u},p,T)} = \mathbf{f} && \text{in} \ \Omega. \label{eq:TPE_momentum_cons}
	\end{empheq}
\end{subequations}
Here $T$ represents the variation of the temperature with respect to a reference value \cite{Coussy2003} and $\mathbf{u}$, $p$ represent the solid displacement and the pressure, respectively. The source terms $H$, $g$, $\mathbf{f}$ are a heat source, a fluid mass source, and a body force, respectively. 
The constitutive law for the stress tensor $\boldsymbol{\sigma}$ (in \eqref{eq:TPE_momentum_cons}) in the linear elasticity regime is given by
\begin{equation}
	\label{eq:TPE_const_law_stress}
	\boldsymbol{\sigma}(\mathbf{u},p,T) = 2 \mu \boldsymbol{\epsilon}(\mathbf{u}) + \lambda \div{\mathbf{u}} \mathbf{I} - \alpha p \mathbf{I} - \beta T \mathbf{I}, 
\end{equation}
where $\mathbf{I} \in \mathbb{R}^{d \times d}$ is the identity tensor and $\boldsymbol{\epsilon}(\mathbf{u})= \frac{1}{2}(\nabla \mathbf{u} + \nabla \mathbf{u}^T)$ is the strain tensor. For the sake of simplicity, we supplement \eqref{eq:TPE_system} 
with homogeneous Dirichlet conditions, namely $\mathbf{u} = \boldsymbol{0}$, $p = 0$, and $T=0$ on $\partial\Omega$. More general boundary can be considered up to minor modifications in the variational formulation, see for instance the ones considered in the numerical validation of Section~\ref{sec:geo3D} below.

In Table~\ref{tab:TPE_params} we detail the parameters characterizing problem \eqref{eq:TPE_system}-\eqref{eq:TPE_const_law_stress} specifying their physical interpretation and their corresponding unit of measure. For a detailed discussion on the parameters and the relations among them we refer to \cite{AntoniettiBonetti2022}.
\begin{table}[ht]
	\centering 
	\footnotesize
	\begin{tabular}{ c | c | l }
		\textbf{Notation} & \textbf{Quantity} & \textbf{Unit} \\[3pt]
		$a_0$ & thermal capacity & \si[per-mode = symbol]{\pascal \per \kelvin \squared} \\
		$b_0$ & thermal dilatation coefficient & \si{\per \kelvin} \\
		$c_0$ & specific storage coefficient & \si{\per\pascal} \\
		$\alpha$ & Biot--Willis constant & - \\
		$\beta$ & thermal stress coefficient & \si[per-mode = symbol]{\pascal \per \kelvin} \\
		$c_f$  & fluid volumetric heat capacity divided by reference temperature & \si[per-mode = symbol]{\pascal \per \kelvin\squared} \\ 
		$\mu, \lambda$ & Lamé parameters & \si{\pascal} \\
		$\tilde{\mathbf{K}}$ & permeability divided by fluid viscosity & \si[per-mode = symbol]
		{\metre\squared \per \pascal \per \second} \\
		$\tilde{\boldsymbol{\Theta}}$ & effective thermal conductivity & \si[per-mode = symbol]{\metre\squared \pascal \per \kelvin\squared \per \second} \\
		$\phi$ & porosity & - \\
	\end{tabular}
	\caption{Thermo-hydro-mechanics coefficients appearing in \eqref{eq:TPE_system}, \eqref{eq:TPE_const_law_stress}.}
	\label{tab:TPE_params}
\end{table}

\begin{remark}
	Problem~\ref{eq:TPE_system} can be seen as one step of an implicit time advancing scheme (e.g. backward Euler method) applied to the quasi-static problem \cite{AntoniettiBonetti2022}. In this case, the conductivity tensors $\mathbf{K}$ in \eqref{eq:TPE_energy_cons} and $\boldsymbol{\Theta}$ in \eqref{eq:TPE_mass_cons} are scaled by the time-step $\delta t$, namely $\mathbf{K}=\delta t \tilde{\mathbf{K}}$ and $\boldsymbol{\Theta}=\delta t \tilde{\boldsymbol{\Theta}}$, where $\tilde{\mathbf{K}}$ and $\tilde{\boldsymbol{\Theta}}$ are the actual hydraulic mobility and heat conductivity of the medium, respectively. For a detailed derivation of the quasi-static model we refer the reader to \cite{Brun2018}.    
\end{remark}

\subsection{Notation and assumptions}
For $X\subseteq\Omega$, we denote by $L^p(X)$ the standard Lebesgue space of index $p\in [1, \infty]$ and by $H^q(X)$ the Sobolev space of index $q \geq 0$ of real-valued functions defined on $X$, with the convention that $H^0(X)=L^2(X)$. 
The notation $\mathbf{L}^2(X)$ and $\mathbf{H}^q(X)$ is adopted in place of $\left[ L^2(X) \right]^d$ and $\left[ H^q(X) \right]^d$, respectively. 
In addition, we denote by $\mathbf{H}(\textrm{div},X)$ the space of $\mathbf{L}^2(X)$ vector fields whose divergence is square integrable. These spaces are equipped with natural inner products and norms denoted by $(\cdot, \cdot)_X = (\cdot, \cdot)_{L^2(X)}$ and $||\cdot||_X = ||\cdot||_{L^2(X)}$, with the convention that the subscript can be omitted in the case $X=\Omega$.
For the sake of brevity, in what follows, we make use of the symbol $x \lesssim y$ to denote $x \le C y$, where $C$ is a positive constant independent of the discretization parameters.
\smallskip

Following \cite{Brun2019}, we introduce suitable assumptions on the problem's data:
\begin{assumption}[Assumptions on the problem data]
	\hspace{0pt}
	We assume that:
	\label{assumption:TPE_model_problem}
	\begin{enumerate}
		\item the hydraulic mobility $\textup{\textbf{K}}= [K_{ij}]^{d \times d}_{i,j=1}$ and heat conductivity $\boldsymbol{\Theta}=[\Theta_{ij}]^{d \times d}_{i,j=1}$ are symmetric tensor fields which, for strictly positive real numbers $k_M>k_m$ and $\theta_M> \theta_m$, satisfy
		\begin{equation}
			k_m |\zeta|^2 \leq \zeta^T \textup{\textbf{K}}(x)\zeta \leq k_M |\zeta|^2
			\quad \text{and} \quad
			\theta_m |\zeta|^2 \leq \zeta^T \boldsymbol{\Theta}(x) \zeta \leq \theta_M |\zeta|^2 \quad\forall\zeta \in \mathbb{R}^d, \ \text{a.e.} \ x \in \Omega;
		\end{equation}
		\item The shear modulus $\mu$ and the fluid heat capacity $c_f$ are scalar fields such that $\mu:\Omega\to[\mu_m,\mu_M]$ and $c_f:\Omega\to[0, c_{fM}]$ with $0<\mu_m\le\mu_M$ and $0\le c_{fM}$;
		\item 
		The coupling parameters $\alpha:\Omega\to (\phi,1]$ and $\beta:\Omega\to (0,\beta_M]$ are strictly positive;
		\item The scalar fields $\lambda$, $c_0$, $b_0$, $a_0 \in \mathbb{R}$ are such that $\lambda\ge 0$ and $a_0, c_0 \geq b_0\geq 0$;
		\item The forcing terms are chosen such that $g,H \in L^2(\Omega)$ and $\mathbf{f} \in \mathbf{L}^2(\Omega)$.
	\end{enumerate}
\end{assumption}

\subsection{Four-field formulation}
\label{sec:TPE_4field_form}

As in \cite{AntoniettiBonetti2022}, we refer to the four-field formulation of the THM problem obtained by introducing the total pressure auxiliary variable $\varphi = \lambda\div{\mathbf{u}}-\alpha p -\beta T$ (\cite{Oyarzua2016, BottiDiPietro2020}), in which we include all the volumetric contributions to the stress tensor.

We introduce the functional spaces $\mathbf{V} = \mathbf{H}^1_0(\Omega),\ V = H^1_0(\Omega)$, and $Q = L^2(\Omega)$. Then, the weak formulation of \eqref{eq:TPE_system} reads: \textit{find $(\mathbf{u},p,T,\varphi) \in \mathbf{V} \times V \times V \times Q$ such that:}
\begin{equation}
	\label{eq:TPE_weakform_4fields}
	\begin{aligned}
		& \mathcal{M}((p, T, \varphi), (q, S, \psi)) + (\boldsymbol{\Theta} \nabla T,\nabla S) - (c_f \nabla T \cdot (\mathbf{K} \nabla p),S) + (\mathbf{K} \nabla p,\nabla q) +  (2 \mu \boldsymbol{\epsilon}(\mathbf{u}),\boldsymbol{\epsilon}(\mathbf{v})) \\
		& + (\varphi,\nabla \cdot \mathbf{v}) - (\nabla \cdot \mathbf{u},\psi) = (H,s) + (g,q) + (\mathbf{f},\mathbf{v}),
	\end{aligned}
\end{equation}
for all $(\mathbf{v},q,S,\psi) \in \mathbf{V} \times V \times V \times Q$ where the bilinear form $\mathcal{M}:V\times V \times Q\to\mathbb{R}$ is given by:
$$
\begin{aligned}
	\mathcal{M}((p, T, \varphi), (q, S, \psi)) = & \, \mathcal{M}_{p}(p,q) + \mathcal{M}_{T}(T,S) + \mathcal{M}_{\varphi}(\varphi,\psi) + \mathcal{M}_{pT}(T, q) + \mathcal{M}_{pT}(S, p) \\
	& + \mathcal{M}_{p\varphi}(q, \varphi) + \mathcal{M}_{p\varphi}(p, \psi) + \mathcal{M}_{T\varphi}(S, \varphi) + \mathcal{M}_{T\varphi}(T, \psi).
\end{aligned}
$$
with:
$$
\begin{aligned}
	&\mathcal{M}_{p}(p,q) = \left(c_{\alpha} \, p, q\right), \, &&\mathcal{M}_{T}(T,S) = \left(a_{\beta} \, T, S\right), \, && \mathcal{M}_{\varphi}(\varphi,\psi) = \frac{1}{\lambda}(\varphi,\psi), \\
	&\mathcal{M}_{pT}(S,q) = \left(b_{\alpha\beta} \, S, q\right), \, &&\mathcal{M}_{p\varphi}(q,\psi) = \left(\frac{\alpha}{\lambda}q, \psi\right), \, && \mathcal{M}_{T\varphi}(S,\psi) = \left(\frac{\beta}{\lambda}S, \psi\right).
\end{aligned}
$$
where:
$$
c_{\alpha} = c_0 + \frac{\alpha^2}{\lambda} \qquad a_{\beta} = a_0 + \frac{\beta^2}{\lambda} \qquad b_{\alpha\beta} = -b_0 + \frac{\alpha\beta}{\lambda}.
$$
\begin{remark}
	The convection term $c_f \nabla T \cdot (\mathbf{K} \nabla p)$ should not be tested by an $H^1$-regular function, since it is only in $L^1(\Omega)$. However, it can be inferred from the thermal energy equation \eqref{eq:TPE_energy_cons} and the assumption on the problem data that $c_f \nabla T \cdot (\mathbf{K} \nabla p) \in H^{-1}(\Omega)$, where $H^{-1}(\Omega)$ is the dual space of $V$. Therefore, the third term in the left-hand side of \eqref{eq:TPE_weakform_4fields} has to be intended as the duality product 
	$$
	\langle c_f \nabla T \cdot (\mathbf{K} \nabla p),S \rangle_{H^{-1}(\Omega),H^1_0(\Omega)}.
	$$
	For additional details on the well-posedness of coupled Darcy and heat equations, we refer to \cite{Bernardi.Dib:18}.
\end{remark}

\section{Discretization}
\label{sec:Discretization}
The aim of this section is to derive the PolyDG approximation of problem \eqref{eq:TPE_system}. We start by introducing some preliminaries, then, we discuss the PolyDG scheme, in which we exploit the Weighted Symmetric Interior Penalty formulation (WSIP) \cite{Ern2009}, in order to make the method able to cope with strong heterogeneities in the physical coefficients. A particular focus will be devoted to the linearization and stabilization procedures.

\subsection{Preliminaries}
\label{sec:TPE_DG_preliminaries}
The aim of this section is to introduce some instrumental assumptions and results on the PolyDG method. For designing the PolyDG discretization of Problem~\eqref{eq:TPE_system} we start by introducing a polytopal subdivision $\mathcal{T}_h$ of the computational domain $\Omega$ and its features. An interface is defined as a planar, simplicial subset of the intersection of the boundaries of any two neighboring elements of $\mathcal{T}_h$. 
In the following, we denote with $\mathcal{F}$, $\mathcal{F}_I$, and $\mathcal{F}_B$ the set of faces, interior faces, and boundary faces, respectively.
In what follows, we introduce the main assumptions on the mesh $\mathcal{T}_h$ \cite{Cangiani2014, CangianiDong:17}.
\begin{definition}[Polytopic regular mesh]
	\label{def:unif_regular}
	A mesh $\mathcal{T}_h$ is said to be polytopic regular if $\forall \kappa \in \mathcal{T}_h$, there exist a set of non-overlapping $d$-dimensional simplices contained in $\kappa$, denoted by $\{S_{\kappa}^F\}_{F \subset \partial \kappa}$, such that, for any face $F \subset \partial \kappa$ the following condition holds: $h_{\kappa} \lesssim d \ |S_{\kappa}^F| \ |F|^{-1}$.
\end{definition}

\begin{assumption}
	\label{ass:TPE_mesh_Th1}
	Given $\{\mathcal{T}_h\}_h, h>0$, we assume that the following properties are uniformly satisfied:
	\begin{enumerate}[start=1,label={\bfseries A.\arabic* }]
		\item \label{ass:TPE_A1} $\mathcal{T}_h$ is uniformly polytopic-regular;
		\item \label{ass:TPE_A3} For any neighbouring elements $\kappa^+, \kappa^- \in \mathcal{T}_h$, an hp-local bounded variation property holds, i.e. \\$h_{\kappa^+} \lesssim h_{\kappa^-} \lesssim h_{\kappa^+}$.
	\end{enumerate}
\end{assumption}
Note that the bounded variation hypothesis \ref{ass:TPE_A3} is introduced to avoid technicalities. Under \ref{ass:TPE_A1} the following inequality (\textit{trace-inverse} inequality) holds \cite{Cangiani2017}: 
\begin{equation}
	\label{eq:TPE_trace_inverse_ineq}
	\| v\|_{L^2(\partial \kappa)} \le C_{\mathrm{tr}}  
	h_{\kappa}^{-\frac12}\, \ell\, \|v\|_{L^2(\kappa)} \quad \forall v \in \mathbb{P}^{\ell}(\kappa) \quad \forall \kappa \in \mathcal{T}_h,
\end{equation}
where $\mathbb{P}^{\ell}(\kappa)$ is the space of polynomials of maximum degree equal to $\ell$ in $\kappa$, $h_{\kappa}$ is the diameter of $\kappa \in \mathcal{T}_h$, and $C_{\mathrm{tr}}>0$ is independent of $\ell, h_{\kappa}$, the number of faces per element, and the relative size of a face compared to the diamater of the element it belongs to.

For the sake of simplicity, we assume that the parameters $\boldsymbol{\Theta}, \mathbf{K}$, $\mu$, and $c_f$ are element-wise constant. Then, we can introduce the following quantities:
\begin{equation}
	\Theta_{\kappa} = \left(|\sqrt{\boldsymbol{\Theta}\rvert_{\kappa}}|_2^2 \right), \qquad K_{\kappa} = \left(|\sqrt{\mathbf{K}\rvert_{\kappa}}|_2^2 \right), \qquad \mu_{\kappa} = \mu \rvert_{\kappa},
	\qquad c_{f,\kappa} = c_f \rvert_{\kappa},
\end{equation}
where $|\cdot|_2$ is the $\ell^2$-norm in $\mathbb{R}^{d \times d}$. We remark that this assumption is reasonable in the context of groundwater flow models, where the data are derived through local measurements.

\subsection{The PolyDG-WSIP discrete problem}
\label{sec:TPE_DG_problem}
In this section, we present the WSIP method \cite{Ern2009} and discuss its application to the THM problem. The key ingredient of the method is to use weighted averages instead of arithmetic ones. The use of weighted averages has been introduced for elliptic problems in \cite{Stenberg1998} and then developed for discontinuous Galerkin methods (dG-WSIP) for dealing with advection-diffusion problems with locally vanishing diffusion \cite{Ern2009}. As one of the aims of this work is to inspect the robustness with respect to the model's coefficients, we use this modification of the standard PolyDG scheme. 

For the definition of the WSIP method we introduce the weight function $\omega^+:\mathcal{F}_I\to [0,1]$ \cite{Heinrich2002, Heinrich2003, Heinrich2005, Stenberg1998, Ern2009}. Given an interior face $F \in \mathcal{F}_I$, we denote the values taken by $\omega^+$ and $\omega^- = 1-\omega^+$ on the face $F$ as $\omega\rvert_F^{+}$ and $\omega\rvert_F^{-}$, respectively. Given the function $\omega$ we can introduce the notion of weighted averages and jump operators, denoted with $\wavg{\cdot}$ and $\jump{\cdot}$, and of normal jump, denoted by $\jump{\cdot}_n$ \cite{Arnold2002, Ern2009}:
\begin{equation}
	\label{eq:TPE_avg_jump_operators}
	\begin{aligned}
		& \jump{a} = a^+ \mathbf{n^+} + a^- \mathbf{n^-}, \quad && \jump{\mathbf{a}} = \mathbf{a}^+ \odot \mathbf{n^+} + \mathbf{a}^- \odot \mathbf{n^-}, \quad &&\jump{\mathbf{a}}_n = \mathbf{a}^+ \cdot \mathbf{n^+} + \mathbf{a}^- \cdot \mathbf{n^-}, \\ 
		& \wavg{a} = \omega^+ a^+ + \omega^- a^, \quad && \wavg{\mathbf{a}} = \omega^+ \mathbf{a}^+ + \omega^- \mathbf{a}^-, \quad && \wavg{\mathbf{A}} = \omega^+ \mathbf{A}^+ + \omega^- \mathbf{A}^-,
	\end{aligned}
\end{equation}
where $\mathbf{a} \odot \mathbf{n} = \mathbf{a}\mathbf{n}^T$, and $a, \ \mathbf{a}, \ \mathbf{A}$ are (regular enough) scalar-valued, vector-valued, and tensor-valued functions, respectively. The subscript $\omega$ in the weighted-average operator is omitted whenever $\omega^+ = \omega^- = 1/2$. On boundary faces $F\in\mathcal{F}_B$, we set
$ \jump{a} = a \mathbf{n},\ \wavg{a} = a,\ \jump{\mathbf{a}} = \mathbf{a} \odot \mathbf{n},\ \wavg{\mathbf{a}} = \mathbf{a},\ \jump{\mathbf{a}}_n = \mathbf{a} \cdot \mathbf{n},\ \wavg{\mathbf{A}} = \mathbf{A}.$ For the averages, this corresponds to consider $\omega^\pm$ single-valued and equal to $1$.

We start deriving the PolyDG-WSIP formulation of problem \eqref{eq:TPE_weakform_4fields} by intoducing the discrete spaces that are used in the following. Given $m,\ell \geq 1$, we define:
\begin{equation}
	\begin{aligned}
		Q_h^{m} &= \left\{ \psi \in L^2(\Omega) : \psi |_{\kappa} \in \mathbb{P}^{m}(\kappa) \  \forall \kappa \in \mathcal{T}_h \right\}\hspace{-0.5mm}, \;
		V_h^{\ell} = \left\{ v \in L^2(\Omega) : v |_{\kappa} \in \mathbb{P}^{\ell}(\kappa) \  \forall \kappa \in \mathcal{T}_h \right\}\hspace{-0.5mm}, \;
		\mathbf{V}_h^{\ell} = \left[V_h^{\ell}\right]^d\hspace{-1mm}.
	\end{aligned}
\end{equation}
The PolyDG-WSIP discretization of problem \eqref{eq:TPE_weakform_4fields} reads:
\textit{find $(\mathbf{u}_h,p_h,T_h,\varphi_h) \in \mathbf{V}_h^{\ell} \times V_h^{\ell} \times V_h^{\ell} \times Q_h^m$ such that for all $(\mathbf{v}_h,q_h,S_h,\varphi_h) \in \mathbf{V}_h^{\ell} \times V_h^{\ell} \times V_h^{\ell} \times Q_h^m$ it holds}
\begin{equation}
	\label{eq:TPE_discrete_4fields}
	\begin{aligned}
		& \mathcal{M}((p_h, T_h, \varphi_h),(q_h,S_h,\psi_h)) + \mathcal{A}_h^{T}(T_h,S_h) + \mathcal{C}_h(T_h,p_h,S_h)
		+ \mathcal{A}_h^{p}(p_h,q_h) + \mathcal{A}_h^{e}(\mathbf{u}_h,\mathbf{v}_h) \\
		&\; - \mathcal{B}_h(\varphi_h,\mathbf{v}_h) +  \mathcal{B}_h(\psi_h, \mathbf{u}_h) +  \mathcal{D}_h(\varphi_h,\psi_h)
		= \left((\mathbf{f}, g, H),(\mathbf{v}_h, q_h, S_h)\right),
	\end{aligned}
\end{equation}
where the bilinear and trilinear forms are defined by:
\begin{equation}
	\label{eq:TPE_bilinear_forms_discr}
	\begin{aligned}
		& \mathcal{A}_h^T(T,S) = \left(\boldsymbol{\Theta}\nabla_h T, \nabla_h S\right) - \sum_{F \in \mathcal{F}} \int_F \left(\wTavg{\boldsymbol{\Theta}\nabla_h T} \mkern-2.5mu \cdot \mkern-2.5mu \jump{S} + \jump{T} \mkern-2.5mu \cdot \mkern-2.5mu \wTavg{\boldsymbol{\Theta}\nabla_h S} - \sigma \jump{T} \mkern-2.5mu \cdot \mkern-2.5mu \jump{S}\right),\\
		& \mathcal{A}_h^p(p,q) = (\mathbf{K} \nabla_h p,\nabla_h q) - \sum_{F \in \mathcal{F}} \int_F \left(\wPavg{\mathbf{K} \nabla_h p} \mkern-2.5mu \cdot \mkern-2.5mu \jump{q} + \jump{p} \mkern-2.5mu \cdot \mkern-2.5mu \wPavg{\mathbf{K} \nabla_h q} - \xi \jump{p} \mkern-2.5mu \cdot \mkern-2.5mu \jump{q}\right),\\
		& \mathcal{A}_h^e(\mathbf{u},\mathbf{v}) = (2 \mu\boldsymbol{\epsilon}_h(\mathbf{u}),\boldsymbol{\epsilon}_h(\mathbf{v}))
		- \sum_{F \in \mathcal{F}} \int_F \left( \wUavg{2 \mu\boldsymbol{\epsilon}_h(\mathbf{u})} \mkern-2.5mu : \mkern-2.5mu \jump{\mathbf{v}} + \jump{\mathbf{u}} \mkern-2.5mu : \mkern-2.5mu \wUavg{2 \mu\boldsymbol{\epsilon}_h(\mathbf{v})} -  \zeta \jump{\mathbf{u}} \mkern-2.5mu : \mkern-2.5mu \jump{\mathbf{v}}\right),\\
		& \mathcal{B}_h(\varphi,\mathbf{v}) = - (\varphi,\nabla_h \mkern-2.5mu \cdot \mkern-2.5mu \mathbf{v}) + \sum_{F \in \mathcal{F}} \int_F \avg{ \varphi} \mkern-2.5mu \cdot \mkern-2.5mu \jump{\mathbf{v}}_n,\\
		& \begin{aligned}
			\mathcal{C}_h (T,p,S) = \ & \big( -c_f \left(\mathbf{K} \ \nabla_h p\right) \cdot \nabla_h T, S \big) - \sum_{F \in \mathcal{F}_I} \int_{F} \left( \avg{ - c_f \ \mathbf{K} \nabla_h p } \cdot \jump{T} \right) \avg{S} \\
			& + \frac{1}{2}\sum_{F \in \mathcal{F}} \int_F \ \left| \avg{ - c_f \ \mathbf{K} \nabla_h p} \cdot \mathbf{n} \right| \jump{T} \mkern-2.5mu \cdot \mkern-2.5mu \jump{S} - \frac{1}{2}\sum_{F \in \mathcal{F}_B} \int_F \  (- c_f \ \mathbf{K} \nabla_h p) \cdot \mathbf{n} \, T \, S \end{aligned} \\
		&\mathcal{D}_h(\varphi,\psi) = \sum_{F \in \mathcal{F}_I} \int_F \varrho \jump{\varphi} \mkern-2.5mu \cdot \mkern-2.5mu \jump{\psi}.
	\end{aligned}
\end{equation}
For all $w\in V_h^{\ell}$ and $\mathbf{w}\in \mathbf{V}_h^{\ell}$, $\nabla_h w$ and $\divh{\mathbf{w}}$ denote the broken differential operators whose restrictions to each element $\kappa \in \mathcal{T}_h$ are defined as $\nabla w_{|\kappa}$ and $\div{\mathbf{w}}_{|\kappa}$, respectively, and $\boldsymbol{\epsilon}_h(\mathbf{u}) = \left(\nabla_h \mathbf{u} + \nabla_h \mathbf{u}^T\right)/2$. In \eqref{eq:TPE_bilinear_forms_discr} we set:
\begin{equation}
	\omega_{\boldsymbol{\Theta}}^{\pm} = \frac{\delta_{\boldsymbol{\Theta}_n}^{\mp}}{\delta_{\boldsymbol{\Theta}_n}^{+} + \delta_{\boldsymbol{\Theta}_n}^{-}}, \qquad \omega_{\mathbf{K}}^{\pm} = \frac{\delta_{\mathbf{K}_n}^{\mp}}{\delta_{\mathbf{K}_n}^{+} + \delta_{\mathbf{K}_n}^{-}}, \qquad 
	\omega_{\mu}^{\pm} = \frac{\mu^{\mp}}{\mu^{+} + \mu^{-}},
\end{equation}
where $\delta_{\boldsymbol{\Theta}_n}^{\pm} = \mathbf{n}^{{\pm}^T} \, \boldsymbol{\Theta}^{\pm} \, \mathbf{n}^{{\pm}}$, $\delta_{\mathbf{K}_n}^{\pm} = \mathbf{n}^{{\pm}^T} \, \mathbf{K}^{\pm} \, \mathbf{n}^{{\pm}}$. Note that, the PolyDG-WSIP method requires also a different definition of penalty coefficients with respect to standard IP method \cite{Arnold1982, Wheeler1978, Ern2021, Cangiani2017}. Thus, the stabilization functions $\sigma, \xi, \zeta, \varrho \in L^{\infty}(\mathcal{F}_h)$ appearing in \eqref{eq:TPE_bilinear_forms_discr} are defined according to \cite{Ern2009} as:
\begin{equation}
	\label{eq:TPE_stabilization_func}
	\begin{aligned}
		\sigma &= \left\{\begin{aligned}
			&\alpha_1 \gamma_{\boldsymbol{\Theta}} \underset{\kappa \in \{\kappa^+,\kappa^-\}}{\mbox{max}} \left(\frac{\ell^2}{h_{\kappa}}\right) \quad & F \in \mathcal{F}_I,\\
			&\alpha_1 \overline{\Theta}_{\kappa} \frac{\ell^2}{h_{\kappa}} \quad & F \in \mathcal{F}_B,\\
		\end{aligned}
		\right.
		\qquad
		\xi &&= \left\{\begin{aligned}
			&\alpha_2 \gamma_{\mathbf{K}} \underset{\kappa \in \{\kappa^+,\kappa^-\}}{\mbox{max}}\left(\frac{\ell^2}{h_{\kappa}}\right) \quad & F \in \mathcal{F}_I,\\
			&\alpha_2 \overline{K}_{\kappa} \frac{\ell^2}{h_{\kappa}} \quad & F \in \mathcal{F}_B,\\
		\end{aligned}
		\right.\\
		\zeta &= \left\{\begin{aligned}
			&\alpha_3 \gamma_{\mu} \underset{\kappa \in \{\kappa^+,\kappa^-\}}{\mbox{max}}\left(\frac{\ell^2}{h_{\kappa}}\right) \quad & F \in \mathcal{F}_I,\\
			&\alpha_3 \mu_{\kappa} \frac{\ell^2}{h_{\kappa}} \quad & F \in \mathcal{F}_B,\\
		\end{aligned}
		\right.
		\qquad
		\varrho &&= \left\{\begin{aligned}
			&\alpha_4 \underset{\kappa \in \{\kappa^+,\kappa^-\}}{\mbox{min}}\left(\frac{h_{\kappa}}{m}\right) \quad & F \in \mathcal{F}_I,\\
			&\alpha_4 \frac{h_{\kappa}}{m} \quad & F \in \mathcal{F}_B,\\
		\end{aligned}
		\right.\\
	\end{aligned}   
\end{equation}
where $\alpha_1, \alpha_2, \alpha_3, \alpha_4 \in \mathbb{R}$ are positive constants to be properly defined, $h_{\kappa}$ is the diameter of the element $\kappa \in \mathcal{T}_h$, and the coefficients $\gamma_{\boldsymbol{\Theta}}$, $\gamma_{\mathbf{K}}$, $\gamma_{\mu}$ are given by:
\begin{equation}
	\gamma_{\boldsymbol{\Theta}} = \frac{ \delta_{\boldsymbol{\Theta}_n}^{+} \, \delta_{\boldsymbol{\Theta}_n}^{-}}{\delta_{\boldsymbol{\Theta}_n}^{+} + \delta_{\boldsymbol{\Theta}_n}^{-}}, \qquad \gamma_{\mathbf{K}} = \frac{\delta_{\mathbf{K}_n}^{+} \, \delta_{\mathbf{K}_n}^{-}}{\delta_{\mathbf{K}_n}^{+} + \delta_{\mathbf{K}_n}^{-}}, \qquad \gamma_{\mu}^{\pm} = \frac{\mu^{+} \, \mu^{-}}{\mu^{+} + \mu^{-}}.
\end{equation}
We point out that in the discrete formulation above, we have decided to consider the same polynomial degree for the spaces $V_h^{\ell}$ and $\mathbf{V}_h^{\ell}$, because we are mainly interested in approximation schemes yielding the same accuracy for all the primary variables.
\begin{remark}
	Note that, in \eqref{eq:TPE_discrete_4fields}, we have added an additional weakly consistent stabilization term $\mathcal{D}_h$ for the total pressure following the dG discretization of the Stokes problem analyzed in \cite{antonietti2020_stokesDG}.
\end{remark}

\section{Solution strategies}
\label{sec:SolutionStrategy}
The aim of this section is to present three different solution strategies we can exploit for solving problem \eqref{eq:TPE_discrete_4fields}. We first present the fixed-point strategy described in \cite{Bonetti2024}, then we extend two splitting strategies presented in \cite{Brun2020, Cai2025} to the PolyDG approximation of the four-fields formulation of the nonlinear TPE problem. All these algorithms are based on an iterative fashion, where the solution is computed using the sequence $(\mathbf{u}_h^k, p_h^k, T_h^k, \varphi_h^k)$ for $k \geq 0$, where the
iterate $(\mathbf{u}_h^0, p_h^0, T_h^0, \varphi_h^0)$  is a (possibly educated) initial guess.
\smallskip

The proposed algorithms rely on the decoupling of the subproblems involved in the physical phenomena. In the following, we use the letters \textbf{F}, \textbf{H}, and \textbf{M} to denote the flow \eqref{eq:TPE_mass_cons}, heat \eqref{eq:TPE_energy_cons}, and mechanics \eqref{eq:TPE_momentum_cons} problems, respectively. We remark that, as a first approach, in the mechanics problem we consider both the displacement $\mathbf{u}$ and the total-pressure $\varphi$ as unknowns. Then, the subproblems in the continuous form read:
$$
\begin{aligned}
	& \mathbf{F} && c_0 p - \div{(\mathbf{K} \nabla p)} = g + b_0 T - \alpha \div{\mathbf{u}} \\
	& \mathbf{H} && a_0 T - c_f \nabla T \cdot (\mathbf{K} \nabla p) - \div{(\boldsymbol{\Theta} \nabla T )} = H + b_0 p - \beta \div{\mathbf{u}} \\
	& \mathbf{M} && \left\{ \begin{aligned}
		&-\div{\left(2 \mu \boldsymbol{\epsilon}(\mathbf{u}) + \varphi \mathbf{I}\right)} = \mathbf{f} \\
		&\varphi - \lambda \div{\mathbf{u}} = - \alpha p - \beta T 
	\end{aligned}
	\right.
\end{aligned}
$$

We start by presenting the fully-coupled scheme \textbf{FHM} -- presented in Algorithm~\ref{alg:FHM} -- which only results in a linearization procedure, coupled with a fixed-point iteration scheme, in which we approximate the nonlinear convective term as $\mathcal{C}_h(T_h^k,p_h^{k-1},S_h)$. The main challenge presented by this algorithm is that is does require the solution of a large linear system generated by \textbf{FHM} at each iteration $k \geq 0$. We remark that, the \textbf{FHM} scheme is the solution strategy presented in \cite{Bonetti2024}.

\begin{algorithm}[!ht]
	\small
	\caption{FHM: the monolithic scheme}\label{alg:FHM}
	\begin{algorithmic}
		
		\State \textbf{Input:} 
		\begin{itemize}
			\item $\mathbf{X}_h^{0} = (\mathbf{u}_h^0, p_h^0, T_h^0, \varphi_h^0) \rightarrow$ initial condition, \Comment{zero solution or solution to the linear problem}
			\item $\mathrm{toll}_{\text{abs}}, \, \mathrm{toll}_{\text{rel}} \rightarrow$ tolerance for the absolute and relative errors (between successive iterations),
			\item $\mathrm{nmax} \rightarrow$ maximum number of iterations.
		\end{itemize}
		
		\State \textbf{Output:} 
		\begin{itemize}
			\item $\mathbf{X}_h = (\mathbf{u}_h, p_h, T_h, \varphi_h) \rightarrow$ solution,
			\item $k \rightarrow$ iterations count for convergence
		\end{itemize}
		
		\State Initialize $k \gets 0$, $E_{\text{abs}} \gets \mathrm{toll}_{\text{abs}} + 1$, and $E_{\text{rel}} \gets \mathrm{toll}_{\text{rel}} + 1$
		\vspace{0.1cm}
		\While{$E_{\text{abs}} > \mathrm{toll}_{\text{abs}} \,\,\,\,\text{and}\,\,\,\, E_{\text{rel}} > \mathrm{toll}_{\text{rel}} \,\,\,\,\text{and}\,\,\,\, k < \mathrm{nmax}$}
		\vspace{0.1cm}
		\State given $p_h^k$, find $\mathbf{X}_h^{k+1} \in \mathbf{V}_h^{\ell} \times V_h^{\ell} \times V_h^{\ell} \times Q_h^m$ such that
		$$
		\begin{aligned}
			& \mathcal{M}((p_h^{k+1}, T_h^{k+1}, \varphi_h^{k+1}),(q_h,S_h,\psi_h)) + \mathcal{A}_h^{T}(T_h^{k+1},S_h) + \mathcal{C}_h(T_h^{k+1},p_h^{k},S_h) + \mathcal{A}_h^{p}(p_h^{k+1},q_h) \\
			& + \mathcal{A}_h^{e}(\mathbf{u}_h^{k+1},\mathbf{v}_h) - \mathcal{B}_h(\varphi_h^{k+1},\mathbf{v}_h) +  \mathcal{B}_h(\psi_h, \mathbf{u}_h^{k+1}) + \mathcal{D}_h(\varphi_h^{k+1},\psi_h) = \left((\mathbf{f}, g, H),(\mathbf{v}_h, q_h, S_h)\right)
		\end{aligned}
		$$
		\State compute $E_{\text{abs}} \gets \|\mathbf{u}_h^{k+1} - \mathbf{u}_h^{k}\| + \|p_h^{k+1} - p_h^{k}\| +  \|T_h^{k+1} - T_h^{k}\| +  \|\varphi_h^{k+1} - \varphi_h^{k}\|$
		\vspace{0.1cm}
		\If {$\|\mathbf{u}_h^{k}\| \neq 0 \,\,\,\,\text{and}\,\,\,\, \|p_h^{k}\| \neq 0 \,\,\,\,\text{and}\,\,\,\, \|T_h^{k}\| \neq 0 \,\,\,\,\text{and}\,\,\,\, \|\varphi_h^{k}\| \neq 0$}
		\State compute $E_{\text{rel}} \gets \frac{\|\mathbf{u}_h^{k+1} - \mathbf{u}_h^{k}\|}{\|\mathbf{u}_h^{k}\|} + \frac{\|p_h^{k+1} - p_h^{k}\|}{\|p_h^{k}\|} + \frac{\|T_h^{k+1} - T_h^{k}\|}{\|T_h^{k}\|} + \frac{\|\varphi_h^{k+1} - \varphi_h^{k}\|}{\|\varphi_h^{k}\|}$
		\Else 
		\State assign $E_{\text{rel}} \gets \mathrm{toll}_{\text{rel}} + 1$ 
		\vspace{0.1cm}
		\EndIf
		\vspace{0.1cm}
		\State update $\mathbf{X}_h^{k} \gets \mathbf{X}_h^{k+1}$, $k \gets k + 1$
		\vspace{0.1cm}
		\EndWhile
		\vspace{0.1cm}
		\State Set $\mathbf{X}_h \gets \mathbf{X}_h^k$
	\end{algorithmic}
\end{algorithm}
\smallskip

We consider splitting schemes in which we either decouple all the subproblems and solve each separately at every iteration (\textit{three-step} algorithm) or decouple only one subproblem from the other two which are then solved together (\textit{two-step} algorithm). Following this notation, we refer to the fixed-point strategy in which we solve the linearized problem monolithically at every iteration as a \textit{one-step} algorithm. For instance, a \textit{two-step} algorithm where the flow and mechanics subproblems are solved together decoupled from the heat subproblem is referred to as \textbf{FM-H} and similarly for other combinations of coupling/decoupling of the subproblems.

The two splitting schemes we propose are the \textbf{FM-H} and the \textbf{F-H-M} schemes and are presented in Algorithm~\ref{alg:FM-H} and Algorithm~\ref{alg:F-H-M}, respectively. We remark that, in these schemes the \textbf{F} and \textbf{H} problems are always decoupled. Following this strategy, we observe that the problem in which the temperature is involved is always a linear problem, then we do not need to consider two nested iterative procedures due to the splitting iterations together with the linearization algorithm to deal with the convective nonlinear term. 

\begin{algorithm}[ht!]
	\small
	\caption{FM-H: coupled flow and mechanics}
	\label{alg:FM-H}
	\begin{algorithmic}
		\State \textbf{Input:} 
		\begin{itemize}
			\item $\mathbf{X}_h^{0} = (\mathbf{u}_h^0, p_h^0, T_h^0, \varphi_h^0) \rightarrow$ initial condition, \Comment{zero solution or solution to the linear problem}
			\item $\mathrm{toll}_{\text{abs}}, \, \mathrm{toll}_{\text{rel}} \rightarrow$ tolerance for the absolute and relative errors (between successive iterations),
			\item $\mathrm{nmax} \rightarrow$ maximum number of iterations.
		\end{itemize}
		
		\State \textbf{Output:} 
		\begin{itemize}
			\item $\mathbf{X}_h = (\mathbf{u}_h, p_h, T_h, \varphi_h) \rightarrow$ solution,
			\item $k \rightarrow$ iterations count for convergence
		\end{itemize}
		
		\State Initialize $k \gets 0$, $E_{\text{abs}} \gets \mathrm{toll}_{\text{abs}} + 1$, and $E_{\text{rel}} \gets \mathrm{toll}_{\text{rel}} + 1$
		\vspace{0.1cm}
		\While{$E_{\text{abs}} > \mathrm{toll}_{\text{abs}} \,\,\,\,\text{and}\,\,\,\, E_{\text{rel}} > \mathrm{toll}_{\text{rel}} \,\,\,\,\text{and}\,\,\,\, k < \mathrm{nmax}$}
		\vspace{0.1cm}
		\Statex \textbf{Step 1:} given $T_h^k$, find $(\mathbf{u}_h^{k+1}, p_h^{k+1}, \varphi_h^{k+1}) \in \mathbf{V}_h^{\ell} \times V_h^{\ell} \times Q_h^m$ such that
		$$
		\begin{aligned}
			& \mathcal{M}_{p}(p_h^{k+1},q_h) + \mathcal{M}_{\varphi}(\varphi_h^{k+1},\psi_h) + \mathcal{M}_{p\varphi}(p_h^{k+1}, \psi_h) + \mathcal{M}_{p\varphi}(q_h, \varphi_h^{k+1}) \\
			& + \mathcal{A}_h^{p}(p_h^{k+1},q_h) + \mathcal{A}_h^{e}(\mathbf{u}_h^{k+1},\mathbf{v}_h) - \mathcal{B}_h(\varphi_h^{k+1},\mathbf{v}_h) +  \mathcal{B}_h(\psi_h, \mathbf{u}_h^{k+1}) \\
			& +  \mathcal{D}_h(\varphi_h^{k+1},\psi_h) = \left((\mathbf{f}, g),(\mathbf{v}_h, q_h)\right) - \mathcal{M}_{pT}(T_h^k, q_h) - \mathcal{M}_{T\varphi}(T_h^{k}, \psi) 
		\end{aligned}
		$$
		\State \textbf{Step 2:} given $(\mathbf{u}_h^{k+1}, p_h^{k+1}, \varphi_h^{k+1})$, find $T_h^{k+1} \in V_h^{\ell}$ such that
		$$
		\begin{aligned}
			& \mathcal{M}_{T}(T_h^{k+1}, S_h) + \mathcal{A}_h^{T}(T_h^{k+1},S_h) + \mathcal{C}_h(T_h^{k+1},p_h^{k+1},S_h) \\
			& = (H, S_h) - \mathcal{M}_{pT}(S_h, p_h^{k+1}) - \mathcal{M}_{T\varphi}(S_h, \varphi_h^{k+1}).
		\end{aligned}
		$$
		\State compute $E_{\text{abs}} \gets \|\mathbf{u}_h^{k+1} - \mathbf{u}_h^{k}\| + \|p_h^{k+1} - p_h^{k}\| +  \|T_h^{k+1} - T_h^{k}\| +  \|\varphi_h^{k+1} - \varphi_h^{k}\|$
		\vspace{0.1cm}
		\If {$\|\mathbf{u}_h^{k}\| \neq 0 \,\,\,\,\text{and}\,\,\,\, \|p_h^{k}\| \neq 0 \,\,\,\,\text{and}\,\,\,\, \|T_h^{k}\| \neq 0 \,\,\,\,\text{and}\,\,\,\, \|\varphi_h^{k}\| \neq 0$}
		\State compute $E_{\text{rel}} \gets \frac{\|\mathbf{u}_h^{k+1} - \mathbf{u}_h^{k}\|}{\|\mathbf{u}_h^{k}\|} + \frac{\|p_h^{k+1} - p_h^{k}\|}{\|p_h^{k}\|} + \frac{\|T_h^{k+1} - T_h^{k}\|}{\|T_h^{k}\|} + \frac{\|\varphi_h^{k+1} - \varphi_h^{k}\|}{\|\varphi_h^{k}\|}$
		\Else 
		\State assign $E_{\text{rel}} \gets \mathrm{toll}_{\text{rel}} + 1$ 
		\vspace{0.1cm}
		\EndIf
		\vspace{0.1cm}
		\State update $\mathbf{X}_h^{k} \gets \mathbf{X}_h^{k+1}$, $k \gets k + 1$
		\vspace{0.1cm}
		\EndWhile
		\vspace{0.1cm}
		\State Set $\mathbf{X}_h \gets \mathbf{X}_h^k$
	\end{algorithmic}
\end{algorithm}

\begin{algorithm}[ht!]
	\small
	\caption{F-H-M: decoupled flow, heat, and mechanics}\label{alg:F-H-M}
	\begin{algorithmic}
		
		\State \textbf{Input:} 
		\begin{itemize}
			\item $\mathbf{X}_h^{0} = (\mathbf{u}_h^0, p_h^0, T_h^0, \varphi_h^0) \rightarrow$ initial condition, \Comment{zero solution or solution to the linear problem}
			\item $\mathrm{toll}_{\text{abs}}, \, \mathrm{toll}_{\text{rel}} \rightarrow$ tolerance for the absolute and relative errors (between successive iterations),
			\item $\mathrm{nmax} \rightarrow$ maximum number of iterations.
		\end{itemize}
		
		\State \textbf{Output:} 
		\begin{itemize}
			\item $\mathbf{X}_h = (\mathbf{u}_h, p_h, T_h, \varphi_h) \rightarrow$ solution,
			\item $k \rightarrow$ iterations count for convergence
		\end{itemize}
		
		\State Initialize $k \gets 0$, $E_{\text{abs}} \gets \mathrm{toll}_{\text{abs}} + 1$, and $E_{\text{rel}} \gets \mathrm{toll}_{\text{rel}} + 1$
		\vspace{0.1cm}
		\While{$E_{\text{abs}} > \mathrm{toll}_{\text{abs}} \,\,\,\,\text{and}\,\,\,\, E_{\text{rel}} > \mathrm{toll}_{\text{rel}} \,\,\,\,\text{and}\,\,\,\, k < \mathrm{nmax}$}
		\vspace{0.1cm}
		\State \textbf{Step 1:} given $(\mathbf{u}_h^{k}, T_h^{k}, \varphi_h^{k})$, find $p_h^{k+1} \in V_h^{\ell}$ such that
		$$ 
		\label{eq:F_H_M_splitting_1}
		\begin{aligned}
			& \mathcal{M}_{p}(p_h^{k+1},q_h) + \mathcal{A}_h^{p}(p_h^{k+1},q_h) = (g,q_h) - \mathcal{M}_{pT}(T_h^k, q_h) - \mathcal{M}_{p\varphi}(q_h, \varphi_h^{k})  
		\end{aligned}
		$$
		\State \textbf{Step 2:} given $(\mathbf{u}_h^{k}, \varphi_h^{k})$ and $p_h^{k+1}$, find $T_h^{k+1} \in V_h^{\ell}$ such that
		$$
		\label{eq:F_H_M_splitting_2}
		\begin{aligned}
			& \mathcal{M}_{T}(T_h^{k+1}, S_h) + \mathcal{A}_h^{T}(T_h^{k+1},S_h) + \mathcal{C}_h(T_h^{k+1},p_h^{k+1},S_h) \\
			& = (H, S_h) - \mathcal{M}_{pT}(S_h, p_h^{k+1}) - \mathcal{M}_{T\varphi}(S_h, \varphi_h^{k})
		\end{aligned}
		$$
		\State \textbf{Step 3:} given $(p_h^{k+1}, T_h^{k+1})$, find $(\mathbf{u}_h^{k+1}, T_h^{k+1}) \in \mathbf{V}_h^{\ell} \times V_h^{\ell}$ such that
		$$
		\label{eq:F_H_M_splitting_3}
		\begin{aligned}
			& \mathcal{M}_{T}(T_h^{k+1}, S_h) + \mathcal{A}_h^{T}(T_h^{k+1},S_h) + \mathcal{C}_h(T_h^{k+1},p_h^{k+1},S_h) \\
			& = (H, S_h) - \mathcal{M}_{pT}(S_h, p_h^{k+1}) - \mathcal{M}_{T\varphi}(S_h, \varphi_h^{k+1}).
		\end{aligned}
		$$
		\State compute $E_{\text{abs}} \gets \|\mathbf{u}_h^{k+1} - \mathbf{u}_h^{k}\| + \|p_h^{k+1} - p_h^{k}\| +  \|T_h^{k+1} - T_h^{k}\| +  \|\varphi_h^{k+1} - \varphi_h^{k}\|$
		\vspace{0.1cm}
		\If {$\|\mathbf{u}_h^{k}\| \neq 0 \,\,\,\,\text{and}\,\,\,\, \|p_h^{k}\| \neq 0 \,\,\,\,\text{and}\,\,\,\, \|T_h^{k}\| \neq 0 \,\,\,\,\text{and}\,\,\,\, \|\varphi_h^{k}\| \neq 0$}
		\State compute $E_{\text{rel}} \gets \frac{\|\mathbf{u}_h^{k+1} - \mathbf{u}_h^{k}\|}{\|\mathbf{u}_h^{k}\|} + \frac{\|p_h^{k+1} - p_h^{k}\|}{\|p_h^{k}\|} + \frac{\|T_h^{k+1} - T_h^{k}\|}{\|T_h^{k}\|} + \frac{\|\varphi_h^{k+1} - \varphi_h^{k}\|}{\|\varphi_h^{k}\|}$
		\Else 
		\State assign $E_{\text{rel}} \gets \mathrm{toll}_{\text{rel}} + 1$ 
		\vspace{0.1cm}
		\EndIf
		\vspace{0.1cm}
		\State update $\mathbf{X}_h^{k} \gets \mathbf{X}_h^{k+1}$, $k \gets k + 1$
		\vspace{0.1cm}
		\EndWhile
		\vspace{0.1cm}
		\State Set $\mathbf{X}_h \gets \mathbf{X}_h^k$
	\end{algorithmic}
\end{algorithm}

For Algorithms~\ref{alg:FHM}-~\ref{alg:F-H-M}, we employ the following stopping criterion: 
$$
\sum_{x \in \{ \mathbf{u}_h, p_h, T_h, \varphi_h\}} \hspace{-0.6cm} \| x^{k+1} - x^{k} \| \leq \text{toll}_a \quad \text{or} \quad \sum_{x \in \{ \mathbf{u}_h, p_h, T_h, \varphi_h\}} \hspace{-0.6cm} \frac{\| x^{k+1} - x^{k} \|}{\|x^k\|} \leq \text{toll}_r,
$$
where $\text{toll}_a$, $\text{toll}_r$ are suitable positive tolerances that can vary according to the employed solution strategy.

\section{Convergence analysis: F-H-M algorithm}
\label{sec:ConvMainResults}
The aim of this section is to present the main theoretical results related to the \textbf{F-H-M} splitting solution strategy. We provide a discrete stability estimate and establish the convergence theorem of the iterative procedure. The auxiliary lemmas and the proof of Theorem~\ref{thm:stability}, Theorem~\ref{thm:convergence} are postponed to Section~\ref{sec:ConvProof} for the sake of presentation. First, we introduce the $dG$-norms that appear in the results:
\begin{equation}
	\label{eq:RobQSTPE_DG_norms}
	\begin{aligned}
		&\|S\|^2_{dG,T} = \|\sqrt{\boldsymbol{\Theta}} \ \nabla_h S\|^2 +\sum_{F\in\mathcal{F}} \|\sigma^{1/2} \jump{S} \|_F^2 \quad &&  \forall \ S \in V_h^{\ell},\\ 
		&\|q\|^2_{dG,p} = \|\sqrt{\mathbf{K}} \ \nabla_h q\|^2 + \sum_{F\in\mathcal{F}} \|\xi^{1/2} \jump{q} \|_F^2 \quad && \forall \ q \in V_h^{\ell},\\ 
		&\|\mathbf{v}\|^2_{dG,e} = \|\sqrt{2 \mu} \ \boldsymbol{\epsilon}_h(\mathbf{v})\|^2 + \sum_{F\in\mathcal{F}}\|\zeta^{1/2} \jump{\mathbf{v}} \|_F^2 \quad && \forall \ \mathbf{v} \in \mathbf{V}_h^{\ell}.
	\end{aligned}
\end{equation}

\begin{theorem}[Stability of the discrete problem]
	\label{thm:stability}
	Let the assumptions of the Lemmata~\ref{lem:boundcoerc_bil_forms}, ~\ref{lem:gen_inf_sup}, ~\ref{lem:bound_Ch}, ~\ref{lem:aux_stab_bound} below be satisfied and let the transport velocity $\boldsymbol{\eta} = -c_f \mathbf{K} \nabla_h p_h \in \mathbf{V}_h^{\ell}$ be such that
	$$
	|\boldsymbol{\eta}|_{dG,\infty} \lesssim a_1 < a_0,
	$$
	with $a_1 > 0$ defined in Lemma~\ref{lem:aux_stab_bound} and hidden constant independent of $\mathbf{K}$, $\boldsymbol{\Theta}$. Moreover, let us define the solution at step $k$ as $\mathbf{X}_h^{k} = (\mathbf{u}_h^k, p_h^k, T_h^k, \varphi_h^k) \in \mathbf{V}_h^{\ell} \times V_h^{\ell} \times V_h^{\ell} \times Q_h^{m}$. Then, the solution $\mathbf{X}_h^{k+1} \in \mathbf{V}_h^{\ell} \times V_h^{\ell} \times V_h^{\ell} \times Q_h^{m}$ to problem \eqref{eq:TPE_discrete_4fields} satisfies the a-priori bound
	$$
	(a_0 - a_1)\|T_h\|^2 + c_0 \|p_h\|^2 + \mathbb{B}\|\varphi_h\|^2 + \frac{1}{2}\|\mathbf{u}_h\|_{dG,e}^2 + \|p_h\|_{dG,p}^2 + \|T_h\|_{dG,T}^2 \lesssim \mathcal{RHS}(\mathbf{f}, g, H, \mathbf{X}_h^k)
	$$
	where $\mathcal{RHS}(\cdot)$ is a suitable positive function and the hidden constant is independent of the conductivity tensors $\mathbf{K}$, $\boldsymbol{\Theta}$ and the mesh size $h$.
\end{theorem}
\begin{theorem}[Convergence of the iterative scheme]
	\label{thm:convergence}
	Let the assumptions of Theorem \ref{thm:stability} hold. Additionally, assume that for all $k\ge 1$ 
	\begin{equation}
		\label{eq:convergece_ass}
		\frac{\beta^2}{\lambda^2} \gtrsim \frac{a_1}{4}-2, \quad b_1 \lesssim 1, \quad c_1 \geq \frac{b_{\alpha\beta}^2}{2}+\frac{\alpha^2}{\lambda^2}, \quad 
		\lvert T_h^k\rvert_{dG,\infty} \lesssim  c_{fM}^{-1} \sqrt{\frac{a_1\left(c_1 - \frac{b_{\alpha\beta}^2}{2}-\frac{\alpha^2}{\lambda^2}\right)}{2(1 + C_{tr}^4)}},
	\end{equation}
	with $a_1,\, b_1,\, c_1 >0$ defined in Lemma~\ref{lem:aux_stab_bound} and hidden constant independent of the model parameters. Then, the \textbf{F-H-M} splitting strategy defined in Section~\ref{sec:SolutionStrategy} converges, namely
	$$
	\mathbf{V}_h^\ell\times V_h^\ell\times V_h^\ell\times Q_h^m \ni
	(\mathbf{e}_{\mathbf{u}}^k,e_p^k,e_T^k,e_\varphi^k)\to\mathbf{0}\;\text{ as }\;k\to\infty.
	$$
\end{theorem}

\section{Numerical results}
\label{sec:NumericalResults}
\begin{figure}[ht]
	\centering
	\input{Fig/Splitting_F_H_M.tikz}
	\caption{Graphical representation of the \textbf{F-H-M} splitting solution strategy}
\end{figure}
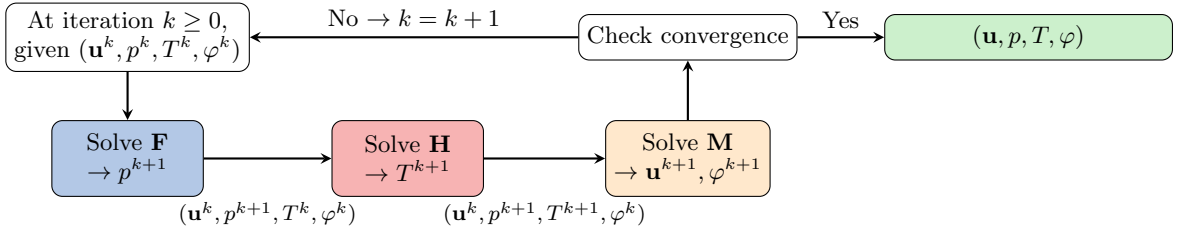

The aim of this section is to assess the performance of the scheme in terms of accuracy and robustness. For the two-dimensional simulations, the sequences of polygonal Voronoi meshes has been generated with \texttt{Polymesher} \cite{Talischi2012}. All the penalty coefficients $\alpha_i$, $i=1,\dots,4$ in \eqref{eq:TPE_stabilization_func} are set equal to $10$. The simulations in the two-dimensional setting have been carried out in \texttt{lymph} \cite{Antonietti2025}, while we have used \texttt{FEniCS} \cite{Alnaes2015} for the three-dimensional ones.

Thanks to the introduction of the total pressure stabilization term, we can use equal order approximations for the four unknowns of the problem. Thus, in every test we set $\ell = m$ and, for the sake of simplicity, we make use only of the symbol $\ell$ to denote the polynomial degree.


\subsection{Convergence Test in 2D}
\label{sec:ConvTest_2D}
We start the analysis by assessing the performance of the method in terms of accuracy. We consider problem \eqref{eq:TPE_system} in the domain $\Omega = (0,1)^2$ with the following manufactured analytical solution:
\begin{equation}
	\begin{aligned}
		\mathbf{u}(x,y) & \ = \left( \begin{aligned}
			& \sin (2 \pi y) \, (\cos(2 \pi x)-1) + \frac{1}{\mu + \lambda}\, \sin(\pi x) \sin(\pi y) \\
			& \sin (2 \pi x) \, (1-\cos(2 \pi y)) + \frac{1}{\mu + \lambda}\, \sin(\pi x) \sin(\pi y)
		\end{aligned} \right), \\[5pt]
		p(x,y) & \ = \sin(\pi x) \sin(\pi y), \ \
		T(x,y) = \sin(\pi x) \sin(\pi y),
	\end{aligned}
\end{equation}
through which we infer both the boundary conditions and forcing terms. The model coefficients are reported in Table~\ref{tab:TPE_params_convtest}, see e.g., \cite{Li2025}. 
\begin{table}[ht]
	\centering 
	\footnotesize
	\begin{tabular}{l | c c l | c c l | c}
		$a_0 \ [\si[per-mode = symbol]{\pascal \per \kelvin\squared}]$ & 0.2 & & $\alpha \ [-]$ & 0.1 & & $\mu, \lambda \ [\si{\pascal}]$ & 1, 20  \\
		$b_0 \ [\si{\per \kelvin}]$ & 0.1 & & $\beta \ [\si{\pascal \per \kelvin}]$ & 0.1 & & $\mathbf{K} \ [\si{\meter \squared \per \pascal \per \second}]$ & 1 \\
		$c_0 \ [\si{\per \pascal}]$ & 0.3 & & $c_f \ [\si{\pascal \per \kelvin\squared}]$ & 0.1 & & $\boldsymbol{\Theta} \ [\si{\meter \squared \pascal \per \kelvin\squared \per \second}]$ & 1
	\end{tabular}
	\caption{Test case of Section~\ref{sec:ConvTest_2D}: problem parameters for the convergence analysis.}
	\label{tab:TPE_params_convtest}
\end{table}

The convergence of the PolyDG scheme is tested both with respect to the mesh size $h$ and to the polynomial degree $\ell$. For the $h$-convergence a sequence of polygonal Voronoi meshes is considered and we set the polynomial degree $\ell = 2$. For what concerns the convergence with respect to $\ell$ and for a fixed mesh size, we fix a computational mesh of $50$ elements and vary the polynomial degree $\ell = 1,2,\dots,8$.
In Figure~\ref{fig:ConvH_2D}, Figure~\ref{fig:ConvP} we show the computed errors in the $L^2$- and $dG$-norms defined as in \eqref{eq:RobQSTPE_DG_norms} versus $h$ and $\ell$, respectively. In both cases, we observe that the results match the predicted convergence rates in the framework of PolyDG spatial discretizations \cite[Theorem 5.3]{AntoniettiBonetti2022}. 

\input{Fig/ConvH_2D.tikz}
\input{Fig/ConvP.tikz}

For what concerns the convergence vs $h$ we observe that, by using $\ell = 2$, the $L^2$- and $dG$-errors for all the three unknowns go to zero as $h^3$ and $h^2$, respectively, as predicted by the theory (see e.g., \cite{AntoniettiBonetti2022}). We point out that the lines representing the error trends for pressure and temperature are superimposed. For what concerns the $L^2$-errors, we achieve $h^{\ell+1}$ convergence. Looking at the convergence with respect to $\ell$ we see that in the two cases, both for the $L^2$- and the $dG$-errors, we observe an exponential decrease of the error.

We conclude this section by investigating the iteration counts that are needed for achieving the desired convergence of the three different schemes. We observe that a trivial comparison of the iteration counts of the three methods is not sufficient, since the specific iteration will have a different computational cost. For the three methods we consider the same stopping criterion: we stop the scheme when the norm of the absolute difference of two successive iterations is below $\num[exponent-product=\ensuremath{\cdot}, print-unity-mantissa=false]{1e-6}$ or when the norm of the relative difference of two successive iterations is below $\num[exponent-product=\ensuremath{\cdot}, print-unity-mantissa=false]{1e-6}$ as in \cite{Brun2020}. In Table~\ref{tab:ConvH_itertimes}, Table~\ref{tab:ConvP_itertimes} we display the iteration counts for the three methods and the corresponding speed-ups (for the splitting solution strategies) in the $h$-convergence test and in the $\ell$-convergence test, respectively. The correspondiing computational times [\si{\second}] can be found in \ref{sec:appendix}. First, we observe that all the three solution strategy converge for every refinement considered and that the iteration counts are independent of the number of elements in the mesh. We remark that this makes the solution strategies extremely scalable with respect to $N_{el}$. For what concerns the convergence with respect to the mesh size $h$, we also notice a significant speed-up for the two splitting solution strategies -- that perform in a similar way -- with respect to \textbf{FHM}, while the speed-up is less evident in the convergence test vs the polynomial degree of approximation. We observe that one of the reasons for which we obtain this speed-up relies in the fact that \texttt{lymph} solves the resulting linear systems in a direct fashion. Indeed, in the splitting solution strategies it is possible to store the factorization of the matrices related to \textbf{F}, \textbf{M}, and \textbf{FM} problems computed at the first iteration and then exploit them in the remaining iterations.

\begin{table}[ht]
	\centering 
	\footnotesize
	\begin{tabular}{c c c c c c c c c c}
		& $N_{\text{el}}$ & $50$ & $100$  & $310$ & $1000$  & $3100$ & $10000$  & $31000$ & $100000$ \B\\
		\hline
		\textbf{FHM} 
		& \#it & $4$ & $4$ & $4$ & $4$ & $4$ & $4$ & $4$ & $4$ \T\B \\
		\hline
		\multirow{2}{*}{\textbf{FM-H}} 
		& \#it & $4$ & $4$ & $4$ & $4$ & $4$ & $4$ & $4$ & $4$ \T\\
		& Speed-up & 1.19 & 1.31 & 1.49 & 1.73 & 2.13 & 2.58 & 3.30 & 4.83 \B\\
		\hline
		\multirow{2}{*}{\textbf{F-H-M}} 
		& \#it & $3$ & $3$ & $3$ & $3$ & $3$ & $3$ & $3$ & $3$ \T\\
		& Speed-up & 0.46 & 1.31 & 1.59 & 1.91 & 2.29 & 2.86 & 3.69 & 4.93 \B\\ 
		\hline
	\end{tabular}
	\caption{Convergence test vs $h$ in 2D: iteration counts of the three solution algorithms for the convergence and speed-ups of the splitting solution algorithms with respect to the \textbf{FHM} scheme versus the  number of elements. The computational times [\si{\second}] are reported in \ref{sec:appendix}.}
	\label{tab:ConvH_itertimes}
\end{table}

\begin{table}[H]
	\centering 
	\footnotesize
	\begin{tabular}{c c c c c c c c c c}
		& $\ell$ & $1$  & $2$ & $3$  & $4$ & $5$  & $6$ & $7$ & $8$ \B\\
		\hline
		\textbf{FHM}
		& \#it & $4$ & $4$ & $4$ & $4$ & $4$ & $4$ & $4$ & $4$ \T\B\\
		\hline
		\multirow{2}{*}{\textbf{FM-H}} 
		& \#it  & $4$ & $4$ & $4$ & $4$ & $4$ & $4$ & $4$ & $4$ \T\\
		& Speed-up & 3.19 & 1.45 & 1.49 & 1.59 & 1.80 & 1.82 & 1.87 & 1.96 \T\B \\
		\hline
		\multirow{2}{*}{\textbf{F-H-M}} 
		& \#it  & $4$ & $4$ & $4$ & $4$ & $4$ & $4$ & $4$ & $4$ \T\\
		& Speed-up & 3.49 & 1.41 & 1.37 & 1.44 & 1.58 & 1.51 & 1.17 & 1.13 \T\B \\
		\hline
	\end{tabular}
	\caption{Convergence test vs $\ell$ in 2D: iteration counts of the three solution algorithms for the convergence and speed-ups of the splitting solution algorithms with respect to the \textbf{FHM} scheme versus the polynomial degree of approximation. The computational times [\si{\second}] are reported in \ref{sec:appendix}.}
	\label{tab:ConvP_itertimes}
\end{table}

\subsection{Convergence Test in 3D}
\label{sec:ConvTest_3D}
We now assess the performance of the solution strategies, in terms of accuracy and computational costs, in a three-dimensional setting. We consider problem \eqref{eq:TPE_system} in the domain $\Omega = (0,1)^3$ with the following manufactured analytical solution:
\begin{equation}
	\begin{aligned}
		\mathbf{u}(x,y) & \ = \left( \begin{aligned}
			& 2 (\cos(2 \pi x)-1) \, \sin (2 \pi y) \, \sin (2 \pi z) + \frac{1}{\mu + \lambda}\, \sin(\pi x) \sin(\pi y) \sin(2\pi z)\\
			& \sin (2 \pi x) \, (1-\cos(2 \pi y)) \, \sin(2 \pi z) + \frac{1}{\mu + \lambda}\, \sin(\pi x) \sin(\pi y) \sin(2\pi z) \\
			& \sin (2 \pi x) \, \sin(2\pi y) \, (1-\cos(2 \pi z)) + \frac{1}{\mu + \lambda}\, \sin(\pi x) \sin(\pi y) \sin(2\pi z) \\
		\end{aligned} \right), \\[5pt]
		p(x,y) & \ = \sin(\pi x) \sin(\pi y) \sin(\pi z), \ \
		T(x,y) = \sin(\pi x) \sin(\pi y) \sin(\pi z),
	\end{aligned}
\end{equation}
through which we infer the boundary conditions and forcing terms. The model coefficients are reported in Table~\ref{tab:TPE_params_convtest} and are the same of Section~\ref{sec:ConvTest_2D}. The convergence of the DG scheme is tested with respect to the mesh size $h$. To this aim, a sequence of theatrahedral meshes is considered and we set the polynomial degree $\ell = 1$. In Figure~\ref{fig:ConvH_3D} we show the computed errors in the $L^2$- and $dG$-norms defined as in \eqref{eq:RobQSTPE_DG_norms} versus $h$ and and we observe that the results match the predicted convergence rates in the framework of PolyDG spatial discretizations \cite[Theorem 5.3]{AntoniettiBonetti2022}. Indeed, we observe that, by using $\ell = 1$, the $L^2$- and $dG$-errors for all the three unknowns go to zero as $h^2$ and $h$, respectively.

\input{Fig/ConvH_3D.tikz}

As before, we conclude this section by investigating the iteration counts for the three different schemes and the associated speed-ups. We consider the same stopping criterion as in Section~\ref{sec:ConvTest_2D}. In Table~\ref{tab:ConvH3D_itertimes}, we display the iteration counts needed for the convergence of the three methods and the speed-ups (for the splitting solution strategies) in the $h$-convergence test. The corresponding computational times [\si{\second}] can be found in \ref{sec:appendix}. As observed for the two-dimensional setting, we can appreciate the scalability of the solution algorithms with respect to the considered number of elements. Also in this case we observe a good speed-up, especially for the \textbf{F-H-M} strategy, even if it is less evident with respect to the the two-dimensional setting.

\begin{remark}
	In this setting, the linear systems cannot be solved with a direct method, but an iterative solver needs to be introduced. To compare the different solution strategies without influencing the results by the solvers of the individual subproblems, we consider the solution of all linear systems using \texttt{GMRES} with an \texttt{ILU} preconditioner. We note that a more extensive investigation could be conducted for the individual subproblems, also taking into account the robustness of the solution strategies with respect to the model’s physical parameters (cf. \ref{sec:RobTest}). The implementation of these techniques would improve the performance of splitting methods compared to the \textbf{FHM} strategy, allowing us to achieve a larger speed-up. This will be the subject of future work. 
\end{remark}

\begin{table}[H]
	\centering 
	\footnotesize
	\begin{tabular}{c c c c c c c c c c}
		& $N_{\text{el}}$ & $48$ & $1296$  & $6000$ & $16464$  & $34992$ & $63888$  & $105456$ & $162000$ \B\\
		\hline
		\textbf{FHM} 
		& \#it & $3$ & $3$ & $3$ & $4$ & $4$ & $4$ & $4$ & $4$ \T\B\\
		\hline
		\multirow{2}{*}{\textbf{FM-H}} 
		& \#it & $3$ & $3$ & $3$ & $3$ & $3$ & $3$ & $3$ & $3$ \T\\
		& Speed-up & $1.05$ & $1.39$ & $1.39$ & $1.69$ & $1.73$ & $1.73$ & $1.64$ & $1.60$ \T\B\\  
		\hline
		\multirow{2}{*}{\textbf{F-H-M}} 
		& \#it & $3$ & $3$ & $3$ & $3$ & $3$ & $3$ & $3$ & $3$ \T\\
		& Speed-up & $1.00$ & $1.82$ & $1.91$ & $2.38$ & $2.48$ & $2.41$ & $2.56$ & $3.07$ \T\B\\
		\hline
	\end{tabular}
	\caption{Convergence test vs $h$ in 3D: iteration counts of the three solution algorithms for the convergence and speed-ups of the splitting solution algorithms with respect to the \textbf{FHM} scheme versus the  number of elements. The computational times [\si{\second}] are reported in \ref{sec:appendix}.}
	\label{tab:ConvH3D_itertimes}
\end{table}

\subsection{Robustness test}
\label{sec:RobTest}
The aim of this Section is to investigate the robustness properties of the solution strategies. In \cite{AntoniettiBonetti2022} the robustness of the four-fields PolyDG approximation of the TPE problem, coupled with the fixed-point iteration strategy (cf. \textbf{FHM} scheme) has already been addressed for the quasi-static problem. Moreover, in \cite{Bonetti2024} the robustness of this formulation has been studied for the steady problem, with a particular focus on the nonlinear convection term. We focus on the four test (inspired by \cite{AntoniettiBonetti2022, Li2025}) listed in Table~\ref{tab:TPE_params_robtest}, while the parameters $\alpha$, $\beta$, $\mu$, and $c_f$ are chosen as in Table~\ref{tab:TPE_params_convtest}. The space discretization parameters and the stopping criterion are chosen as in the $h$-convergence test case presented in Section~\ref{sec:ConvTest_2D}.

\begin{table}[H]
	\centering 
	\footnotesize
	\begin{tabular}{l c c c c}
		\textbf{Coefficient} & \textbf{Test $\mathbf{(i)}$} & \textbf{Test $\mathbf{(ii)}$} & \textbf{Test $\mathbf{(iii)}$} & \textbf{Test $\mathbf{(iv)}$} \B\\
		\hline 
		$a_0 \ [\si{\pascal \per \kelvin\squared}]$ & 0 & 0.2 & 0.2 & 0.2 \T\\
		$b_0 \ [\si{\per \kelvin}]$ & 0 & 0.1 & 0.1 & 0.1  \\
		$c_0 \ [\si{\per \pascal}]$ & 0 & 0.3 & 0.3 & 0.3  \\
		$\lambda \ [\si{\pascal}]$ & \num[exponent-product=\ensuremath{\cdot}, print-unity-mantissa=false]{e9} & 1 & 1 & 1  \\
		$\mathbf{K} \ [\si{\meter \squared \per \pascal \per \second}]$ & $\mathbf{I}$ & \num[exponent-product=\ensuremath{\cdot}, print-unity-mantissa=false]{e-9}$\mathbf{I}$  & $\mathbf{I}$ & \num[exponent-product=\ensuremath{\cdot}, print-unity-mantissa=false]{e-9}$\mathbf{I}$  \\
		$\boldsymbol{\Theta} \ [\si{\meter \squared \pascal \per \kelvin\squared \per \second}]$ & $\mathbf{I}$  & \num[exponent-product=\ensuremath{\cdot}, print-unity-mantissa=false]{e-9}$\mathbf{I}$  & \num[exponent-product=\ensuremath{\cdot}, print-unity-mantissa=false]{e-9}$\mathbf{I}$ & $\mathbf{I}$
	\end{tabular}
	\caption{Test cases of Section~\ref{sec:RobTest}: problem parameters for the robustness analysis.}
	\label{tab:TPE_params_robtest}
\end{table}

We recall that increasing the value of the first Lamè coefficient $\lambda \gg 1$ means going towards the quasi-incompressible limit. In Figure~\ref{fig:Rob} we report the results of the robustness tests for the three solution strategies. Along the rows of Figure~\ref{fig:Rob} we report the dG-errors for the displacement, pressure, and temperature fields, respectively, while along the columns we report the errors for the same field in different model parameters' configurations. Then, for every subplot of Figure~\ref{fig:Rob} we compare the performance of the three solution strategies keeping fixed the test case and the unknown of interest. First, we notice that for (almost) all the test cases and for all the unknowns the computed results for the three schemes are very similar. The only exception is the first test case $\mathbf{(iv)}$ for which the \textbf{FHM} fails in solving the last refinement. However, the results for previous mesh levels are the same for the three schemes. In terms of robustness of the method, the results are coherent with our expectations and with what has been observed in \cite{AntoniettiBonetti2022, Bonetti2024}. Indeed, the schemes are robust in all the four tested scenarios and we always achieve convergence with the desired convergence rate, i.e. $h^{\ell}$ (as already shown in Section~\ref{sec:ConvTest_2D}, Section~\ref{sec:ConvTest_3D}). By looking at the absolute values of the errors, we observe that the displacement-error values are the same for all the tests. For what concerns the pressure and temperature errors, we notice that their values are also affected by the absolute values of the conductivity tensors, that enter in the definition of the $dG$-norms, cf. \eqref{eq:RobQSTPE_DG_norms}. We remark that, with the stabilized convective trilinear form, we have achieve convergence also in the advection-dominated regime $\mathbf{(iii)}$. The results are in agreement with what observed in \cite{AntoniettiBonetti2022, Bonetti2024}.
\input{Fig/Rob.tikz}

In Table~\ref{tab:Rob1_itertimes}--~\ref{tab:Rob4_itertimes} we report the iteration counts for the convergence of the three methods and the computed speed-ups (for the splitting schemes) versus the number of elements. The corresponding computational times [\si{\second}] can be found in \ref{sec:appendix}. As already observed for the convergence test, we see that all the solution strategies are scalable with respect to the number of elements, as the iteration counts for achieving convergence are independent with respect to the number of elements considered in the mesh. Second, we observe that in the first three tests -- especially in $\mathbf{(i)}$, $\mathbf{(iii)}$ -- we achieve a remarkable speed-up with the use of splitting strategies with respect to the \textbf{FHM} scheme. There is no clear evidence that one of the two splitting strategies can guarantee better results in terms of computational times, but it seems that when very small coefficients appear in the model, keeping the coupling and not splitting the overall problem can give better results with respect to solving the subproblems separately. We remark that these considerations are in agreement with the theory developed in \cite{AntoniettiBonetti2022}. This is more evident in the test case $\mathbf{(iv)}$, where the \textbf{FHM} performs better in the first refinements with respect to the splitted strategies. However, we observe that considering a splitting scheme gets more and more advantageous when increasing the number of elements. Indeed, in the last refinement, we are not able to solve the problem with the monolithic scheme, but we achieve convergence with the splitting ones.

\begin{table}[H]
	\centering 
	\footnotesize
	\begin{tabular}{c c c c c c c c c c}
		& $N_{\text{el}}$ & $50$ & $100$  & $310$ & $1000$  & $3100$ & $10000$  & $31000$ & $100000$ \B\\
		\hline
		\textbf{FHM}
		& \#it & $5$ & $3$ & $3$ & $3$ & $3$ & $3$ & $3$ & $3$ \T\B\\
		\hline
		\multirow{2}{*}{\textbf{FM-H}} 
		& \#it & $3$ & $3$ & $3$ & $3$ & $3$ & $3$ & $3$ & $3$ \T\\
		& Speed-up & 1.59 & 1.30 & 1.45 & 1.64 & 1.98 & 2.34 & 2.91 & 4.12  \B\\
		\hline
		\multirow{2}{*}{\textbf{F-H-M}} 
		& \#it & $2$ & $2$ & $2$ & $2$ & $2$ & $2$ & $2$ & $2$ \T\\
		& Speed-up & 0.57 & 1.39 & 1.66 & 2.00 & 2.40 & 2.99 & 3.91 & 5.23 \B\\
		\hline
	\end{tabular}
	\caption{Robustness test $\mathbf{(i)}$: iteration counts of the three solution algorithms for the convergence and speed-ups of the splitting solution algorithms with respect to the \textbf{FHM} scheme versus the  number of elements. The computational times [\si{\second}] are reported in \ref{sec:appendix}.}
	\label{tab:Rob1_itertimes}
\end{table}

\begin{table}[H]
	\centering 
	\footnotesize
	\begin{tabular}{c c c c c c c c c c}
		& $N_{\text{el}}$ & $50$ & $100$  & $310$ & $1000$  & $3100$ & $10000$  & $31000$ & $100000$ \B\\
		\hline
		\textbf{FHM} 
		& \#it & $2$ & $2$ & $2$ & $2$ & $2$ & $2$ & $2$ & $2$ \T\B\\
		\hline
		\multirow{2}{*}{\textbf{FM-H}} 
		& \#it & $4$ & $4$ & $4$ & $4$ & $4$ & $4$ & $4$ & $4$ \T\\
		& Speed-up & 0.83 & 0.88 & 0.97 & 1.07 & 1.24 & 1.46 & 1.79 & 2.86 \B\\
		\hline
		\multirow{2}{*}{\textbf{F-H-M}} 
		& \#it & $5$ & $5$ & $5$ & $5$ & $5$ & $5$ & $5$ & $5$ \T\\
		& Speed-up & 0.28 & 0.67 & 0.73 & 0.80 & 0.92 & 1.0656 & 1.29 & 1.93 \B\\
		\hline
	\end{tabular}
	\caption{Robustness test $\mathbf{(ii)}$: iteration counts of the three solution algorithms for the convergence and speed-ups of the splitting solution algorithms with respect to the \textbf{FHM} scheme versus the  number of elements. The computational times [\si{\second}] are reported in \ref{sec:appendix}.}
	\label{tab:Rob2_itertimes}
\end{table}

\begin{table}[H]
	\centering 
	\footnotesize
	\begin{tabular}{c c c c c c c c c c}
		& $N_{\text{el}}$ & $50$ & $100$  & $310$ & $1000$  & $3100$ & $10000$  & $31000$ & $100000$ \B\\
		\hline
		\textbf{FHM}
		& \#it & $5$ & $5$ & $5$ & $5$ & $5$ & $5$ & $5$ & $5$ \T\B\\
		\hline
		\multirow{2}{*}{\textbf{FM-H}} 
		& \#it & $5$ & $5$ & $5$ & $5$ & $5$ & $5$ & $5$ & $5$ \T\\
		& Speed-up & 1.17 & 1.33 & 1.53 & 1.81 & 2.27 & 2.82 & 3.67 & 5.59 \B\\
		\hline
		\multirow{2}{*}{\textbf{F-H-M}} 
		& \#it & $5$ & $4$ & $5$ & $5$ & $5$ & $5$ & $5$ & $5$ \T\\
		& Speed-up & 0.44 & 1.33 & 1.33 & 1.56 & 1.88 & 2.32 & 2.94 & 4.06 \B\\
		\hline
	\end{tabular}
	\caption{Robustness test $\mathbf{(iii)}$: iteration counts of the three solution algorithms for the convergence and speed-ups of the splitting solution algorithms with respect to the \textbf{FHM} scheme versus the  number of elements. The computational times [\si{\second}] are reported in \ref{sec:appendix}.}
	\label{tab:Rob3_itertimes}
\end{table}

\begin{table}[H]
	\centering 
	\footnotesize
	\begin{tabular}{c c c c c c c c c c}
		& $N_{\text{el}}$ & $50$ & $100$  & $310$ & $1000$  & $3100$ & $10000$  & $31000$ & $100000$ \B\\
		\hline
		\textbf{FHM}
		& \#it & $2$ & $2$ & $2$ & $2$ & $2$ & $2$ & $2$ & $-$\T\B\\
		\hline
		\multirow{2}{*}{\textbf{FM-H}} 
		& \#it & $3$ & $3$ & $3$ & $3$ & $3$ & $3$ & $3$ & $3$ \T\\
		& Speed-up & 0.48 & 0.49 & 0.54 & 0.60 & 0.71 & 0.86 & 1.21 & $-$ \B\\
		\hline
		\multirow{2}{*}{\textbf{F-H-M}} 
		& \#it & $2$ & $2$ & $2$ & $2$ & $2$ & $2$ & $2$ & $2$ \T\\
		& Speed-up & 0.21 & 0.39 & 0.42 & 0.45 & 0.50 & 0.57 & 0.74 & $-$ \B\\
		\hline
	\end{tabular}
	\caption{Robustness test $\mathbf{(iv)}$: iteration counts of the three solution algorithms for the convergence and speed-ups of the splitting solution algorithms with respect to the \textbf{FHM} scheme versus the  number of elements. The computational times [\si{\second}] are reported in \ref{sec:appendix}.}
	\label{tab:Rob4_itertimes}
\end{table}

\section{Geothermal test case in three-dimensions}
\label{sec:geo3D}
In this section we present the results for a realistic model problem inspired by geothermal energy production \cite{AntoniettiBonetti2022}. We consider a box domain $\Omega = (0,4\si{\kilo \meter}) \times (0,2\si{\kilo \meter}) \times (0,2\si{\kilo \meter})$ and we discretize it with a tetrahedral mesh. We consider different levels of refinement for the mesh: $96000$ tetrahedrons with mesh size $h \sim 0.17 \si{\kilo \meter}$, $324000$ tetrahedrons with mesh size $h \sim 0.12 \si{\kilo \meter}$, $768000$ tetrahedrons with mesh size $h \sim 0.09 \si{\kilo \meter}$, and $1500000$ tetrahedrons with mesh size $h \sim 0.07 \si{\kilo \meter}$. The approximation degree is $\ell = 1$. On top of these meshes, we compare the performance of the three algorithms \textbf{FHM}, \textbf{FM-H}, and \textbf{F-H-M}. In order to mimic the injection and extraction of a fluid (e.g. water) in the subsoil we impose the following set of boundary conditions:
\begin{equation}
	\begin{aligned}
		& \mathbf{u} = 0, && p = p_{\text{inj}}, && T = T_{\text{inj}}, \quad && \text{on} \ \Gamma_{\text{inj}}, \\
		& \mathbf{u} = 0, \quad && p = p_{\text{ext}}, \quad && \gamma (T-T_{\text{ref}}) + \boldsymbol{\Theta} \nabla T \cdot \mathbf{n} = 0, \quad && \text{on} \ \Gamma_{\text{ext}}, \\
		& \boldsymbol{\sigma} \mathbf{n} = 0, \quad && \mathbf{K} \nabla p \cdot \mathbf{n}  = 0, \quad && \gamma (T-T_{\text{ref}}) + \boldsymbol{\Theta} \nabla T \cdot \mathbf{n} = 0, \quad && \text{on} \ \partial \Omega \setminus \left(\Gamma_{\text{inj}} \cup \Gamma_{\text{ext}}\right), \\
	\end{aligned}
\end{equation}
where the parameter $\gamma$ is taken equal to $0.05$, $T_{\text{ref}} = 20\si{\celsius}$, $\Gamma_{\text{inj}} = \{0\} \times (0,2\si{\kilo \meter}) \times (0,2\si{\kilo \meter})$, and $\Gamma_{\text{ext}} = \{4\si{\kilo\meter}\} \times (0,2\si{\kilo \meter}) \times (0,2\si{\kilo \meter})$. The injection temperature profile is taken as and $T_{\text{in}}$ take the following general form
\begin{equation}
	\phi (y,z) = 
	\begin{cases}
		\phi_{m} & \text{if} \ 0 \leq y < a,\\
		\phi_{m} + \frac{\phi_{M}-\phi_{m}}{2}\left(1 - \cos\left(\pi \frac{x - a}{b-a}\right)\right)  & \text{if} \ a \leq y < b,\\
		\phi_{M} & \text{if} \ b \leq y < c,\\
		\phi_{m} + \frac{\phi_{M}-\phi_{m}}{2}\left(1 - \cos\left(\pi \frac{x - c}{d-c}\right)\right) & \text{if} \ c \leq y < d,\\
		\phi_{m} & \text{if} \ d \leq y \leq e,
	\end{cases}
\end{equation}
with the following choice of parameters 
\begin{equation}
	\begin{aligned}
		T_{\text{inj}}= -\phi(x) && \textrm{ with } (a,b,c,d,e,\phi_{m}, \phi_{M}) = (0.7\si{\kilo\meter}, 0.9\si{\kilo\meter}, 1.1\si{\kilo\meter}, 1.3\si{\kilo\meter}, 2\si{\kilo\meter}, 0\si{\celsius},100\si{\celsius}).
	\end{aligned}
\end{equation}

We supplement our problem with zero loading terms $\mathbf{f}, g, h$, zero initial conditions $(\mathbf{u}_0, p_0, T_0)$. We perform the simulations with a realistic choice of parameters, in particular they are taken identical to \cite{Brun2020, Bonetti2023, Bonetti2024}, their values are reported in Table~\ref{tab:geo_model_params}. 
\begin{table}[ht]
	\centering 
	\footnotesize
	\begin{tabular}{l | c c l | c c l | c}
		$a_0 \ [\si[per-mode = symbol]{\mega\pascal \per \kelvin\squared}]$ & \num[exponent-product=\ensuremath{\cdot}]{4.6e-5} & & $\alpha \ [-]$ & 1.0 & & $\mu, \lambda \ [\si{\mega\pascal}]$ & \num[exponent-product=\ensuremath{\cdot}]{2.475e+2}, \num[exponent-product=\ensuremath{\cdot}]{1.65e+2} \\
		$b_0 \ [\si{\per \kelvin}]$ & \num[exponent-product=\ensuremath{\cdot}]{3.03e-11} & & $\beta \ [\si{\mega\pascal \per \kelvin}]$ & \num[exponent-product=\ensuremath{\cdot}]{4.5e-1} & & $\mathbf{K} \ [\si{\kilo \m \squared \per \mega\pascal \per \milli\second}]$ & \num[exponent-product=\ensuremath{\cdot}]{9.87e-8}\textbf{I} \\
		$c_0 \ [\si{\per \mega\pascal}]$ & \num[exponent-product=\ensuremath{\cdot}]{3.03e-4} & & $c_f \ [\si{\mega \pascal \per \kelvin\squared}]$ & 4.186 & & $\boldsymbol{\Theta} \ [\si{\kilo \meter \squared \mega\pascal \per \kelvin\squared \per \milli\second}]$ & \num[exponent-product=\ensuremath{\cdot}]{1.7e-3}
	\end{tabular}
	\caption{Geothermal test case in three dimensions: problem parameters for the geothermal energy production test case}
	\label{tab:geo_model_params}
\end{table}

\begin{figure}[ht!]
	\centering
	
	\begin{subfigure}[b]{.45\textwidth}
		\centering
		\includegraphics[width=0.9\textwidth]{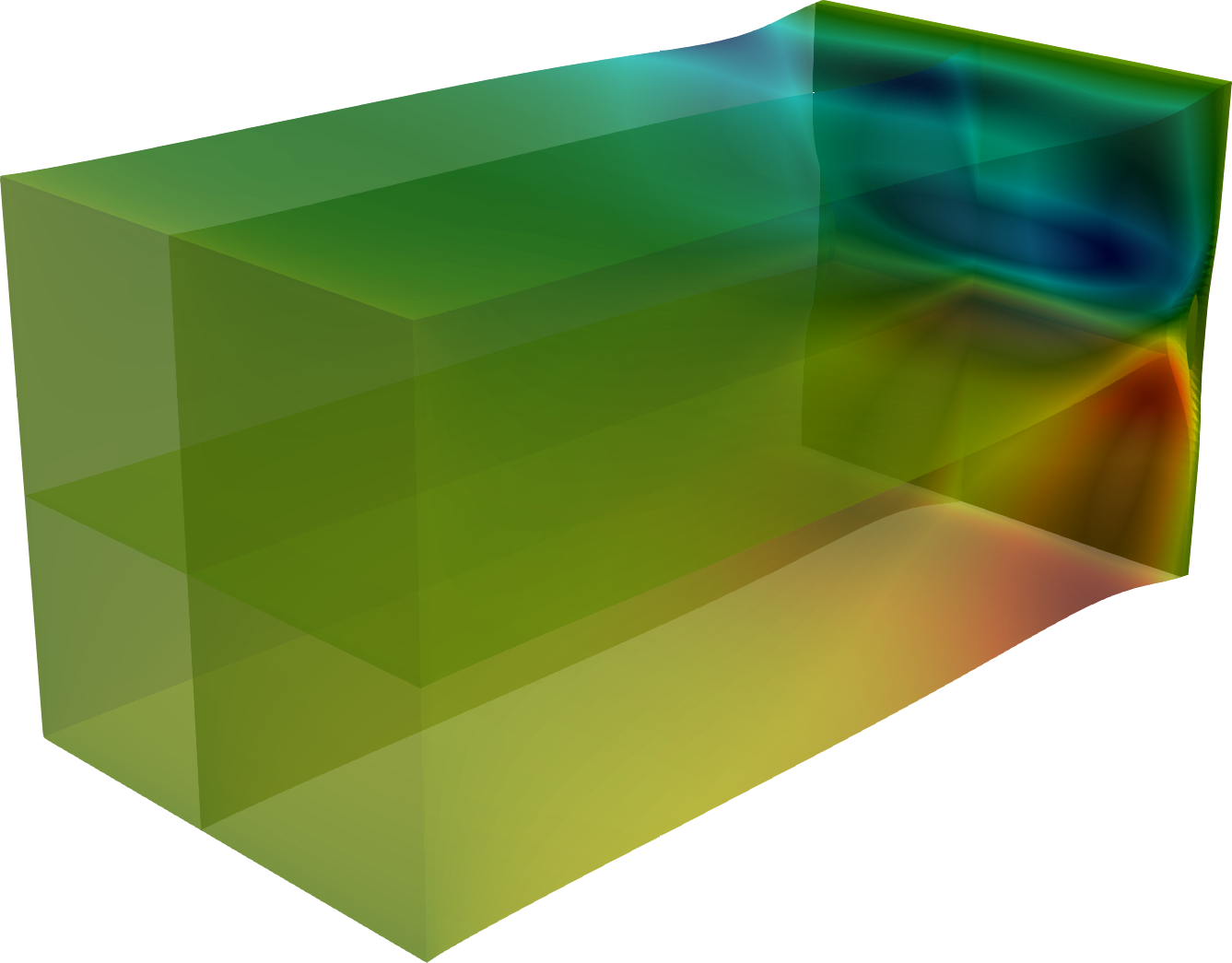}
		\label{fig:uy_F_H_M}
	\end{subfigure}
	\begin{subfigure}[b]{.45\textwidth}
		\centering
		\includegraphics[width=0.9\textwidth]{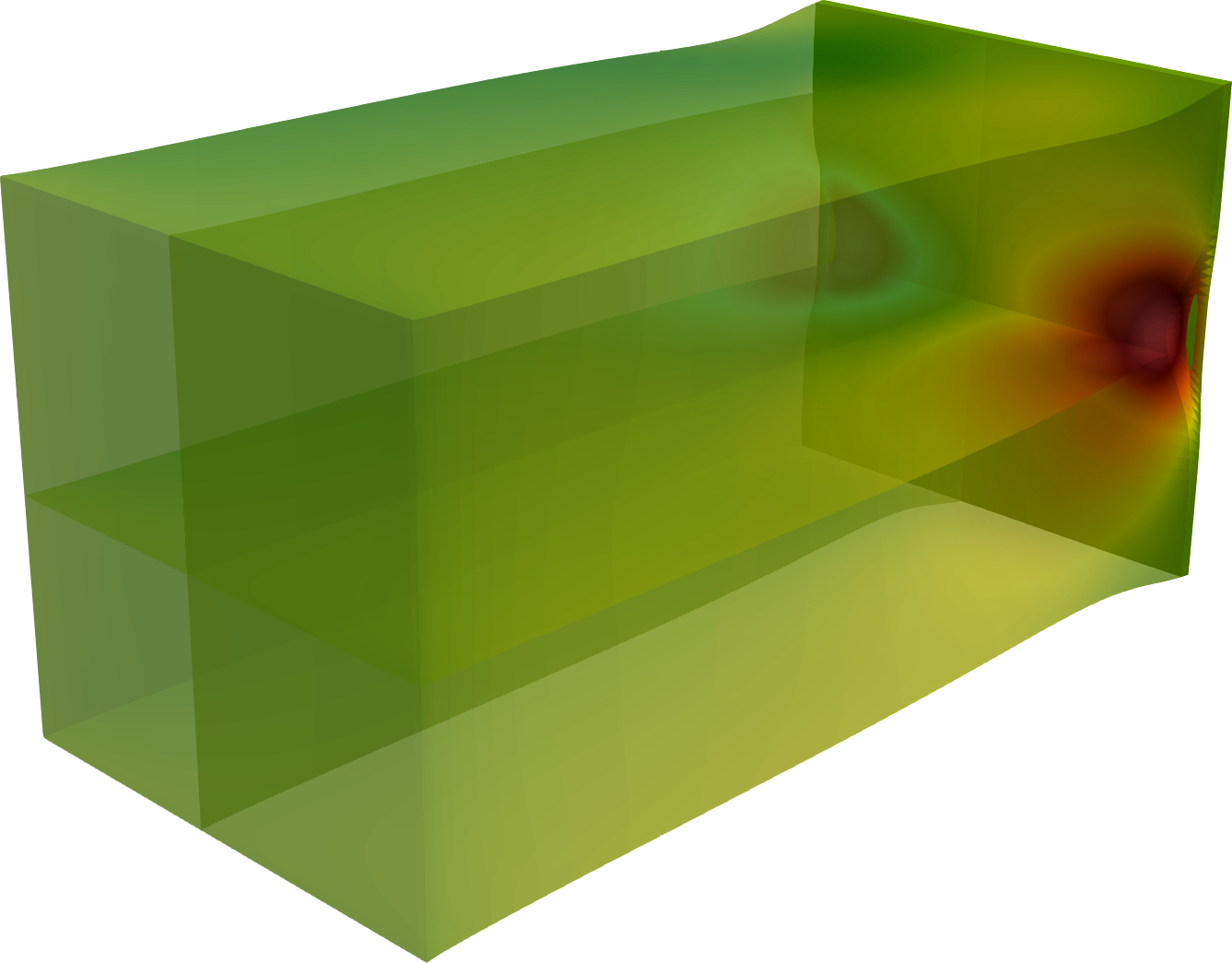}
		\label{fig:uz_F_H_M}
	\end{subfigure}
	
	\begin{subfigure}[b]{.45\textwidth}
		\centering
		\includegraphics[width=0.9\textwidth]{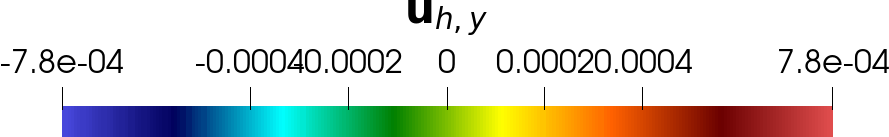}
		\label{fig:uy_F_H_M_cb}
	\end{subfigure}
	\begin{subfigure}[b]{.45\textwidth}
		\centering
		\includegraphics[width=0.9\textwidth]{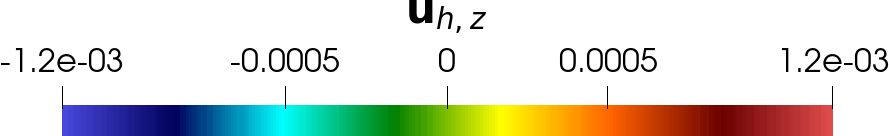}
		\label{fig:uz_F_H_M_cb}
	\end{subfigure}
	
	\vspace{0.5cm}
	
	\begin{subfigure}[b]{.45\textwidth}
		\centering
		\includegraphics[width=0.9\textwidth]{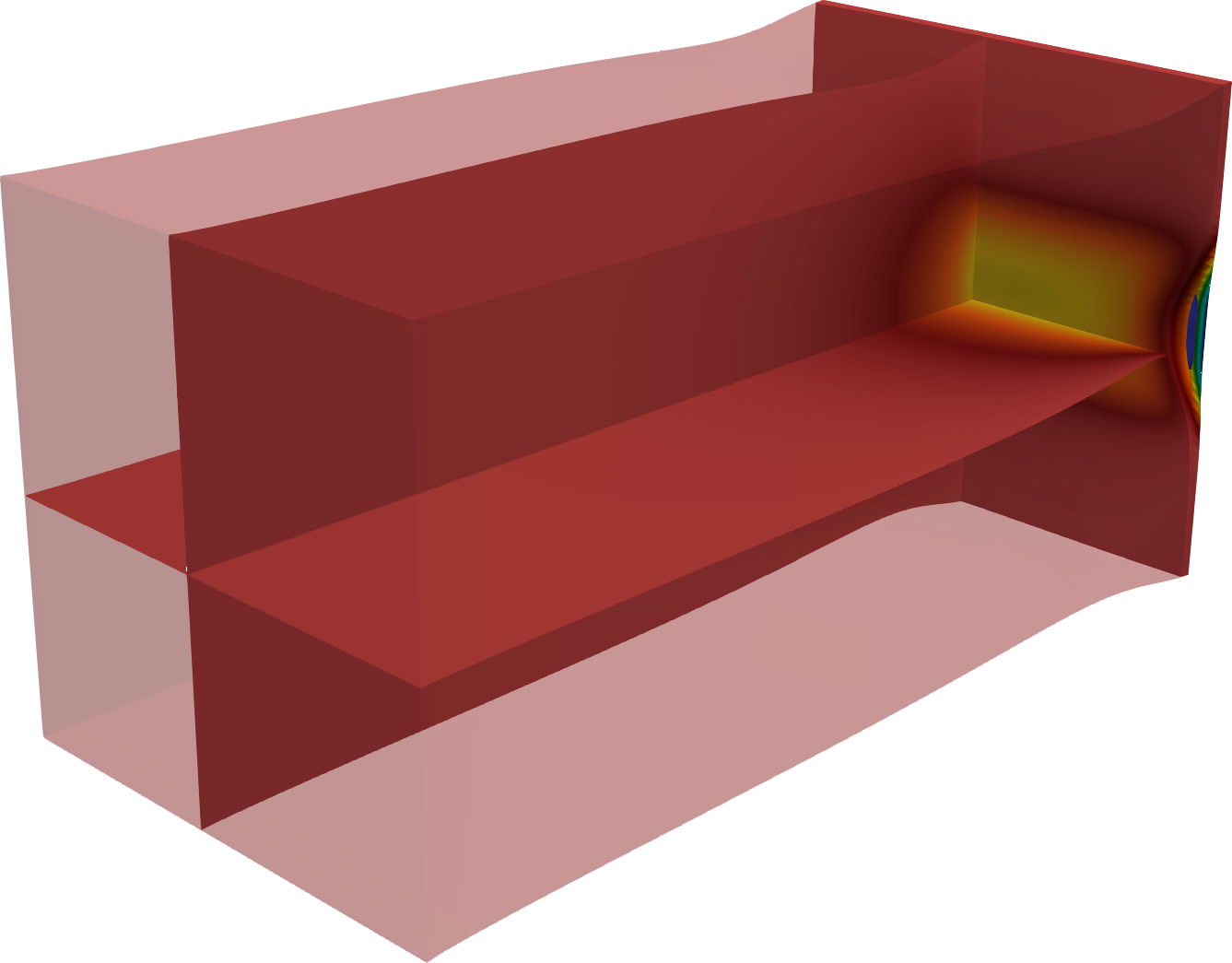}
		\label{fig:T_F_H_M}
	\end{subfigure}
	\begin{subfigure}[b]{.45\textwidth}
		\centering
		\includegraphics[width=0.9\textwidth]{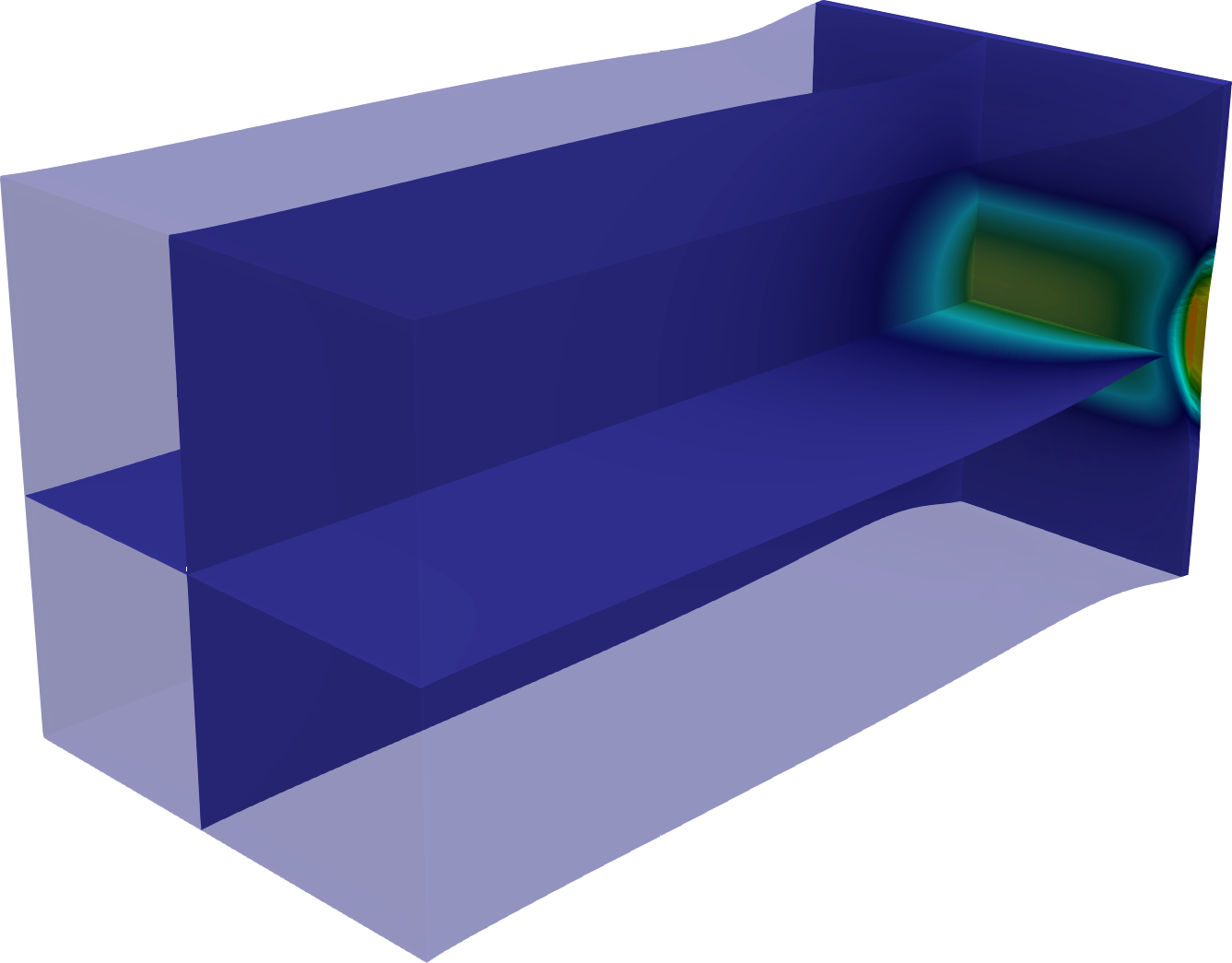}
		\label{fig:p_F_H_M}
	\end{subfigure}
	
	\begin{subfigure}[b]{.45\textwidth}
		\centering
		\includegraphics[width=0.9\textwidth]{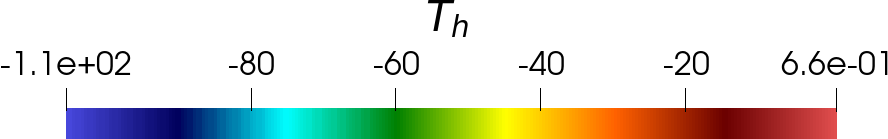}
		\label{fig:T_F_H_M_cb}
	\end{subfigure}
	\begin{subfigure}[b]{.45\textwidth}
		\centering
		\includegraphics[width=0.9\textwidth]{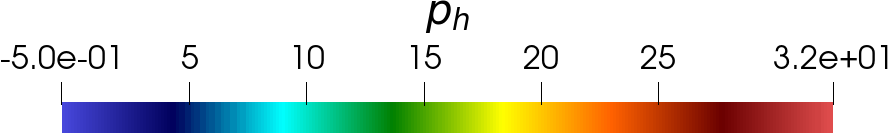}
		\label{fig:p_F_H_M_cb}
	\end{subfigure}
	\caption{Test case of Section~\ref{sec:geo3D}: computed displacement field in $y$- (top-left) and $z-$ (top-right) directions, computed temperature field (bottom-left), and pressure-field (bottom-right). The domain is clipped at $x=0.05$ and the two slices are at $y = 0.9$ and $z = 1.1$. The deformation is magnified by a factor $150$. The solution strategy is the \textbf{F-H-M} scheme and the solution is computed over $1.5$ millions of elements (total number of degrees of freedom $\num[exponent-product=\ensuremath{\cdot}, print-unity-mantissa=false]{3.6e+7}$).}
	\label{fig:geo_F_H_M_3D}
\end{figure}

\begin{figure}[ht!]
	\centering
	
	\begin{subfigure}[b]{.45\textwidth}
		\centering
		\includegraphics[width=0.9\textwidth]{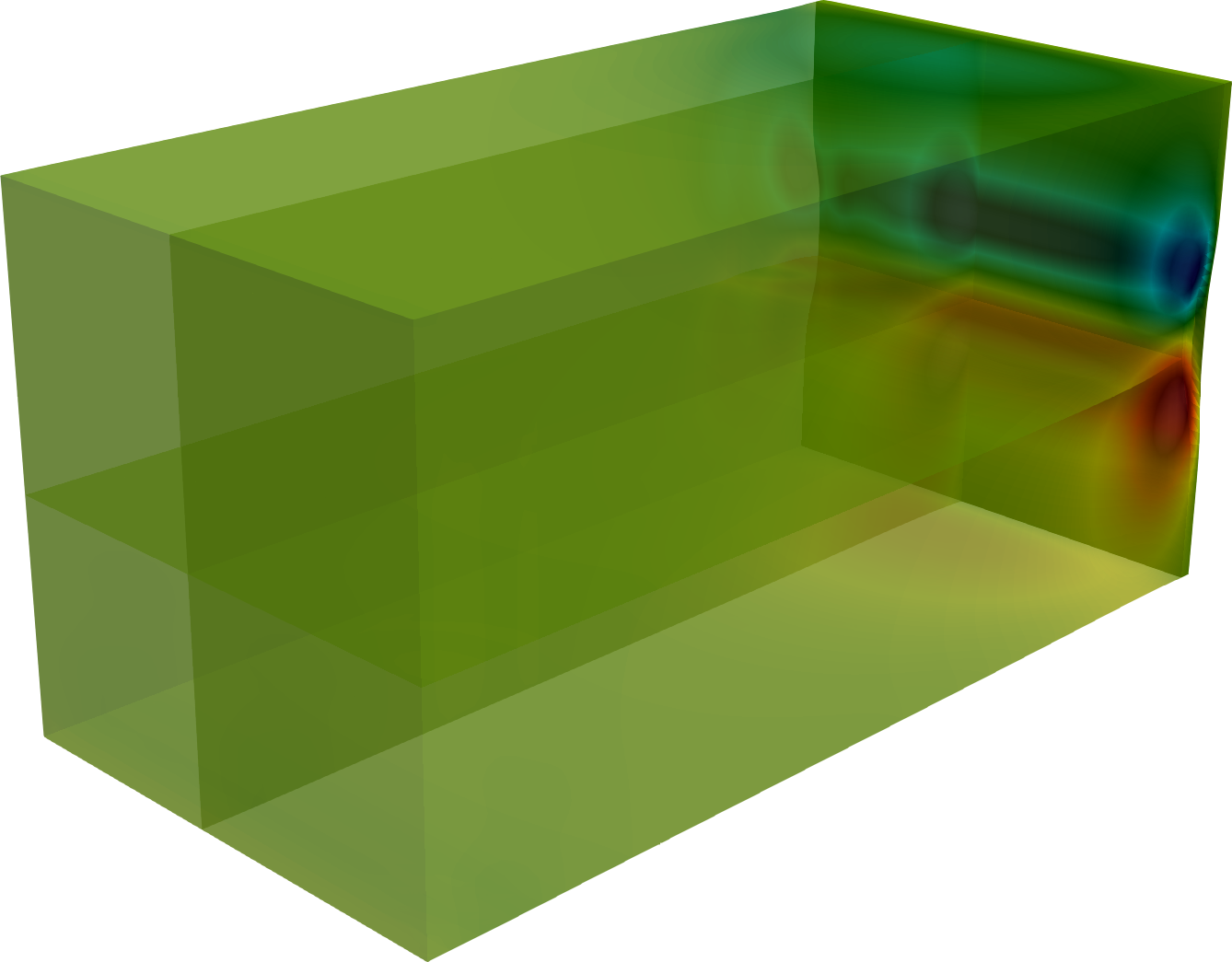}
		\label{fig:uy_FM_H}
	\end{subfigure}
	\begin{subfigure}[b]{.45\textwidth}
		\centering
		\includegraphics[width=0.9\textwidth]{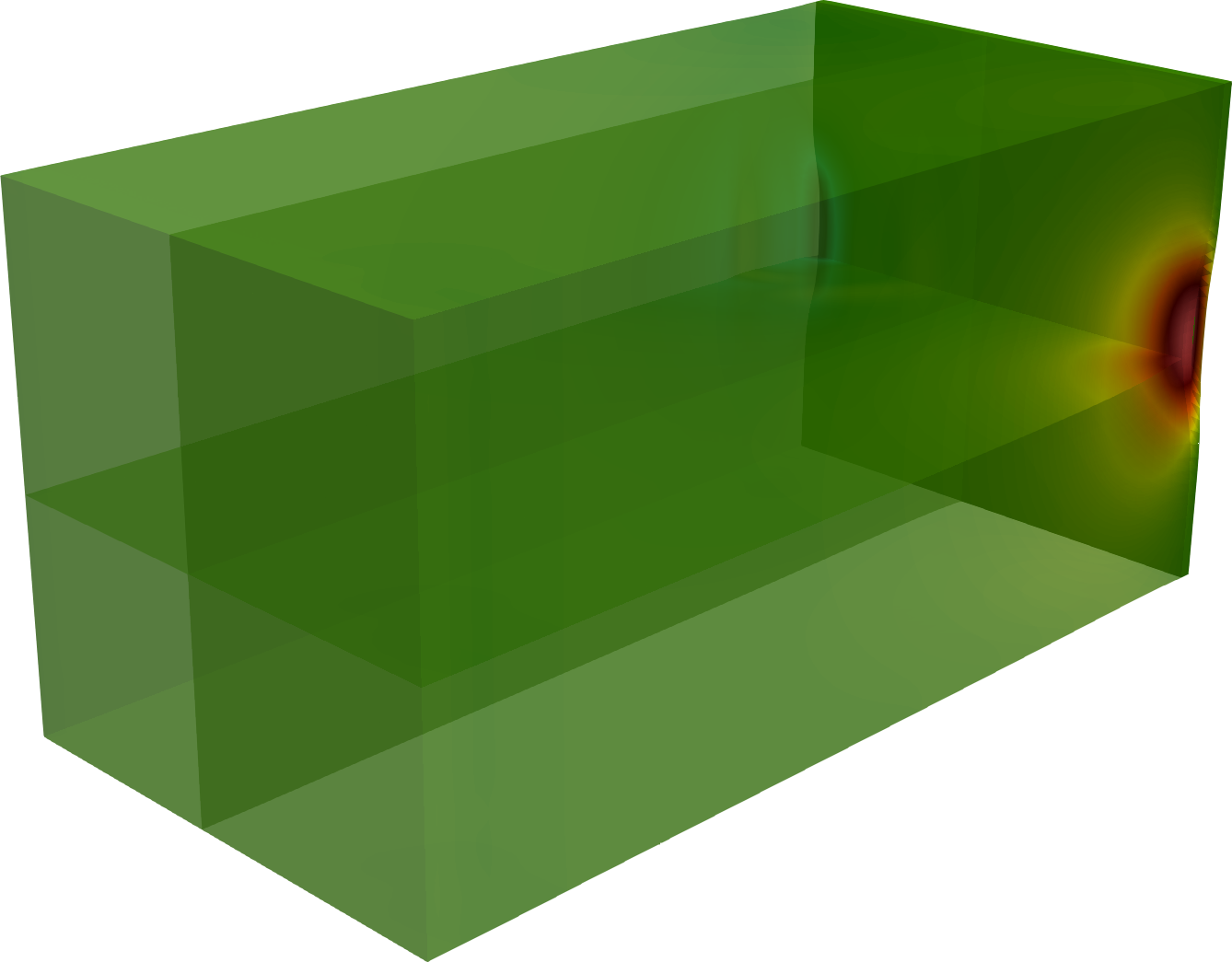}
		\label{fig:uz_FM_H}
	\end{subfigure}
	
	\begin{subfigure}[b]{.45\textwidth}
		\centering
		\includegraphics[width=0.9\textwidth]{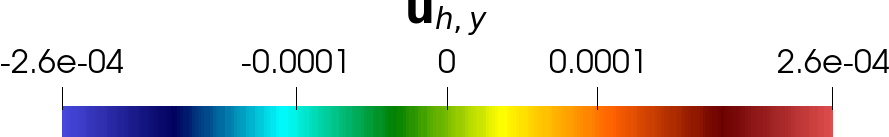}
		\label{fig:uy_FM_H_cb}
	\end{subfigure}
	\begin{subfigure}[b]{.45\textwidth}
		\centering
		\includegraphics[width=0.9\textwidth]{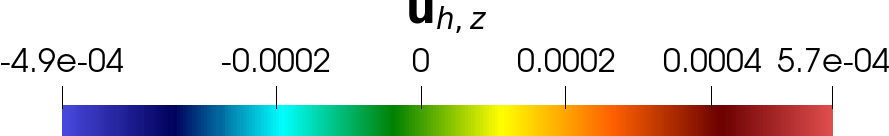}
		\label{fig:uz_FM_H_cb}
	\end{subfigure}
	
	\vspace{0.5cm}
	
	\begin{subfigure}[b]{.45\textwidth}
		\centering
		\includegraphics[width=0.9\textwidth]{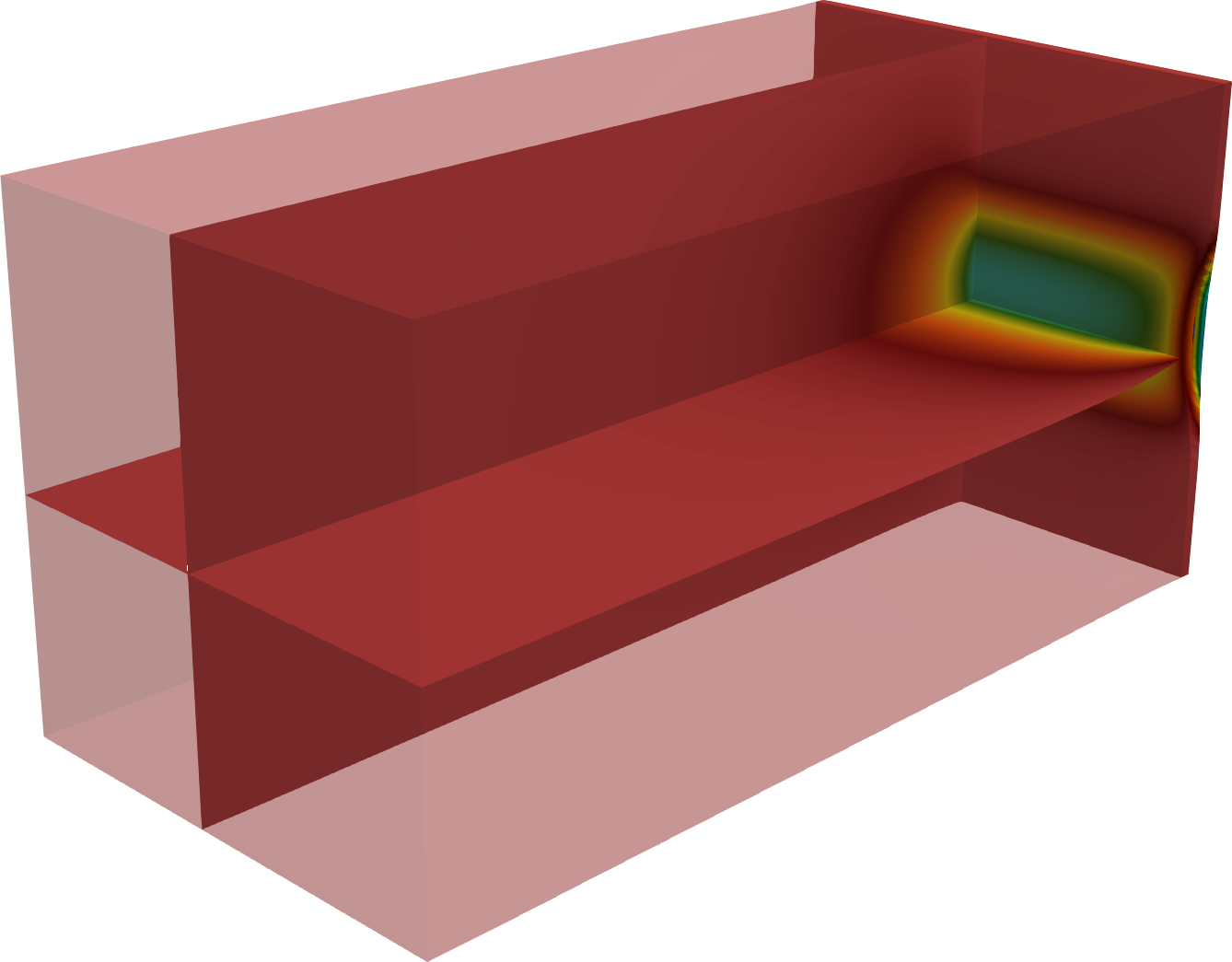}
		\label{fig:T_FM_H}
	\end{subfigure}
	\begin{subfigure}[b]{.45\textwidth}
		\centering
		\includegraphics[width=0.9\textwidth]{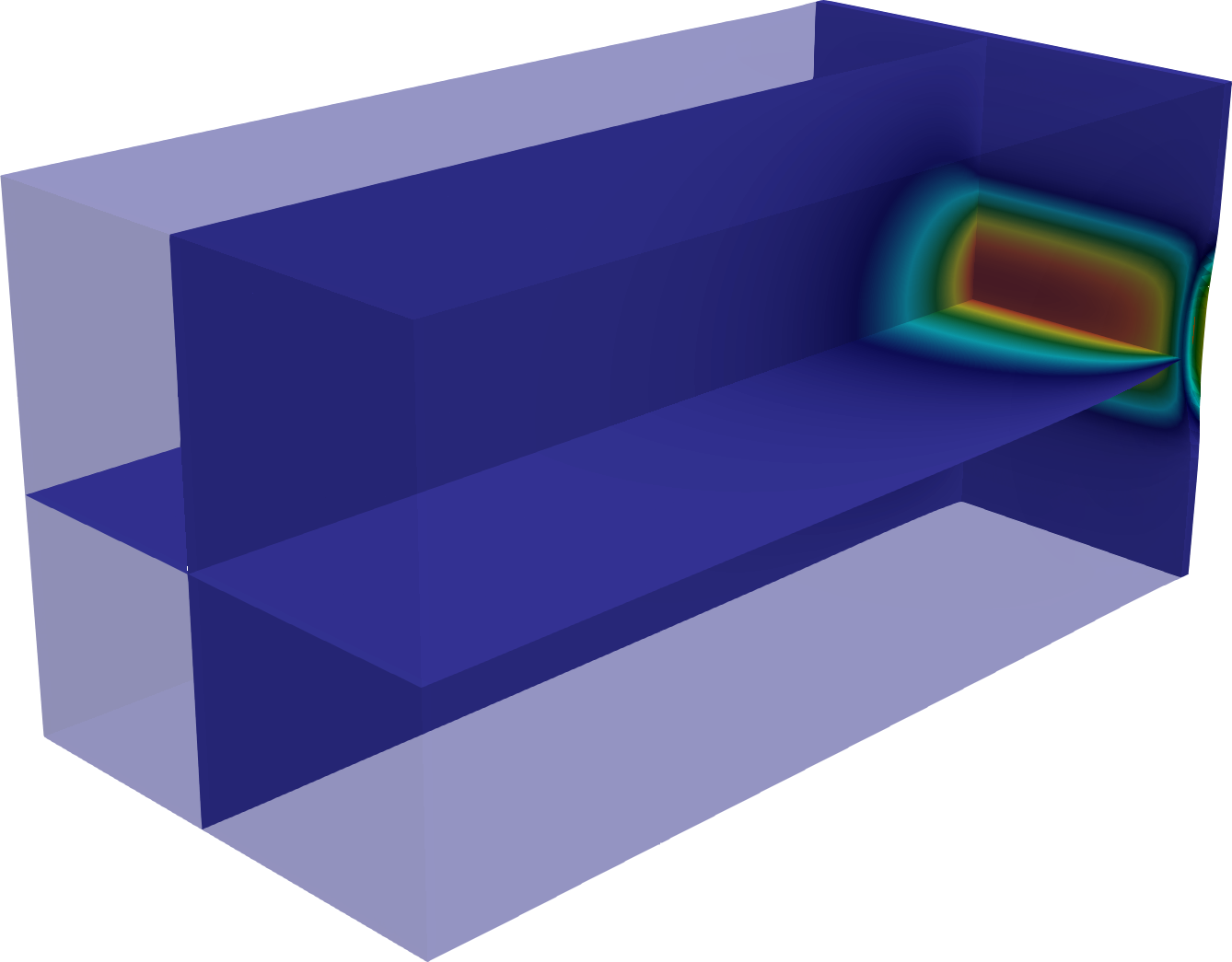}
		\label{fig:p_FM_H}
	\end{subfigure}
	
	\begin{subfigure}[b]{.45\textwidth}
		\centering
		\includegraphics[width=0.9\textwidth]{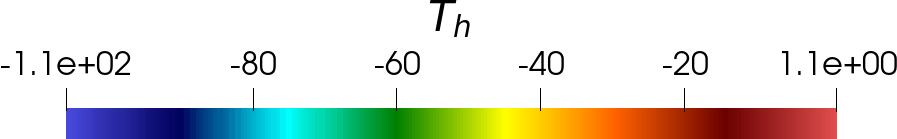}
		\label{fig:T_FM_H_cb}
	\end{subfigure}
	\begin{subfigure}[b]{.45\textwidth}
		\centering
		\includegraphics[width=0.9\textwidth]{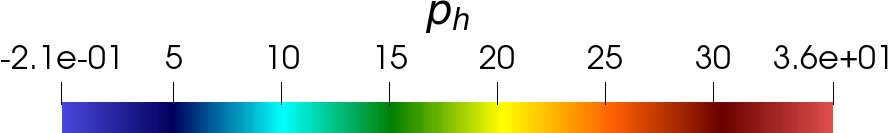}
		\label{fig:p_FM_H_cb}
	\end{subfigure}
	\caption{Test case of Section~\ref{sec:geo3D}: computed displacement field in $y$- (top-left) and $z-$ (top-right) directions, computed temperature field (bottom-left), and pressure-field (bottom-right). The domain is clipped at $x=0.05$ and the two slices are at $y = 0.9$ and $z = 1.1$. The deformation is magnified by a factor $150$. The solution strategy is the \textbf{FM-H} scheme and the solution is computed over $768$ thousands of elements.}
	\label{fig:geo_FM_H_3D}
\end{figure}

\begin{figure}[ht!]
	\centering
	
	\begin{subfigure}[b]{.45\textwidth}
		\centering
		\includegraphics[width=0.9\textwidth]{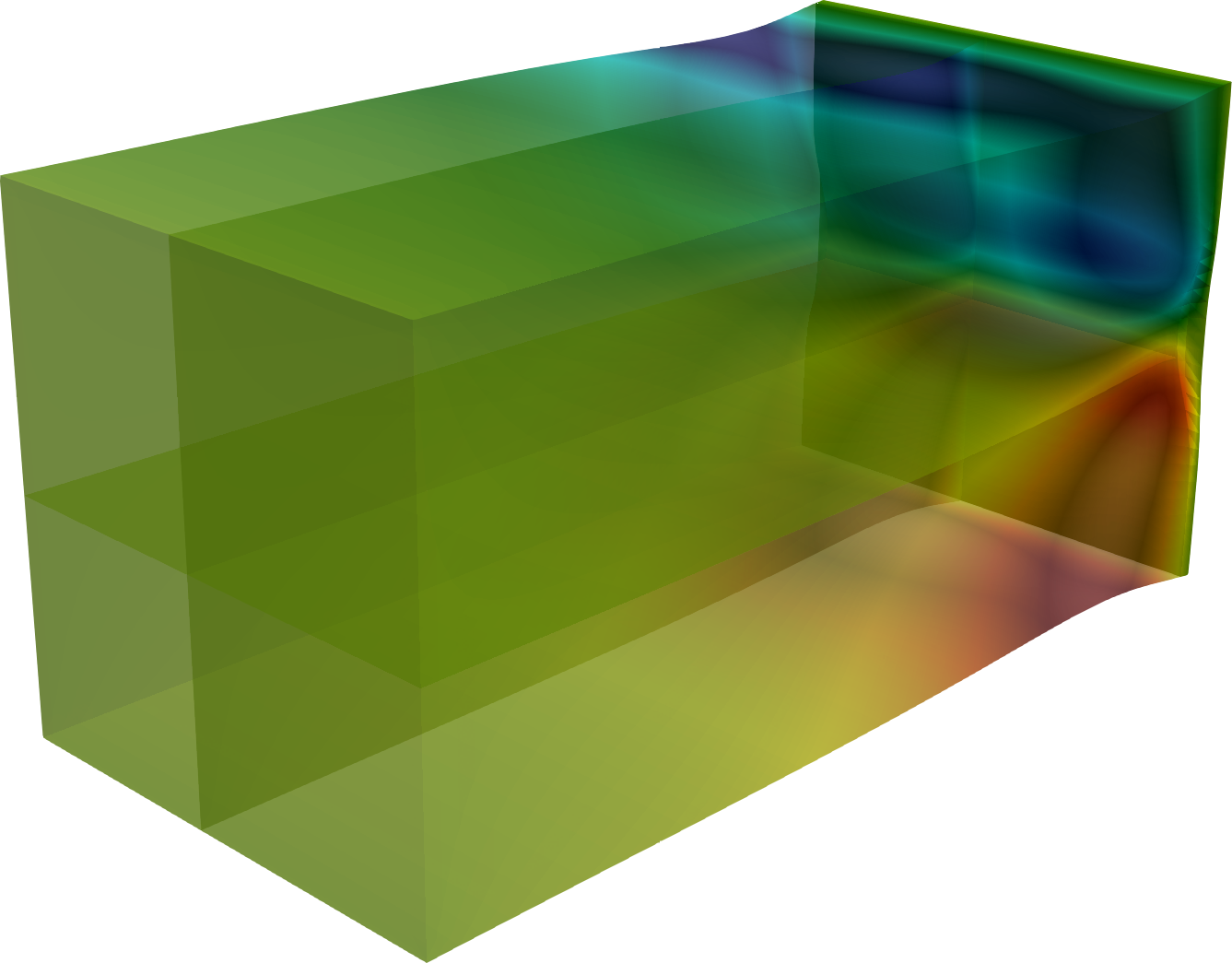}
		\label{fig:uy_fp}
	\end{subfigure}
	\begin{subfigure}[b]{.45\textwidth}
		\centering
		\includegraphics[width=0.9\textwidth]{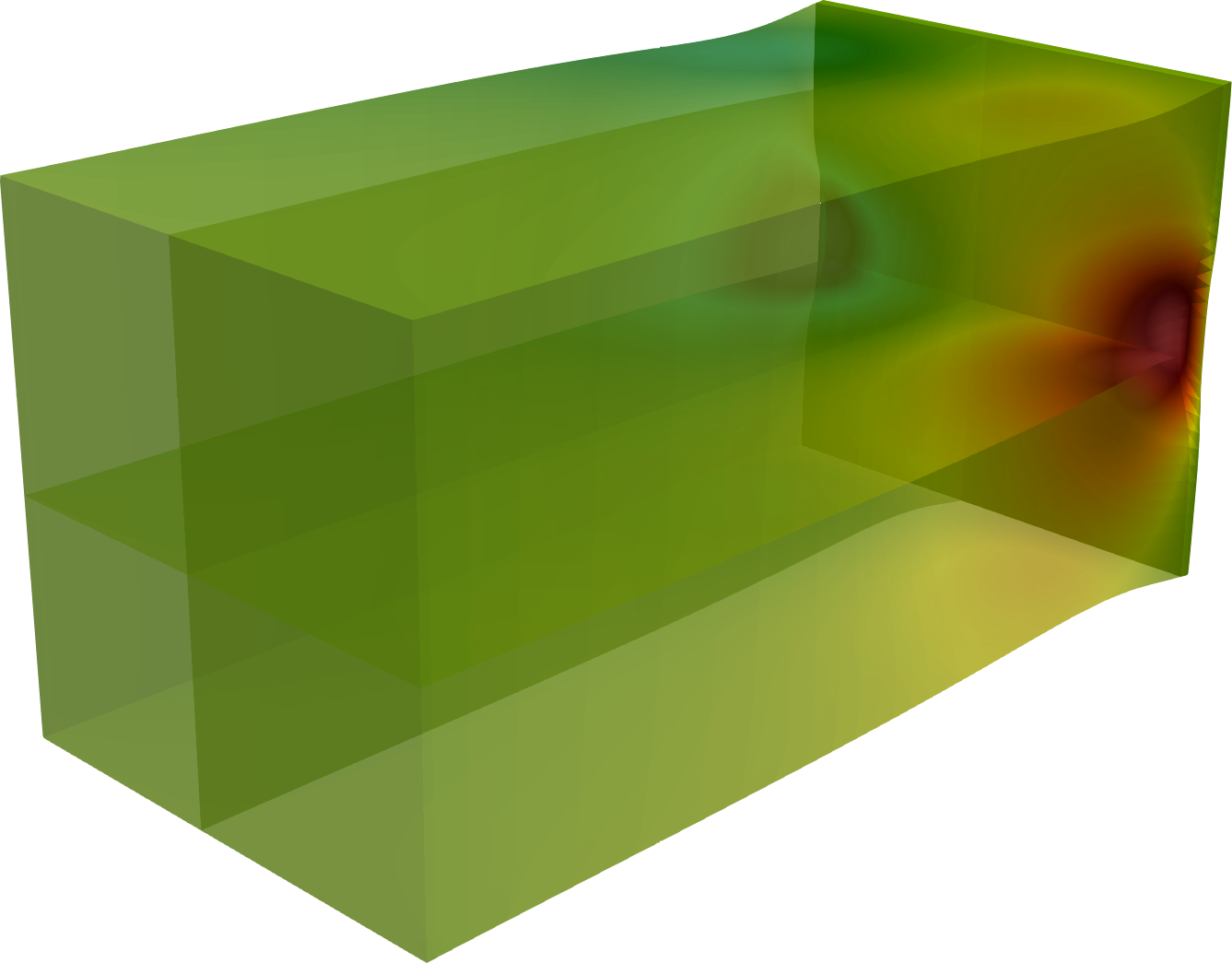}
		\label{fig:uz_fp}
	\end{subfigure}
	
	\begin{subfigure}[b]{.45\textwidth}
		\centering
		\includegraphics[width=0.9\textwidth]{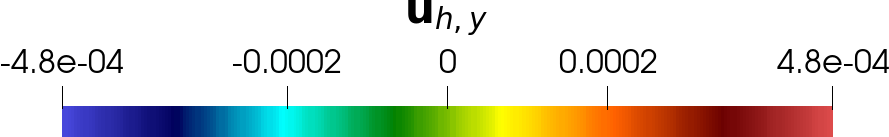}
		\label{fig:uy_fp_cb}
	\end{subfigure}
	\begin{subfigure}[b]{.45\textwidth}
		\centering
		\includegraphics[width=0.9\textwidth]{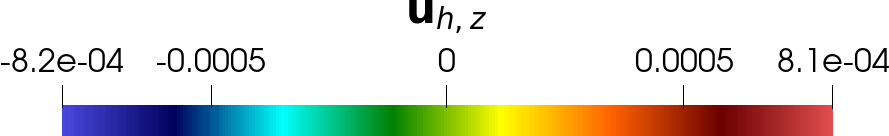}
		\label{fig:uz_fp_cb}
	\end{subfigure}
	
	\vspace{0.5cm}
	
	\begin{subfigure}[b]{.45\textwidth}
		\centering
		\includegraphics[width=0.9\textwidth]{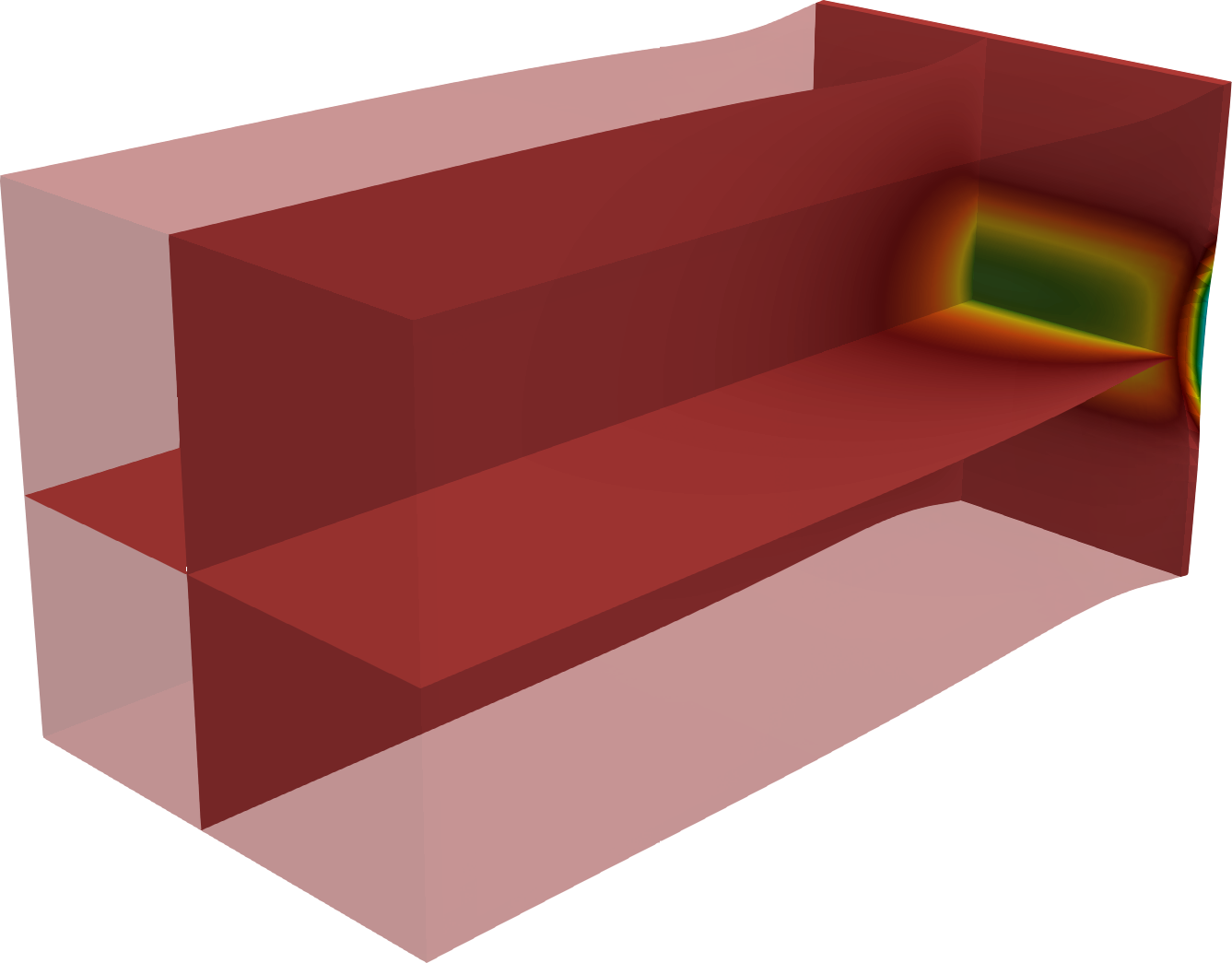}
		\label{fig:T_fp}
	\end{subfigure}
	\begin{subfigure}[b]{.45\textwidth}
		\centering
		\includegraphics[width=0.9\textwidth]{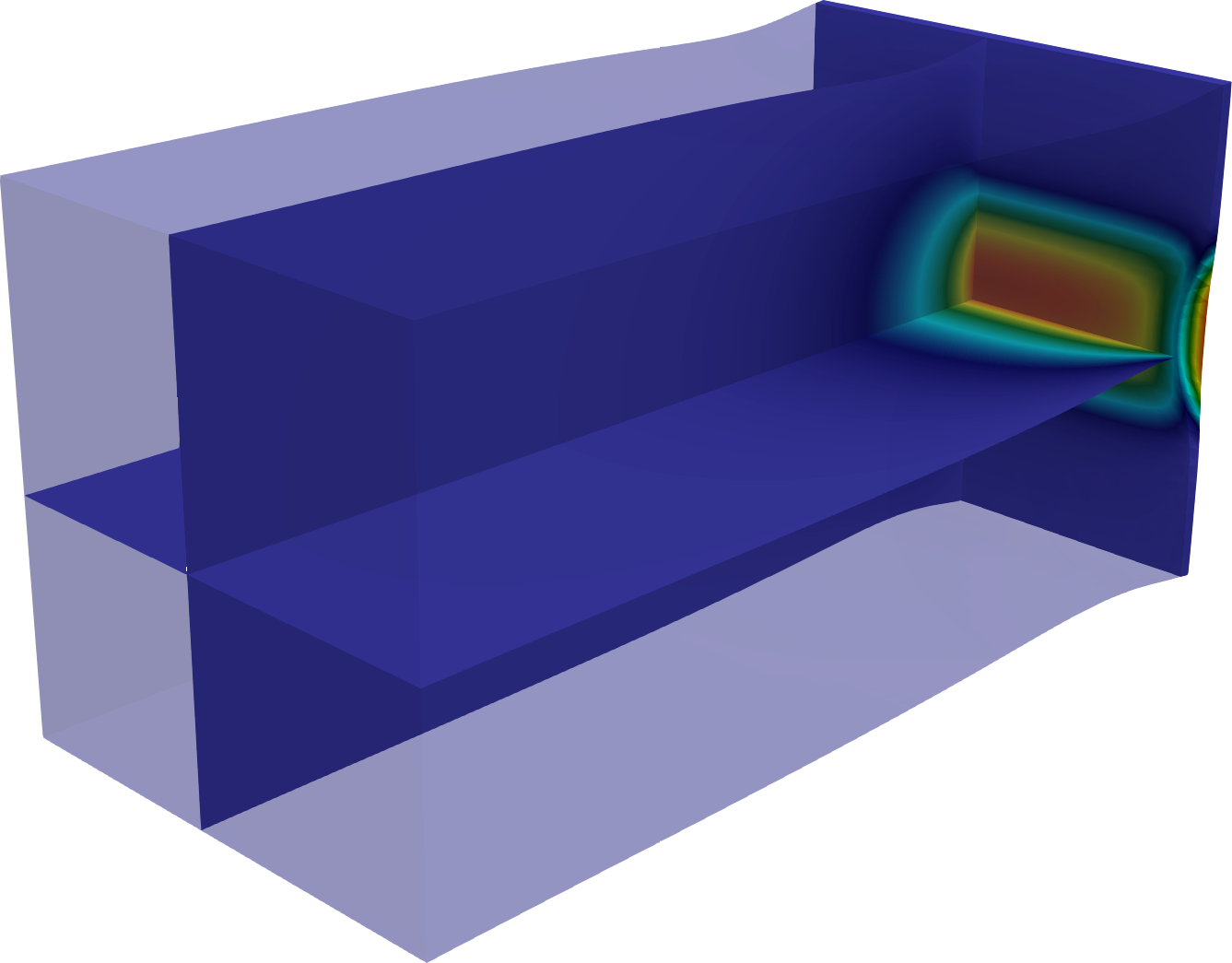}
		\label{fig:p_fp}
	\end{subfigure}
	
	\begin{subfigure}[b]{.45\textwidth}
		\centering
		\includegraphics[width=0.9\textwidth]{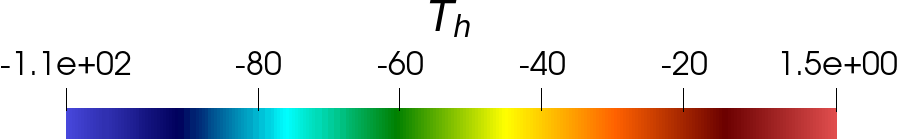}
		\label{fig:T_fp_cb}
	\end{subfigure}
	\begin{subfigure}[b]{.45\textwidth}
		\centering
		\includegraphics[width=0.9\textwidth]{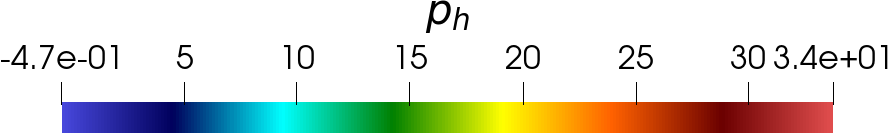}
		\label{fig:p_fp_cb}
	\end{subfigure}
	\caption{Test case of Section~\ref{sec:geo3D}: computed displacement field in $y$- (top-left) and $z-$ (top-right) directions, computed temperature field (bottom-left), and pressure-field (bottom-right). The domain is clipped at $x=0.05$ and the two slices are at $y = 0.9$ and $z = 1.1$. The deformation is magnified by a factor $150$. The solution strategy is the \textbf{FHM} scheme and the solution is computed over $96$ thousands of elements.}
	\label{fig:geo_fp_3D}
\end{figure}

In Figure~\ref{fig:geo_F_H_M_3D} we report the results for the geothermal test case in a three-dimensional setting solved with the \textbf{F-H-M} solution strategy on $1.5$ millions of elements (total number of degrees of freedom $\num[exponent-product=\ensuremath{\cdot}, print-unity-mantissa=false]{3.6e+7}$). First, we observe that the injected fluid is rapidly brought to the reference temperature by the system (cf. Figure~\ref{fig:geo_F_H_M_3D}~(bottom-left)). Second, by looking at the pressure field (cf. Figure~\ref{fig:geo_F_H_M_3D}~(bottom-right)) we observe a zone of high-pressure in the inflow region. Last, by looking at the displacement field, we observe that the injection of cold fluid induces a negative displacement in the $x$-direction in the inflow region and it induces a tighting of the porous media in the $y$- and $z$-direction. This particular phenomena is triggered by the temperature gradient of the fluid and it highlights the importance of taking into account the thermal coupling with the more classical poroelastic model. These results are in agreement with the ones presented in \cite{AntoniettiBonetti2022}. We remark that, for the sake of representation, in Figure~\ref{fig:geo_F_H_M_3D} the deformations of the domain are scaled by an appropriate factor (indicated in the figures caption). No particular differences are observed between the results of the \textbf{FHM} and the \textbf{F-H-M} schemes, consistent with what was observed in the previous Sections. For what concerns the \textbf{FM-H} scheme, we can notice that the simulation captures the most important behaviours of the solution (e.g., the shrinking of the domain in the inflow region), but the results are slightly different with respect to the other two solution strategies. 

\begin{table}[H]
	\centering 
	\footnotesize
	\begin{tabular}{c c c c c c}
		& $N_{\text{el}}$ & $96000$ & $324000$  & $768000$ & $1500000$ \B\\
		\hline
		\textbf{FHM}
		& \#it & $3$ & $3$ & $-$ &  $-$\T\B\\
		\hline
		\multirow{2}{*}{\textbf{FM-H}} 
		& \#it & $5$ & $5$ & $5$ & $-$ \T\\
		& Speed-up & $1.76$ & $1.89$ & $-$ & $-$ \B\\
		\hline
		\multirow{2}{*}{\textbf{F-H-M}} 
		& \#it & $18$ & $18$ & $18$ & $18$ \T\\
		& Speed-up & $0.24$ & $0.28$ & $-$ & $-$ \B\\
		\hline
	\end{tabular}
	\caption{Geothermal test case in three dimensions: iteration counts of the three solution algorithms for the convergence and speed-ups of the splitting solution algorithms with respect to the \textbf{FHM} scheme versus the number of elements. The computational times [\si{\second}] are reported in \ref{sec:appendix}.}
	\label{tab:geo_itertimes}
\end{table}

In Table~\ref{tab:geo_itertimes} we report the iteration counts and the computed speed-ups for the solution of the test case considered different number of elements and different solution strategies. The computational times [\si{\second}] can be found in \ref{sec:appendix}. First, we observe that the iteration counts needed by the schemes for solving the problem is independent by the number of elements of the mesh (moreover, we remark that the tolerances of the fixed-point and of the splitting strategies are independent of the mesh size). Second, we observe that the \textbf{FM-H} solution strategy is the one that performs the best in terms of computational times, while the \textbf{F-H-M} is the one that performs the worst. These results are in agreement with the theoretical results presented in \cite{AntoniettiBonetti2022}, where it is clear that the complete fully-coupling of the four-field PolyDG formulation for this problem can enhance the robustness properties of the method. However, it is important to highlight that, despite being the scheme with the worst performance in terms of computational times, the \textbf{F-H-M} is the only scheme that can be employed for the solution  of the last refinement considered, while the \textbf{FHM} and \textbf{FM-H} shows memory issues after 324000 and 768000 elements, respectively.

\section{Theoretical convergence analysis}
\label{sec:ConvProof}
We recall that the well-posedness of problem \eqref{eq:TPE_system}, \eqref{eq:TPE_discrete_4fields}, and the convergence of the \textbf{FHM} fixed-point strategy have already been addressed in \cite{Brun2019, AntoniettiBonetti2022, Bonetti2024}. The aim of this section is to prove the stability and the convergence of the proposed splitting schemes. More specifically, we focus on the \textbf{F-H-M} scheme, since the convergence of the \textbf{FM-H} scheme follows by applying similar arguments. We remark that the proof of convergence of these schemes can be easily extended to the quasi-static problem, where the splitting iterations are employed at every time-step. Similarly to the results obtained in \cite{Bonetti2024}, for the quasi-static problem the conditions on the the problem parameters for ensuring the convergence of the schemes, can be translated into a condition on the smallness of the time-step.

\subsection{Preliminaries}
For starting the analysis on the convergence of the \textbf{F-H-M} splitting scheme, we introduce some auxiliary lemmata; the proof, when not explicitly stated, can be found in \cite{antonietti2020_stokesDG, AntoniettiBonetti2022, Bonetti2024}.
\begin{lemma}
	\label{lem:boundcoerc_bil_forms}
	Let Assumptions~\ref{assumption:TPE_model_problem} and ~\ref{ass:TPE_mesh_Th1} be satisfied and assume that the parameters $\alpha_1$, $\alpha_2$, and $\alpha_3$ appearing in \eqref{eq:TPE_stabilization_func} are chosen large enough. Then, the following bounds hold:
	\begin{equation}
		\begin{aligned}
			\mathcal{A}_h^T(T,S) \lesssim \ & \|T\|_{dG,T} \|S\|_{dG,T}, \qquad && \mathcal{A}_h^T(T,T) \gtrsim \|T\|_{dG,T}^2 \qquad &&\forall \ T,S \in V_h^{\ell},\\
			\mathcal{A}_h^p(p,q) \lesssim \ & \|p\|_{dG,p} \|q\|_{dG,p}, \qquad && \mathcal{A}_h^p(p,p) \gtrsim \|p\|_{dG,p}^2 \qquad &&\forall \ p,q \in V_h^{\ell},\\
			\mathcal{A}_h^e(\mathbf{u},\mathbf{v}) \lesssim \ & \|\mathbf{u}\|_{dG,e} \|\mathbf{v}\|_{dG,e}, \qquad && \mathcal{A}_h^e(\mathbf{u},\mathbf{u}) \gtrsim \|\mathbf{u}\|_{dG,e}^2 \qquad &&\forall \ \mathbf{u},\mathbf{v} \in \mathbf{V}_h^{\ell},
		\end{aligned}
	\end{equation}
	where the hidden constants do not depend on the material properties and the discretization parameters.
\end{lemma}

\begin{lemma}
	\label{lem:gen_inf_sup}
	Let assume that Assumption~\ref{ass:TPE_mesh_Th1} holds and that the polynomial degrees $\ell$ and $m$  satisfy $\ell+1 \geq m$. Moreover assume that the parameter $\alpha_4$ in \eqref{eq:TPE_stabilization_func} is large enough. Then, following inequality holds true:
	\begin{equation}
		\label{eq:gen_inf_sup}
		\underset{\mathbf{0} \neq \mathbf{v}_h \in \mathbf{V}^{\ell}_h}{\mbox{sup}} \frac{\mathcal{B}_h(\mathbf{v}_h, \varphi_h)}{\|\mathbf{v}_h\|_{DG,e}} + \mathcal{D}_h(\varphi_h,\varphi_h)^{\frac12} \geq \mathbb{B} \|\varphi_h\| \qquad \forall \varphi_h \in Q_h^m;
	\end{equation}
	where $\mathbb{B}>0$ is possibly dependent on $\ell$ and $m$ and the hidden constant is independent of the mesh size $h$.
\end{lemma}

\begin{lemma}
	\label{lem:bound_Ch}
	Given Assumption~\ref{ass:TPE_mesh_Th1}, for all $\mathbf{v}_h \in \mathbf{V}_h^{\ell}$ we define the broken $L^{\infty}$ seminorm as
	$$
	|\mathbf{v}_h|_{dG, \infty} = \| \nabla_h \mathbf{v}_h\|_{L^{\infty}(\Omega)} + \max_{F \in \mathcal{F}} \, \max_{\kappa \in \{\kappa^+, \kappa^-\}} \, \frac{\ell^2}{h_{\kappa}} \| \jump{\mathbf{v}_h}\|_{L^{\infty}(F)}
	$$
	then
	$$
	\mathcal{C}(T_h, p_h, T_h) \gtrsim -\lvert c_f \mathbf{K} \nabla_h p_h \rvert_{dG, \infty} \|T_h\|^2 \quad \forall T_h \in V_h^{\ell}
	$$
\end{lemma}
\begin{lemma}
	\label{lem:aux_stab_bound}
	Let Assumptions~\ref{assumption:TPE_model_problem} and ~\ref{ass:TPE_mesh_Th1} be satisfied. Moreover, assume that $\alpha_4$ in \eqref{eq:TPE_stabilization_func} is taken large enough, and let $\ell+1 \geq m$. Additionally, assume that at least two of the following requirements are verified:
	$$
	(i)\; b_0 \geq b_m > 0 \quad 
	(ii)\; a_0 - b_0 \geq a_m > 0 \quad 
	(iii)\; c_0 - b_0 \geq c_m > 0 \quad 
	(iv)\; \lambda < \lambda_M < \infty
	$$
	Then, there exist strictly positive constants $a_1$, $b_1$, and $c_1$ such that
	$$
	a_1 \|T_h\|^2 + b_1 \|\varphi_h\|^2 + c_1 \|p_h\|^2 \lesssim \mathcal{M}((p_h, T_h, \varphi_h), (p_h, T_h, \varphi_h)) + \mathcal{D}_h(\varphi_h, \varphi_h) + \| \mathbf{u}_h\|_{dG,e}^2 + \| \mathbf{f}\|^2.
	$$
\end{lemma}

\subsection{Proof of Theorem~\ref{thm:stability}}
In the following, we report the proof of Theorem~\ref{thm:stability}, that investigates the stability of the \textbf{F-H-M} problem.

\begin{proof}
	We start the proof by focusing on the fluid equation Algorithm~\hyperref[eq:F_H_M_splitting_1]{3--\textit{Step1}}. To this aim, we take $(\mathbf{v}_h, q_h, S_h, \psi_h) = (\mathbf{0}, p_h^{k+1}, 0, 0)$ as test function. This choice leads to:
	\begin{equation}
		\label{eq:stab_step_1}
		\begin{aligned}
			& \mathcal{M}_{p}(p_h^{k+1},p_h^{k+1}) + \mathcal{A}_h^{p}(p_h^{k+1},p_h^{k+1}) = (g,p_h^{k+1}) - \mathcal{M}_{pT}(T_h^k, p_h^{k+1}) - \mathcal{M}_{p\varphi}(p_h^{k+1}, \varphi_h^{k}),  
		\end{aligned}
	\end{equation}
	by using Lemma~\ref{lem:boundcoerc_bil_forms} and Young's inequality we obtain
	\begin{equation}
		\label{eq:stab_step_2}
		\begin{aligned}
			& \frac{c_{\alpha}}{2} \| p_h^{k+1} \|^2 + 
			\| p_h^{k+1} \|_{dG,p}^2 \lesssim \frac{1}{2 c_{\alpha}} \left\|g - b_{\alpha\beta}\, T^k - \frac{\alpha}{\lambda}\varphi^k\right\|^2,
		\end{aligned}
	\end{equation}
	from which we can infer a bound for the $L^2$-norm of the pressure, that reads:
	\begin{equation}
		\label{eq:stab_step_3}
		\begin{aligned}
			& \| p_h^{k+1} \|^2 \lesssim \frac{1}{c_{\alpha}^2} \left\|g - b_{\alpha\beta} \, T^k - \frac{\alpha}{\lambda}\varphi^k\right\|^2.
		\end{aligned}
	\end{equation}
	We now take $(\mathbf{v}_h, q_h, S_h, \psi_h) = (\mathbf{0}, 0, T_h^{k+1}, 0)$ as test function in Algorithm~\hyperref[eq:F_H_M_splitting_1]{3--\textit{Step2}}:
	\begin{equation}
		\label{eq:stab_step_4}
		\begin{aligned}
			& \mathcal{M}_{T}(T_h^{k+1}, T_h^{k+1}) + \mathcal{A}_h^{T}(T_h^{k+1},S_h) + \mathcal{C}_h(T_h^{k+1},p_h^{k+1},S_h) = (H, S_h) - \mathcal{M}_{pT}(S_h, p_h^{k+1}) - \mathcal{M}_{T\varphi}(S_h, \varphi_h^{k}),
		\end{aligned}
	\end{equation}
	By using Lemma~\ref{lem:boundcoerc_bil_forms}, Lemma~\ref{lem:bound_Ch}, Young inequality, and triangle inequality, we get:
	\begin{equation}
		\label{eq:stab_step_5}
		\begin{aligned}
			& \frac{a_{\beta}-a_1}{2} \| T_h^{k+1} \|^2 + 
			\| T_h^{k+1} \|_{dG,T}^2 \lesssim \frac{1}{2(a_{\beta} - a_1)} \left\|H - \frac{\beta}{\lambda}\varphi^k\right\|^2 + \frac{b_{\alpha\beta}}{2(a_{\beta} - a_1)} \left\| p^{k+1} \right\|^2,
		\end{aligned}
	\end{equation}
	and by using \eqref{eq:stab_step_3}, we obtain:
	\begin{equation}
		\label{eq:stab_step_6}
		\begin{aligned}
			\frac{a_{\beta}-a_1}{2} \| T_h^{k+1} \|^2 + 
			\| T_h^{k+1} \|_{dG,T}^2 \lesssim & \frac{1}{2(a_{\beta} - a_1)}\left( \left\|H - \frac{\beta}{\lambda}\varphi^k\right\|^2 + \frac{b_{\alpha\beta}}{c_{\alpha}^2} \left\| g - b_{\alpha\beta} \, T^k - \frac{\alpha}{\lambda} \varphi^k \right\|^2 \right),
		\end{aligned}
	\end{equation}
	that yields:
	\begin{equation}
		\label{eq:stab_step_7}
		\begin{aligned}
			\| T_h^{k+1} \|^2 \lesssim & \frac{1}{(a_{\beta} - a_1)^2} \left\|H - \frac{\beta}{\lambda}\varphi^k\right\|^2 + \frac{b_{\alpha\beta}}{c_{\alpha}^2 \, (a_{\beta} - a_1)^2} \left\| g - b_{\alpha\beta} \, T^k - \frac{\alpha}{\lambda} \varphi^k \right\|^2.
		\end{aligned}
	\end{equation}
	In the last step of the proof we derive a control for the displacement and the pseudo-total pressure. Taking $(\mathbf{v}_h, q_h, S_h, \psi_h) = (\mathbf{v}_h, 0, 0, 0)$ as test function in Algorithm~\hyperref[eq:F_H_M_splitting_1]{3--\textit{Step3}} we obtain:
	\begin{equation}
		\label{eq:stab_step_8}
		\begin{aligned}
			&  \mathcal{B}_h(\varphi_h^{k+1},\mathbf{v}_h) = \mathcal{A}_h^{e}(\mathbf{u}_h^{k+1},\mathbf{v}_h) -\left(\mathbf{f},\mathbf{v}_h\right).
		\end{aligned}
	\end{equation}
	Plugging \eqref{eq:stab_step_8} into Lemma~\ref{lem:gen_inf_sup}, using Lemma~\ref{lem:boundcoerc_bil_forms}, and using the discrete Poincaré-Korn inequality \cite{Botti2025}, we get \cite{Bonetti2024}:
	\begin{equation}
		\label{eq:stab_step_9}
		\begin{aligned}
			&  \mathbb{B}^2 \|\varphi_h^{k+1}\|^2 \lesssim \mathcal{D}(\varphi_h^{k+1}, \varphi_h^{k+1}) + \|\mathbf{u}_h^{k+1}\|_{dG,e}^2 + \mu_m^{-1}\|\mathbf{f}\|^2. 
		\end{aligned}
	\end{equation}
	We now consider $(\mathbf{v}_h, q_h, S_h, \psi_h) = (\mathbf{u}_h^{k+1}, 0, 0, \varphi_h^{k+1})$ as test function in Algorithm~\hyperref[eq:F_H_M_splitting_1]{3--\textit{Step3}}:
	\begin{equation}
		\label{eq:stab_step_10}
		\begin{aligned}
			& \mathcal{M}_{\varphi}(\varphi_h^{k+1},\varphi_h^{k+1}) + \mathcal{A}_h^{e}(\mathbf{u}_h^{k+1},\mathbf{u}_h^{k+1}) +  \mathcal{D}_h(\varphi_h^{k+1},\varphi_h^{k+1}) \\
			& = \left(\mathbf{f},\mathbf{u}_h^{k+1}\right) - \mathcal{M}_{p\varphi}(p_h^{k+1}, \varphi_h^{k+1}) - \mathcal{M}_{T\varphi}(T_h^{k+1}, \varphi_h^{k+1}).
		\end{aligned}
	\end{equation}
	We observe that, thanks to \eqref{eq:gen_inf_sup}, we have:
	\begin{equation}
		\label{eq:stab_step_11}
		\begin{aligned}
			&  \|\mathbf{u}_h^{k+1}\|_{dG,e}^2 +  \mathcal{D}_h(\varphi_h^{k+1},\varphi_h^{k+1}) \gtrsim \frac{\mathbb{B}}{4} \|\varphi_h^{k+1}\|^2 - \frac{\mu_m^{-1}}{4}\|\mathbf{f}\|^2 + \frac{3}{4} \|\mathbf{u}_h^{k+1}\|_{dG,e}^2
		\end{aligned}
	\end{equation}
	that leads to:
	\begin{equation}
		\label{eq:stab_step_12}
		\begin{aligned}
			& \left(\frac{1}{2\lambda} + \frac{\mathbb{B}}{4}\right) \|\varphi_h^{k+1}\|^2 + \frac{1}{2} \|\mathbf{u}_h^{k+1}\|_{dG,e}^2 \lesssim \left(\frac{1}{C_K} + \frac{\mu_m^{-1}}{4}\right) \|\mathbf{f}\|^2 + \frac{\alpha^2}{\lambda} \|p_h^{k+1}\|^2 + \frac{\beta^2}{\lambda} \|T_h^{k+1}\|^2.
		\end{aligned}
	\end{equation}
	By plugging \eqref{eq:stab_step_3}, \eqref{eq:stab_step_7} into \eqref{eq:stab_step_12}, we get obtain our result.
\end{proof}

\subsection{Proof of Theorem~\ref{thm:convergence}}
We conclude the section by proving the convergence of the splitting solution strategy, cf. Theorem~\ref{thm:convergence}. Our aim is to show that the difference of approximations at two successive iterations defines is a contracting sequence. We start by deriving the error equations. Let $(\mathbf{u}_h^{k+1}, p_h^{k+1}, T_h^{k+1}, \varphi_h^{k+1})$ and $(\mathbf{u}_h^k, p_h^k, T_h^k, \varphi_h^k)$ be the solutions to \textbf{F-H-M} at the $(k+1)^{\text{th}}$ and $k^{\text{th}}$ iterations, respectively. For all $k \geq 1$, we define:
$$
\boldsymbol{\delta}_{\mathbf{u}}^k = \mathbf{u}_h^{k+1} - \mathbf{u}_h^{k}, \quad \delta_{p}^k = p_h^{k+1} - p_h^{k}, \quad \delta_{T}^k = T_h^{k+1} - T_h^{k}, \quad \delta_{\varphi}^k = \varphi_h^{k+1} - \varphi_h^{k}. 
$$
Then, it can be observed that $(\boldsymbol{\delta}_{\mathbf{u}}^k, \delta_p^k, \delta_T^k, \delta_{\varphi}^k) \in \mathbf{V}_h^{\ell} \times V_h^{\ell} \times V_h^{\ell} \times Q_h^m$ solves the problem:
\begin{equation}
	\label{eq:error_eq_1}
	\begin{aligned}
		& \mathcal{M}_{p}(\delta_p^{k},q_h) + \mathcal{A}_h^{p}(\delta_p^{k},q_h) + \mathcal{M}_{T}(\delta_T^{k}, S_h) + \mathcal{A}_h^{T}(\delta_T^{k},S_h) + \mathcal{C}_h(T_h^{k+1},p_h^{k+1},S_h) - \mathcal{C}_h(T_h^{k},p_h^{k},S_h) \\
		& + \mathcal{M}_{\varphi}(\delta_{\varphi}^{k},\psi_h) + \mathcal{A}_h^{e}(\boldsymbol{\delta}_{\mathbf{u}}^{k},\mathbf{v}_h) - \mathcal{B}_h(\delta_{\varphi}^{k},\mathbf{v}_h) +  \mathcal{B}_h(\psi_h, \boldsymbol{\delta}_{\mathbf{u}}^{k}) +  \mathcal{D}_h(\delta_{\varphi}^{k},\psi_h) = - \mathcal{M}_{pT}(\delta_T^{k-1}, q_h) \\
		& - \mathcal{M}_{p\varphi}(q_h, \delta_{\varphi}^{k-1}) - \mathcal{M}_{pT}(S_h, \delta_p^{k}) - \mathcal{M}_{T\varphi}(S_h, \delta_{\varphi}^{k-1}) - \mathcal{M}_{p\varphi}(\delta_p^{k}, \psi_h) - \mathcal{M}_{T\varphi}(\delta_T^{k},  \psi_h).
	\end{aligned}
\end{equation}
Now, we add and subtract the terms $\pm \mathcal{C}_h(T_h^k, p_h^{k+1}, S_h)$ and we obtain
\begin{equation}
	\label{eq:error_eq_2}
	\begin{aligned}
		& \mathcal{M}_{p}(\delta_p^{k},q_h) + \mathcal{M}_{T}(\delta_T^{k}, S_h) + \mathcal{M}_{\varphi}(\delta_{\varphi}^{k},\psi_h) + \mathcal{M}_{pT}(S_h, \delta_p^{k}) + \mathcal{M}_{p\varphi}(\delta_p^{k}, \psi_h) + \mathcal{M}_{T\varphi}(\delta_T^{k},  \psi_h) \\
		& + \mathcal{A}_h^{p}(\delta_p^{k},q_h) + \mathcal{A}_h^{T}(\delta_T^{k},S_h) + \mathcal{A}_h^{e}(\boldsymbol{\delta}_{\mathbf{u}}^{k},\mathbf{v}_h) + \mathcal{C}_h(\delta_T^{k},p_h^{k+1},S_h) - \mathcal{C}_h(T_h^{k},\delta_p^{k},S_h)   - \mathcal{B}_h(\delta_{\varphi}^{k},\mathbf{v}_h) \\
		& +  \mathcal{B}_h(\psi_h, \boldsymbol{\delta}_{\mathbf{u}}^{k}) +  \mathcal{D}_h(\delta_{\varphi}^{k},\psi_h) = - \mathcal{M}_{pT}(\delta_T^{k-1}, q_h) - \mathcal{M}_{p\varphi}(q_h, \delta_{\varphi}^{k-1}) - \mathcal{M}_{T\varphi}(S_h, \delta_{\varphi}^{k-1}).
	\end{aligned}
\end{equation}
We can now proceed with the proof of Theorem~\ref{thm:convergence}.
\begin{proof}
	We start the proof by setting $(\mathbf{v}_h, q_h, S_h, \psi_h) = (\boldsymbol{\delta}_{\mathbf{u}}^k, 0, 0, 0)$ in \eqref{eq:error_eq_2} for bounding $\delta_{\varphi}^k$, i.e.,
	\begin{equation}
		\label{eq:error_eq_3}
		\begin{aligned}
			& \mathcal{A}_h^{e}(\boldsymbol{\delta}_{\mathbf{u}}^{k}, \boldsymbol{\delta}_{\mathbf{u}}^{k}) - \mathcal{B}_h(\delta_{\varphi}^{k},\boldsymbol{\delta}_{\mathbf{u}}^{k}) = 0.
		\end{aligned}
	\end{equation}
	As in the proof of Theorem~\ref{thm:stability}, this leads to:
	\begin{equation}
		\label{eq:error_eq_4}
		\begin{aligned}
			& \mathbb{B}^2 \|\delta_{\varphi}^k\|^2 \lesssim \mathcal{D}_h(\delta_{\varphi}^k, \delta_{\varphi}^k) + \|\boldsymbol{\delta}_{\mathbf{u}}^k\|_{dG,e}^2.
		\end{aligned}
	\end{equation}
	Then, we can use Lemma~\ref{lem:bound_Ch} to obtain
	$
	-\frac{a_1}{2}\|\delta_T^k\|^2 \lesssim \mathcal{C}_h(\delta_T^k, p_h^k, \delta_T^k),
	$
	from which we can infer:
	\begin{equation}
		\label{eq:error_eq_6}
		\begin{aligned}
			\frac{a_1}{2} \|\delta_T^k\|^2 + b_1 \|\delta_{\varphi}^k\|^2 + c_1 \|\delta_p^k\|^2 + \|\boldsymbol{\delta}_{\mathbf{u}}^k\|_{dG,e}^2 + \|\delta_T^k\|_{dG,T}^2 + \|\delta_p^k\|_{dG,p}^2
			\lesssim a_1 \|\delta_T^k\|^2 + b_1 \|\delta_{\varphi}^k\|^2 \\
			+ c_1 \|\delta_p^k\|^2 + \|\boldsymbol{\delta}_{\mathbf{u}}^k\|_{dG,e}^2 + \|\delta_T^k\|_{dG,T}^2 + \|\delta_p^k\|_{dG,p}^2 + \mathcal{C}_h(\delta_T^k, p_h^k, \delta_T^k),
		\end{aligned}
	\end{equation}
	by summing the same terms at left and right hand side. We can rearrange \eqref{eq:error_eq_6} to obtain:
	\begin{equation}
		\label{eq:error_eq_7}
		\begin{aligned}
			\frac{a_1}{2} \|\delta_T^k\|^2 + b_1 \|\delta_{\varphi}^k\|^2 + c_1 \|\delta_p^k\|^2 + \|\boldsymbol{\delta}_{\mathbf{u}}^k\|_{dG,e}^2 + \|\delta_T^k\|_{dG,T}^2 + \|\delta_p^k\|_{dG,p}^2
			\lesssim \mathcal{M}((\delta_p^k, \delta_T^k, \delta_{\varphi}^k), (\delta_p^k, \delta_T^k, \delta_{\varphi}^k)) \\
			+ \mathcal{D}_h(\delta_{\varphi}^k, \delta_{\varphi}^k) + \mathcal{A}_h^e(\boldsymbol{\delta}_{\mathbf{u}}^k,\boldsymbol{\delta}_{\mathbf{u}}^k) + \mathcal{A}_h^T(\delta_T^k, \delta_T^k) + \mathcal{A}_h^p(\delta_p^k, \delta_p^k) + \mathcal{C}_h(\delta_T^k, p_h^k, \delta_T^k),
		\end{aligned}
	\end{equation}
	We now set $(\mathbf{v}_h, q_h, S_h, \psi_h) = (\boldsymbol{\delta}_{\mathbf{u}}^k, \delta_p^k, \delta_T^k, \delta_{\varphi}^k)$ in \eqref{eq:error_eq_2}, to obtain:
	\begin{equation}
		\label{eq:error_eq_8}
		\begin{aligned}
			& \mathcal{M}_{p}(\delta_p^{k},\delta_p^{k}) + \mathcal{M}_{T}(\delta_T^{k}, \delta_T^{k}) + \mathcal{M}_{\varphi}(\delta_{\varphi}^{k}, \delta_{\varphi}^{k}) + \mathcal{M}_{pT}(\delta_T^{k}, \delta_p^{k}) + \mathcal{M}_{p\varphi}(\delta_p^{k}, \delta_{\varphi}^{k}) + \mathcal{M}_{T\varphi}(\delta_T^{k}, \delta_{\varphi}^{k}) \\
			& + \mathcal{A}_h^{p}(\delta_p^{k},\delta_p^{k})  + \mathcal{A}_h^{T}(\delta_T^{k},\delta_T^{k}) + \mathcal{A}_h^{e}(\boldsymbol{\delta}_{\mathbf{u}}^{k}, \boldsymbol{\delta}_{\mathbf{u}}^{k}) + \mathcal{C}_h(\delta_T^{k},p_h^{k+1},\delta_T^{k}) - \mathcal{C}_h(T_h^{k},\delta_p^{k},\delta_T^{k}) +  \mathcal{D}_h(\delta_{\varphi}^{k}, \delta_{\varphi}^{k}) = \\
			& - \mathcal{M}_{pT}(\delta_T^{k-1}, \delta_p^{k}) - \mathcal{M}_{p\varphi}(\delta_p^{k}, \delta_{\varphi}^{k-1}) - \mathcal{M}_{T\varphi}(\delta_T^{k}, \delta_{\varphi}^{k-1}).
		\end{aligned}
	\end{equation}
	We can now sum and subtract $\mathcal{M}_{pT}(\delta_T^{k}, \delta_p^{k})$, $\mathcal{M}_{p\varphi}(\delta_p^{k}, \delta_{\varphi}^{k})$, $\mathcal{M}_{T\varphi}(\delta_T^{k}, \delta_{\varphi}^{k})$ in \eqref{eq:error_eq_8} to get:
	\begin{equation}
		\label{eq:error_eq_9}
		\begin{aligned}
			& \mathcal{M}((\delta_p^{k},\delta_T^{k},\delta_{\varphi}^{k}),(\delta_p^{k},\delta_T^{k},\delta_{\varphi}^{k})) + \mathcal{A}_h^{p}(\delta_p^{k},\delta_p^{k})  + \mathcal{A}_h^{T}(\delta_T^{k},\delta_T^{k}) + \mathcal{A}_h^{e}(\boldsymbol{\delta}_{\mathbf{u}}^{k}, \boldsymbol{\delta}_{\mathbf{u}}^{k}) + \mathcal{C}_h(\delta_T^{k},p_h^{k+1},\delta_T^{k}) +  \mathcal{D}_h(\delta_{\varphi}^{k}, \delta_{\varphi}^{k}) \\
			& = \mathcal{C}_h(T_h^{k},\delta_p^{k},\delta_T^{k}) - \mathcal{M}_{pT}(\delta_T^{k} - \delta_T^{k-1}, \delta_p^{k}) - \mathcal{M}_{p\varphi}(\delta_p^{k}, \delta_{\varphi}^{k} - \delta_{\varphi}^{k-1}) - \mathcal{M}_{T\varphi}(\delta_T^{k}, \delta_{\varphi}^{k} - \delta_{\varphi}^{k-1}).
		\end{aligned}
	\end{equation}
	that combined with \eqref{eq:error_eq_7} leads to:
	\begin{equation}
		\label{eq:error_eq_10}
		\begin{aligned}
			& \frac{a_1}{2} \|\delta_T^k\|^2 + b_1 \|\delta_{\varphi}^k\|^2 + c_1 \|\delta_p^k\|^2 + \|\boldsymbol{\delta}_{\mathbf{u}}^k\|_{dG,e}^2 + \|\delta_T^k\|_{dG,T}^2 + \|\delta_p^k\|_{dG,p}^2
			\lesssim \\
			& \mathcal{C}_h(T_h^{k},\delta_p^{k},\delta_T^{k}) - \mathcal{M}_{pT}(\delta_T^{k} - \delta_T^{k-1}, \delta_p^{k}) - \mathcal{M}_{p\varphi}(\delta_p^{k}, \delta_{\varphi}^{k} - \delta_{\varphi}^{k-1}) - \mathcal{M}_{T\varphi}(\delta_T^{k}, \delta_{\varphi}^{k} - \delta_{\varphi}^{k-1}).
		\end{aligned}
	\end{equation}
	We are now left to bound the right hand side of \eqref{eq:error_eq_10}. We start by the trilinear form, that can be treated as in \cite{Bonetti2024}, while the bilinear forms can be bounded as follows:
	\begin{equation}
		\label{eq:error_eq_11}
		(a^{k+1} - a^k, b) \leq \epsilon \| b \|^2 + \frac{1}{2\epsilon}\|a^{k+1}\|^2 + \frac{1}{2\epsilon}\|a^k\|^2,
	\end{equation}
	where we have used the Cauchy-Schwartz's, the Young's, and parallelogram inequalities. The terms that appear at right hand side of \eqref{eq:error_eq_10} are then bounded as follows:
	\begin{equation}
		\label{eq:error_eq_12}
		\begin{aligned}
			\lvert\, \mathcal{C}_h(T_h^{k},\delta_p^{k},\delta_T^{k}) \,\rvert & \lesssim \frac{a_1}{2 \epsilon_c} \| \delta_T^k \|^2 + \frac{\epsilon_c \, c_{f_M}^2 (1 + C_{tr}^4)}{a_1}\lvert T_h^k \rvert_{dG,\infty}^2 \|\delta_p^k\|^2, \\
			\lvert\, -\mathcal{M}_{pT}(\delta_T^{k}-\delta_T^{k-1}, \delta_p^{k}) \,\rvert & \leq \epsilon_{pT}\|\delta_p^{k}\|^2 + \frac{b_{\alpha\beta}^2}{2\epsilon_{pT}}\|\delta_T^{k}\|^2 + \frac{b_{\alpha\beta}^2}{2\epsilon_{pT}}\|\delta_T^{k-1}\|^2, \\
			\lvert\, -\mathcal{M}_{p\varphi}(\delta_p^{k}, \delta_{\varphi}^{k}-\delta_{\varphi}^{k-1}) \,\rvert & \leq \epsilon_{p\varphi}\|\delta_p^{k}\|^2 + \frac{\alpha^2}{2\lambda^2\epsilon_{p\varphi}}\|\delta_{\varphi}^{k}\|^2 + \frac{\alpha^2}{2\lambda^2\epsilon_{p\varphi}}\|\delta_{\varphi}^{k-1}\|^2,\\
			\lvert\, -\mathcal{M}_{T\varphi}(\delta_T^{k}, \delta_{\varphi}^{k}-\delta_{\varphi}^{k-1}) \,\rvert & \leq \epsilon_{T\varphi}\|\delta_T^{k}\|^2 + \frac{\beta^2}{2\lambda^2\epsilon_{T\varphi}}\|\delta_{\varphi}^{k}\|^2 + \frac{\beta^2}{2\lambda^2\epsilon_{T\varphi}}\|\delta_{\varphi}^{k-1}\|^2
		\end{aligned}
	\end{equation}
	By plugging \eqref{eq:error_eq_12} into \eqref{eq:error_eq_10} we obtain:
	\begin{equation}
		\label{eq:error_eq_13}
		\begin{aligned}
			& \frac{a_1}{2} \|\delta_T^k\|^2 + b_1 \|\delta_{\varphi}^k\|^2 + c_1 \|\delta_p^k\|^2 + \|\boldsymbol{\delta}_{\mathbf{u}}^k\|_{dG,e}^2 + \|\delta_T^k\|_{dG,T}^2 + \|\delta_p^k\|_{dG,p}^2
			\lesssim \frac{a_1}{2 \epsilon_c} \| \delta_T^k \|^2 \\
			& + \frac{\epsilon_c \, c_{f_M}^2 (1 + C_{tr}^4)}{a_1}\lvert T_h^k \rvert_{dG,\infty}^2 \|\delta_p^k\|^2
			+ \epsilon_{pT}\|\delta_p^{k}\|^2 + \frac{b_{\alpha\beta}^2}{2\epsilon_{pT}}\|\delta_T^{k}\|^2 + \frac{b_{\alpha\beta}^2}{2\epsilon_{pT}}\|\delta_T^{k-1}\|^2
			+ \epsilon_{p\varphi}\|\delta_p^{k}\|^2 \\
			& + \frac{\alpha^2}{2\lambda^2\epsilon_{p\varphi}}\|\delta_{\varphi}^{k}\|^2 + \frac{\alpha^2}{2\lambda^2\epsilon_{p\varphi}}\|\delta_{\varphi}^{k-1}\|^2
			+ \epsilon_{T\varphi}\|\delta_T^{k}\|^2 + \frac{\beta^2}{2\lambda^2\epsilon_{T\varphi}}\|\delta_{\varphi}^{k}\|^2 + \frac{\beta^2}{2\lambda^2\epsilon_{T\varphi}}\|\delta_{\varphi}^{k-1}\|^2. 
		\end{aligned}
	\end{equation}
	We choose $\epsilon_c = 2$, $\epsilon_{pT} = b_{\alpha\beta}^2/2$, $\epsilon_{p\varphi} = \alpha^2/\lambda^2$, and $\epsilon_{T\varphi} = \beta^2/\lambda^2$, to obtain:
	\begin{equation}
		\label{eq:error_eq_14}
		\begin{aligned}
			& \left( \frac{a_1}{4} - \frac{\beta^2}{\lambda^2} - 1 \right) \|\delta_T^k\|^2 + \left( b_1-1 \right) \|\delta_{\varphi}^k\|^2 + \left( c_1 - \frac{b_{\alpha\beta}^2}{2} - \frac{\alpha^2}{\lambda^2} - \frac{2 \, c_{f_M}^2 (1 + C_{tr}^4)}{a_1}\lvert T_h^k \rvert_{dG,\infty}^2 \right) \|\delta_p^k\|^2 \\
			& + \|\boldsymbol{\delta}_{\mathbf{u}}^k\|_{dG,e}^2 + \|\delta_T^k\|_{dG,T}^2 + \|\delta_p^k\|_{dG,p}^2
			\lesssim \|\delta_T^{k-1}\|^2 + \|\delta_{\varphi}^{k-1}\|^2
		\end{aligned}
	\end{equation}
	Under Assumptions~\ref{eq:convergece_ass}, the map $(\boldsymbol{\delta}_{\mathbf{u}}^{k-1}, \delta_p^{k-1}, \delta_T^{k-1}, \delta_{\varphi}^{k-1}) \rightarrow (\boldsymbol{\delta}_{\mathbf{u}}^k, \delta_p^k, \delta_T^k, \delta_{\varphi}^k)$ is a contraction. Then, the conclusion follows by applying the Banach fixed-point theorem.
\end{proof}

\section{Conclusions}
In this work we have presented two different splitted iterative strategy for the solution of  a four-field PolyDG-WSIP space discretization of the non-linear fully-coupled thermo-hydro-mechanical problem. The stability estimate and the convergence of the iteration schemes are presented. Numerical simulations are performed to assess the convergence and robustness properties of the method and to test the applicability of such strategies for real problems' simulations. 

Further developments of the present work are possible. In particular, we mention an in-depth investigation of the preconditioning of subproblems arising from the splitting schemes, but also the preconditioning of the whole system. Moreover, an extension to unsteady problems, both in the quasi-static and dynamic regimes, is of interest. Finally, the extension to other non-linear coupled models and to viscous media can be relevant both in terms of real-life applications and numerical analysis. 

\section*{Funding}
This work received funding from the European Union (ERC SyG, NEMESIS, project number 101115663). Views and opinions expressed are, however, those of the authors only and do not necessarily reflect those of the European Union or the European Research Council Executive Agency. Neither the European Union nor the granting authority can be held responsible for them. PFA, SB, and MB are members of INdAM-GNCS. MB kindly acknowledge partial financial support by INdAM-GNCS project 2025 CUP E53C24001950001. The present research is part of the activities of “Dipartimento di Eccellenza 2023-2027”, funded by MUR, Italy.

\bibliography{Sections/bibliography}

\appendix
\section{Computational times}
\label{sec:appendix}
The aim of this appendix is to show the computational times for the test cases of Section~\ref{sec:NumericalResults}, where only the speed-ups have been reported. All the numerical simulations presented in Section~\ref{sec:NumericalResults}, Section~\ref{sec:geo3D} were performed using the \texttt{NEMESIS} cluster (2 CPU AMD EPYC 9634 84-Core Processor (336 threads), 1.5 TB di RAM -- Resource for sequential applications) at MOX, Department of Mathematics, Politecnico di Milano.

\begin{table}[H]
	\centering 
	\footnotesize
	\begin{tabular}{c c c c c c c c c c}
		& $N_{\text{el}}$ & $50$ & $100$  & $310$ & $1000$  & $3100$ & $10000$  & $31000$ & $100000$ \B\\
		\hline
		\multirow{2}{*}{\textbf{FHM}} 
		& \#it & $4$ & $4$ & $4$ & $4$ & $4$ & $4$ & $4$ & $4$ \T\\
		& CPU [\si{\second}] & $0.457$ & $0.905$ & $3.372$ & $13.698$ & $59.116$ & $286.246$ & $1709.860$ & $12415.693$\B\\
		\hline
		\multirow{3}{*}{\textbf{FM-H}} 
		& \#it & $4$ & $4$ & $4$ & $4$ & $4$ & $4$ & $4$ & $4$ \T\\
		& Speed-up & 1.19 & 1.31 & 1.49 & 1.73 & 2.13 & 2.58 & 3.30 & 4.83 \T\B\\
		& CPU [\si{\second}] & $0.383$ & $0.691$ & $2.259$ & $7.910$ & $27.756$ & $111.158$ & $518.363$ & $2572.766$ \B\\
		\hline
		\multirow{3}{*}{\textbf{F-H-M}} 
		& \#it & $3$ & $3$ & $3$ & $3$ & $3$ & $3$ & $3$ & $3$ \T\\
		& Speed-up & 0.46 & 1.31 & 1.59 & 1.91 & 2.29 & 2.86 & 3.69 & 4.93 \T\B\\ 
		& CPU[\si{\second}] & $0.996$ & $0.691$ & $2.115$ & $7.184$ & $25.849$ & $100.247$ & $463.334$ & $2517.394$ \B\\
		\hline
	\end{tabular}
	\caption{Convergence test vs $h$ in 3D: number of iterations of the three solution algorithms for the convergence, speed-up of the splitting schemes with respect to the monolithic approach, and computational times [\si{\second}] with respect to $N_{\text{el}}$.}
	\label{tab:ConvH_itertimes_appendix}
\end{table}

\begin{table}[H]
	\centering 
	\footnotesize
	\begin{tabular}{c c c c c c c c c c}
		& $\ell$ & $1$  & $2$ & $3$  & $4$ & $5$  & $6$ & $7$ & $8$ \B\\
		\hline
		\multirow{2}{*}{\textbf{FHM}} 
		& \#it & $4$ & $4$ & $4$ & $4$ & $4$ & $4$ & $4$ & $4$ \T\\
		& CPU[\si{\second}] & $1.028$ & $0.488$ & $0.696$ & $1.163$ & $2.146$ & $3.474$ & $5.905$ & $10.817$ \B\\
		\hline
		\multirow{3}{*}{\textbf{FM-H}} 
		& \#it  & $4$ & $4$ & $4$ & $4$ & $4$ & $4$ & $4$ & $4$ \T\\
		& Speed-up & 3.19 & 1.45 & 1.49 & 1.59 & 1.80 & 1.82 & 1.87 & 1.96 \T\B \\
		& CPU[\si{\second}] & $0.323$ & $0.338$ & $0.467$ & $0.732$ & $1.192$ & $1.908$ & $3.150$ & $5.521$ \B\\
		\hline
		\multirow{3}{*}{\textbf{F-H-M}} 
		& \#it  & $4$ & $4$ & $4$ & $4$ & $4$ & $4$ & $4$ & $4$ \T\\
		& Speed-up & 3.49 & 1.41 & 1.37 & 1.44 & 1.58 & 1.51 & 1.17 & 1.13 \T\B \\
		& CPU[\si{\second}] & $0.295$ & $0.346$ & $0.507$ & $0.805$ & $1.359$ & $2.296$ & $5.048$ & $9.558$ \B \\
		\hline
	\end{tabular}
	\caption{Convergence test vs $\ell$ in 2D: number of iterations of the three solution algorithms for the convergence, speed-up of the splitting schemes with respect to the monolithic approach, and computational times [\si{\second}] with respect to $N_{\text{el}}$.}
	\label{tab:ConvP_itertimes_appendix}
\end{table}

\begin{table}[H]
	\centering 
	\footnotesize
	\begin{tabular}{c c c c c c c c c c}
		& $N_{\text{el}}$ & $48$ & $1296$  & $6000$ & $16464$  & $34992$ & $63888$  & $105456$ & $162000$ \B\\
		\hline
		\multirow{2}{*}{\textbf{FHM}} 
		& \#it & $3$ & $3$ & $3$ & $4$ & $4$ & $4$ & $4$ & $4$ \T\\
		& CPU [\si{\second}] & $0.395$ & $11.802$ & $58.601$ & $232.643$ & $537.915$ & $1071.775$ & $2087.690$ & $4798.707$\B\\
		\hline
		\multirow{3}{*}{\textbf{FM-H}} 
		& \#it & $3$ & $3$ & $3$ & $3$ & $3$ & $3$ & $3$ & $3$ \T\\
		& Speed-up & $1.05$ & $1.39$ & $1.39$ & $1.69$ & $1.73$ & $1.73$ & $1.64$ & $1.60$ \T\B\\
		& CPU [\si{\second}] & $0.375$ & $8.480$ & $42.094$ & $137.795$ & $310.286$ & $619.494$ & $1270.269$ & $2999.107$ \B\\    
		\hline
		\multirow{3}{*}{\textbf{F-H-M}} 
		& \#it & $3$ & $3$ & $3$ & $3$ & $3$ & $3$ & $3$ & $3$ \T\\
		& Speed-up & $1.00$ & $1.82$ & $1.91$ & $2.38$ & $2.48$ & $2.41$ & $2.56$ & $3.07$ \T\B\\
		& CPU[\si{\second}] & $0.394$ & $6.496$ & $30.675$ & $97.904$ & $217.334$ & $444.432$ & $815.054$ & $1562.066$ \B\\
		\hline
	\end{tabular}
	\caption{Convergence test vs $h$ in 3D: number of iterations of the three solution algorithms for the convergence, speed-up of the splitting schemes with respect to the monolithic approach, and computational times [\si{\second}] with respect to $N_{\text{el}}$.}
	\label{tab:ConvH3D_itertimes_appendix}
\end{table}

\begin{table}[H]
	\centering 
	\footnotesize
	\begin{tabular}{c c c c c c c c c c}
		& $N_{\text{el}}$ & $50$ & $100$  & $310$ & $1000$  & $3100$ & $10000$  & $31000$ & $100000$ \B\\
		\hline
		\multirow{2}{*}{\textbf{FHM}} 
		& \#it & $5$ & $3$ & $3$ & $3$ & $3$ & $3$ & $3$ & $3$ \T\\
		& CPU [\si{\second}] & $0.528$ & $0.761$ & $2.791$ & $11.073$ & $47.140$ & $223.988$ & $1321.004$ & $9426.582$ \B\\
		\hline
		\multirow{3}{*}{\textbf{FM-H}} 
		& \#it & $3$ & $3$ & $3$ & $3$ & $3$ & $3$ & $3$ & $3$ \T\\
		& Speed-up & 1.59 & 1.30 & 1.45 & 1.64 & 1.98 & 2.34 & 2.91 & 4.12  \T\B\\
		& CPU [\si{\second}] & $0.332$ & $0.587$ & $1.920$ & $6.742$ & $23.757$ & $95.810$ & $454.086$ & $2288.489$ \B\\
		\hline
		\multirow{3}{*}{\textbf{F-H-M}} 
		& \#it & $2$ & $2$ & $2$ & $2$ & $2$ & $2$ & $2$ & $2$ \T\\
		& Speed-up & 0.57 & 1.39 & 1.66 & 2.00 & 2.40 & 2.99 & 3.91 & 5.23 \T\B\\
		& CPU[\si{\second}] & $0.931$ & $0.549$ & $1.680$ & $5.526$ & $19.638$ & $74.890$ & $338.267$ & $1801.705$ \B\\
		\hline
	\end{tabular}
	\caption{Robustness test $\mathbf{(i)}$: number of iterations of the three solution algorithms for the convergence, speed-up of the splitting schemes with respect to the monolithic approach, and computational times [\si{\second}] with respect to $N_{\text{el}}$.}
	\label{tab:Rob1_itertimes_appendix}
\end{table}

\begin{table}[H]
	\centering 
	\footnotesize
	\begin{tabular}{c c c c c c c c c c}
		& $N_{\text{el}}$ & $50$ & $100$  & $310$ & $1000$  & $3100$ & $10000$  & $31000$ & $100000$ \B\\
		\hline
		\multirow{2}{*}{\textbf{FHM}} 
		& \#it & $2$ & $2$ & $2$ & $2$ & $2$ & $2$ & $2$ & $2$ \T\\
		& CPU [\si{\second}] & $0.321$ & $0.609$ & $2.185$ & $8.419$ & $34.761$ & $161.515$ & $920.434$ & $7564.624$ \B\\
		\hline
		\multirow{3}{*}{\textbf{FM-H}} 
		& \#it & $4$ & $4$ & $4$ & $4$ & $4$ & $4$ & $4$ & $4$ \T\\
		& Speed-up & 0.83 & 0.88 & 0.97 & 1.07 & 1.24 & 1.46 & 1.79 & 2.86 \T\B\\
		& CPU [\si{\second}] & $0.386$ & $0.694$ & $2.262$ & $7.897$ & $27.936$ & $110.904$ & $515.060$ & $2648.855$ \B\\
		\hline
		\multirow{3}{*}{\textbf{F-H-M}} 
		& \#it & $5$ & $5$ & $5$ & $5$ & $5$ & $5$ & $5$ & $5$ \T\\
		& Speed-up & 0.28 & 0.67 & 0.73 & 0.80 & 0.92 & 1.0656 & 1.29 & 1.93 \T\B\\
		& CPU[\si{\second}] & $1.163$ & $0.913$ & $2.981$ & $10.460$ & $37.843$ & $151.573$ & $711.494$ & $3916.191$ \B\\
		\hline
	\end{tabular}
	\caption{Robustness test $\mathbf{(ii)}$: number of iterations of the three solution algorithms for the convergence, speed-up of the splitting schemes with respect to the monolithic approach, and computational times [\si{\second}] with respect to $N_{\text{el}}$.}
	\label{tab:Rob2_itertimes_appendix}
\end{table}

\begin{table}[H]
	\centering 
	\footnotesize
	\begin{tabular}{c c c c c c c c c c}
		& $N_{\text{el}}$ & $50$ & $100$  & $310$ & $1000$  & $3100$ & $10000$  & $31000$ & $100000$ \B\\
		\hline
		\multirow{2}{*}{\textbf{FHM}} 
		& \#it & $5$ & $5$ & $5$ & $5$ & $5$ & $5$ & $5$ & $5$ \T\\
		& CPU [\si{\second}] & $0.516$ & $1.045$ & $3.973$ & $16.362$ & $72.099$ & $352.188$ & $2109.475$ & $16006.520$ \B\\
		\hline
		\multirow{3}{*}{\textbf{FM-H}} 
		& \#it & $5$ & $5$ & $5$ & $5$ & $5$ & $5$ & $5$ & $5$ \T\\
		& Speed-up & 1.17 & 1.33 & 1.53 & 1.81 & 2.27 & 2.82 & 3.67 & 5.59 \T\B\\
		& CPU [\si{\second}] & $0.443$ & $0.784$ & $2.599$ & $9.055$ & $31.829$ & $124.787$ & $574.928$ & $2865.303$ \B\\
		\hline
		\multirow{3}{*}{\textbf{F-H-M}} 
		& \#it & $5$ & $4$ & $5$ & $5$ & $5$ & $5$ & $5$ & $5$ \T\\
		& Speed-up & 0.44 & 1.33 & 1.33 & 1.56 & 1.88 & 2.32 & 2.94 & 4.06 \T\B\\
		& CPU[\si{\second}] & $1.181$ & $0.786$ & $2.996$ & $10.498$ & $38.272$ & $152.071$ & $716.399$ & $3946.521$ \B\\
		\hline
	\end{tabular}
	\caption{Robustness test $\mathbf{(iii)}$: number of iterations of the three solution algorithms for the convergence, speed-up of the splitting schemes with respect to the monolithic approach, and computational times [\si{\second}] with respect to $N_{\text{el}}$.}
	\label{tab:Rob3_itertimes_appendix}
\end{table}

\begin{table}[H]
	\centering 
	\footnotesize
	\begin{tabular}{c c c c c c c c c c}
		& $N_{\text{el}}$ & $50$ & $100$  & $310$ & $1000$  & $3100$ & $10000$  & $31000$ & $100000$ \B\\
		\hline
		\multirow{2}{*}{\textbf{FHM}} 
		& \#it & $2$ & $2$ & $2$ & $2$ & $2$ & $2$ & $2$ & $-$\T\\
		& CPU [\si{\second}] & 0.314 & 0.599 & 2.153 & 8.322 & 34.188 & 159.203 & 974.375  & $-$ \B\\
		\hline
		\multirow{3}{*}{\textbf{FM-H}} 
		& \#it & $3$ & $3$ & $3$ & $3$ & $3$ & $3$ & $3$ & $3$ \T\\
		& Speed-up & 0.48 & 0.49 & 0.54 & 0.60 & 0.71 & 0.86 & 1.21 & $-$ \T\B\\
		& CPU [\si{\second}] & 0.653 & 1.223 & 4.020 & 13.983 & 48.146 & 185.011 & 803.727 & 4048.949 \B\\
		\hline
		\multirow{3}{*}{\textbf{F-H-M}} 
		& \#it & $2$ & $2$ & $2$ & $2$ & $2$ & $2$ & $2$ & $2$ \T\\
		& Speed-up & 0.21 & 0.39 & 0.42 & 0.45 & 0.50 & 0.57 & 0.74 & $-$ \T\B\\
		& CPU[\si{\second}] & 1.469 & 1.520 & 5.186 & 18.682 & 68.628 & 278.019 & 1324.331 & 7494.267 \B\\
		\hline
	\end{tabular}
	\caption{Robustness test $\mathbf{(iv)}$: number of iterations of the three solution algorithms for the convergence, speed-up of the splitting schemes with respect to the monolithic approach, and computational times [\si{\second}] with respect to $N_{\text{el}}$.}
	\label{tab:Rob4_itertimes_appendix}
\end{table}

\begin{table}[H]
	\centering 
	\footnotesize
	\begin{tabular}{c c c c c c}
		& $N_{\text{el}}$ & $96000$ & $324000$  & $768000$ & $1500000$ \B\\
		\hline
		\multirow{2}{*}{\textbf{FHM}} 
		& \#it & $3$ & $3$ & $-$ &  $-$\T\\
		& CPU [\si{\second}] & $880.48$ & $3307.83$ & $-$ &  $-$ \B\\
		\hline
		\multirow{3}{*}{\textbf{FM-H}} 
		& \#it & $5$ & $5$ & $5$ & $-$ \T\\
		& Speed-up & $1.76$ & $1.89$ & $-$ & $-$ \T\B\\
		& CPU [\si{\second}] & $500.95$ & $ 1746.36$ & $ 4437.31$ & $-$ \B\\
		\hline
		\multirow{3}{*}{\textbf{F-H-M}} 
		& \#it & $18$ & $18$ & $18$ & $18$ \T\\
		& Speed-up & $0.24$ & $0.28$ & $-$ & $-$ \T\B\\
		& CPU [\si{\second}] & $3642.24$ & $11713.36$ & $32132.50$ & $64021.62$ \B\\
		\hline
	\end{tabular}
	\caption{Geothermal test case in three dimensions: number of iterations of the three solution algorithms for the convergence and computational times [\si{\second}] of the three solutions algorithms with respect to the number of elements.}
	\label{tab:geo_itertimes_appendix}
\end{table}

\end{document}

%% file: Fig/Splitting_F_H_M.tikz
\tikzstyle{Fprob} = [rectangle, rounded corners, minimum width=2cm, minimum height=1cm,text centered, draw=black, fill=myblue!30]

\tikzstyle{Hprob} = [rectangle, rounded corners, minimum width=2cm, minimum height=1cm,text centered, draw=black, fill=myred!30]

\tikzstyle{Mprob} = [rectangle, rounded corners, minimum width=2cm, minimum height=1cm,text centered, draw=black, fill=myyellow!20]

\tikzstyle{Others} = [rectangle, rounded corners, minimum width=2cm, minimum height=0.6cm,text centered, draw=black, fill=mygreen!20]

\tikzstyle{Others_2} = [rectangle, rounded corners, minimum width=3.8cm, minimum height=0.6cm,text centered, draw=black, fill=mygreen!20]

\tikzstyle{Question} = [rectangle, rounded corners, minimum width=2cm, minimum height=0.6cm,text centered, draw=black, fill=white!20]

\tikzstyle{arrow} = [thick,->,>=stealth]

\begin{tikzpicture}[node distance=2cm]
\centering

\footnotesize

\node (Initial2) [Question, align = center, yshift = +0.4cm] {At iteration $k \geq 0$, \\ given $(\mathbf{u}^k, p^k, T^k, \varphi^k)$};

\node (Fprob) [Fprob, below of=Initial2, align = center, yshift = +0.4cm] {Solve \textbf{F} \\ $\rightarrow p^{k+1}$};

\node (Hprob) [Hprob, right of = Fprob, align = center, xshift = 1.7cm] {Solve \textbf{H} \\ $\rightarrow T^{k+1}$};

\node (Mprob) [Mprob, right of = Hprob, align = center, xshift = 1.7cm] {Solve \textbf{M} \\ $\rightarrow \mathbf{u}^{k+1}, \varphi^{k+1}$};

\node (Cond) [Question, above of = Mprob, yshift = -0.4cm] {Check convergence};

\node (Fine) [Others_2, right of = Cond, align = center, xshift = 2.5cm] {$(\mathbf{u}, p, T, \varphi)$};

\draw [arrow] (Initial2) -- (Fprob);

\draw [arrow] (Fprob) -- node[anchor=north, align = center, yshift= -0.45cm] {\scriptsize$(\mathbf{u}^k, \color{black!100}{p^{k+1}} \color{black}{,T^k, \varphi^k)}$}(Hprob);

\draw [arrow] (Hprob) -- node[anchor=north, align = center, yshift= -0.45cm] {\scriptsize $(\mathbf{u}^k, \color{black!100}{p^{k+1}, T^{k+1}} \color{black}{,\varphi^k)}$}(Mprob);

\draw [arrow] (Mprob) -- (Cond);

\draw [arrow] (Cond) -- node[anchor=south, align = center] {Yes}(Fine);
 
\draw [arrow] (Cond) -- node[anchor=south, align = center] {No $\rightarrow k = k+1$}(Initial2);

\end{tikzpicture}

%% file: Fig/ConvH_2D.tikz
\begin{center}
\begin{figure}[ht]
\begin{subfigure}[b]{1\textwidth}
\begin{subfigure}[b]{0.4\textwidth}
\begin{tikzpicture}
\begin{axis}[%
width=5cm,
height=3.8cm,
at={(0\textwidth,0\textwidth)},
scale only axis,
xmode=log,
xmin=3,
xmax=250,
xminorticks=true,
xlabel={\footnotesize $1/h$},
ymode=log,
ymin=3e-9,
ymax=0.06,
yminorticks=true,
ylabel={\footnotesize $L^2$-errors},
legend style={draw=none,fill=none,legend cell align=left},
legend pos=outer north east,
]

\addplot [color=myred,solid,line width=0.75pt, mark size=3pt,mark=+, mark options={color=myred}]
  table[row sep=crcr]{
4.0832   0.034068 \\
5.7043   0.0084521 \\
8.7622   0.0011048 \\
16.4893   0.00015153 \\
28.7396   2.4993e-05 \\
50.7883   4.1286e-06 \\
90.6731   7.3439e-07 \\
153.862   1.2595e-07 \\
};

\addplot [color=myred,solid,line width=0.75pt, mark size=3pt,mark=o, mark options={color=myred}]
  table[row sep=crcr]{
4.0832   0.034068 \\
5.7043   0.0084521 \\
8.7622   0.0011048 \\
16.4893   0.00015153 \\
28.7396   2.4993e-05 \\
50.7883   4.1286e-06 \\
90.6731   7.3439e-07 \\
153.862   1.2595e-07 \\
};

\addplot [color=myred,solid,line width=0.75pt, mark size=3pt,mark=x,mark options={color=myred}]
  table[row sep=crcr]{
4.0832   0.034068 \\
5.7043   0.0084521 \\
8.7622   0.0011048 \\
16.4893   0.00015153 \\
28.7396   2.4993e-05 \\
50.7883   4.1286e-06 \\
90.6731   7.3439e-07 \\
153.862   1.2595e-07 \\
};

\addplot [color=myblue,solid,line width=0.75pt, mark size=3pt,mark=+,mark options={color=myblue}]
  table[row sep=crcr]{
4.0832   0.00099012 \\
5.7043   0.0003051 \\
8.7622   5.384e-05 \\
16.4893   8.3116e-06 \\
28.7396   1.5759e-06 \\
50.7883   2.7435e-07 \\
90.6731   4.8961e-08 \\
153.862   8.3706e-09 \\
};

\addplot [color=myblue,solid,line width=0.75pt, mark size=3pt,mark=o,mark options={color=myblue}]
  table[row sep=crcr]{
4.0832   0.00099012 \\
5.7043   0.0003051 \\
8.7622   5.384e-05 \\
16.4893   8.3116e-06 \\
28.7396   1.5759e-06 \\
50.7883   2.7435e-07 \\
90.6731   4.8961e-08 \\
153.862   8.3706e-09 \\
};

\addplot [color=myblue,solid,line width=0.75pt, mark size=3pt,mark=x,mark options={color=myblue}]
  table[row sep=crcr]{
4.0832   0.00099012 \\
5.7043   0.0003051 \\
8.7622   5.384e-05 \\
16.4893   8.3117e-06 \\
28.7396   1.5759e-06 \\
50.7883   2.7437e-07 \\
90.6731   4.8974e-08 \\
153.862   8.3854e-09 \\
};

\addplot [color=mygreen,solid,line width=0.75pt, mark size=3pt,mark=+,mark options={color=mygreen}]
  table[row sep=crcr]{
4.0832   0.00099765 \\
5.7043   0.00030565 \\
8.7622   5.365e-05 \\
16.4893   8.2731e-06 \\
28.7396   1.5698e-06 \\
50.7883   2.7463e-07 \\
90.6731   4.9842e-08 \\
153.862   8.9786e-09 \\
};

\addplot [color=mygreen,solid,line width=0.75pt, mark size=3pt,mark=o,mark options={color=mygreen}]
  table[row sep=crcr]{
4.0832   0.00099765 \\
5.7043   0.00030565 \\
8.7622   5.365e-05 \\
16.4893   8.2731e-06 \\
28.7396   1.5698e-06 \\
50.7883   2.7463e-07 \\
90.6731   4.9842e-08 \\
153.862   8.9787e-09 \\
};

\addplot [color=mygreen,solid,line width=0.75pt, mark size=3pt,mark=x,mark options={color=mygreen}]
  table[row sep=crcr]{
4.0832   0.00099765 \\
5.7043   0.00030565 \\
8.7622   5.365e-05 \\
16.4893   8.2731e-06 \\
28.7396   1.5698e-06 \\
50.7883   2.7462e-07 \\
90.6731   4.9836e-08 \\
153.862   8.968e-09 \\
};

\addplot [color=black,solid,line width=0.5pt]
  table[row sep=crcr]{
 90.6731     1.4658e-06 \\
 153.862     1.4658e-06 \\
 153.862     3e-7 \\
 90.6731     1.4658e-06 \\ 
};
\node[right, align=left, text=black, font=\footnotesize]
at (axis cs:153.862,7e-7) {3}; 

\end{axis}
\end{tikzpicture}
\end{subfigure}
\begin{subfigure}[b]{0.59\textwidth}
\begin{tikzpicture}
\begin{axis}[%
width=5cm,
height=3.8cm,
at={(0\textwidth,0\textwidth)},
scale only axis,
xmode=log,
xmin=3,
xmax=250,
xminorticks=true,
xlabel={\footnotesize $1/h$},
ymode=log,
ymin=3e-5,
ymax=3,
yminorticks=true,
ylabel={\footnotesize $dG$-errors},
legend style={draw=none,fill=none,legend cell align=left},
legend pos=outer north east,
]
\addplot [color=myred,solid,line width=0.75pt, mark size=3pt,mark=+, mark options={color=myred}]
  table[row sep=crcr]{
4.0832   1.8614 \\
5.7043   0.93286 \\
8.7622   0.31197 \\
16.4893   0.097144 \\
28.7396   0.031693 \\
50.7883   0.0099361 \\
90.6731   0.0031823 \\
153.862   0.00098225 \\
};
\addlegendentry{\scriptsize $\mathbf{u}_h$(FHM)}

\addplot [color=myred,solid,line width=0.75pt, mark size=3pt,mark=o, mark options={color=myred}]
  table[row sep=crcr]{
4.0832   1.8614 \\
5.7043   0.93286 \\
8.7622   0.31197 \\
16.4893   0.097144 \\
28.7396   0.031693 \\
50.7883   0.0099361 \\
90.6731   0.0031823 \\
153.862   0.00098225 \\
};
\addlegendentry{\scriptsize $\mathbf{u}_h$(FM-H)}

\addplot [color=myred,solid,line width=0.75pt, mark size=3pt,mark=x, mark options={color=myred}]
  table[row sep=crcr]{
4.0832   1.8614 \\
5.7043   0.93286 \\
8.7622   0.31197 \\
16.4893   0.097144 \\
28.7396   0.031693 \\
50.7883   0.0099361 \\
90.6731   0.0031823 \\
153.862   0.00098225 \\
};
\addlegendentry{\scriptsize $\mathbf{u}_h$(F-H-M)}

\addplot [color=myblue,solid,line width=0.75pt, mark size=3pt,mark=+,mark options={color=myblue}]
  table[row sep=crcr]{
4.0832   0.098042 \\
5.7043   0.050867 \\
8.7622   0.017225 \\
16.4893   0.005231 \\
28.7396   0.0017175 \\
50.7883   0.00053649 \\
90.6731   0.00017144 \\
153.862   5.2942e-05 \\
};
\addlegendentry{\scriptsize $p_h$(FHM)}

\addplot [color=myblue,solid,line width=0.75pt, mark size=3pt,mark=o,mark options={color=myblue}]
  table[row sep=crcr]{
4.0832   0.098042 \\
5.7043   0.050867 \\
8.7622   0.017225 \\
16.4893   0.005231 \\
28.7396   0.0017175 \\
50.7883   0.00053649 \\
90.6731   0.00017144 \\
153.862   5.2942e-05 \\
};
\addlegendentry{\scriptsize $p_h$(FM-H)}

\addplot [color=myblue,solid,line width=0.75pt, mark size=3pt,mark=x,mark options={color=myblue}]
  table[row sep=crcr]{
4.0832   0.098042 \\
5.7043   0.050867 \\
8.7622   0.017225 \\
16.4893   0.005231 \\
28.7396   0.0017175 \\
50.7883   0.00053649 \\
90.6731   0.00017144 \\
153.862   5.2942e-05 \\
};
\addlegendentry{\scriptsize $p_h$(F-H-M)}

\addplot [color=mygreen,solid,line width=0.75pt, mark size=3pt, mark=+, mark options={color=mygreen}]
  table[row sep=crcr]{
4.0832   0.098061 \\
5.7043   0.050869 \\
8.7622   0.017225 \\
16.4893   0.0052309 \\
28.7396   0.0017175 \\
50.7883   0.00053648 \\
90.6731   0.00017144 \\
153.862   5.2941e-05 \\
};
\addlegendentry{\scriptsize $T_h$(FHM)}

\addplot [color=mygreen,solid,line width=0.75pt, mark size=3pt, mark=o, mark options={color=mygreen}]
  table[row sep=crcr]{
4.0832   0.098061 \\
5.7043   0.050869 \\
8.7622   0.017225 \\
16.4893   0.0052309 \\
28.7396   0.0017175 \\
50.7883   0.00053648 \\
90.6731   0.00017144 \\
153.862   5.2941e-05 \\
};
\addlegendentry{\scriptsize $T_h$(FM-H)}

\addplot [color=mygreen,solid,line width=0.75pt, mark size=3pt, mark=x, mark options={color=mygreen}]
  table[row sep=crcr]{
4.0832   0.098061 \\
5.7043   0.050869 \\
8.7622   0.017225 \\
16.4893   0.0052309 \\
28.7396   0.0017175 \\
50.7883   0.00053648 \\
90.6731   0.00017144 \\
153.862   5.2941e-05 \\
};
\addlegendentry{\scriptsize $T_h$(F-H-M)}

\addplot [color=black,solid,line width=0.5pt]
  table[row sep=crcr]{
 90.6731     0.0058 \\
 153.862     0.0058 \\
 153.862     2e-3 \\
 90.6731     0.0058 \\ 
};
\node[right, align=left, text=black, font=\footnotesize]
at (axis cs:153.862,0.0035) {2}; 

\end{axis}
\end{tikzpicture}
\end{subfigure}
\end{subfigure}

\caption{Convergence test vs $h$ in 2D: computed errors in $L^2$-norm (left column) and $dG$-norms (right column) versus $1/h$ (\textit{log-log} scale).}
\label{fig:ConvH_2D}

\end{figure}
\end{center}

%% file: Fig/ConvP.tikz
\begin{center}
\begin{figure}[ht]
\begin{subfigure}[b]{1\textwidth}
\begin{subfigure}[b]{0.4\textwidth}
\begin{tikzpicture}
\begin{axis}[%
width=5cm,
height=3.8cm,
at={(0\textwidth,0\textwidth)},
scale only axis,
xmin=0.001,
xmax=9,
xminorticks=true,
xlabel={\footnotesize $1/h$},
ymode=log,
ymin=1e-13,
ymax=1,
yminorticks=true,
ylabel={\footnotesize $L^2$-errors},
legend style={draw=none,fill=none,legend cell align=left},
legend pos=outer north east,
]
\addplot [color=myred,solid,line width=0.75pt, mark size=3pt,mark=+, mark options={color=myred}]
  table[row sep=crcr]{
1   0.46987 \\
2   0.032069 \\
3   0.001775 \\
4   0.00012403 \\
5   8.7659e-06 \\
6   4.6169e-07 \\
7   2.3202e-08 \\
8   9.8371e-10 \\
};

\addplot [color=myred,solid,line width=0.75pt, mark size=3pt,mark=o, mark options={color=myred}]
  table[row sep=crcr]{
1   0.46987 \\
2   0.032069 \\
3   0.001775 \\
4   0.00012403 \\
5   8.7659e-06 \\
6   4.6169e-07 \\
7   2.3202e-08 \\
8   9.8372e-10 \\
};

\addplot [color=myred,solid,line width=0.75pt, mark size=3pt,mark=x,mark options={color=myred}]
  table[row sep=crcr]{
1   0.46987 \\
2   0.032069 \\
3   0.001775 \\
4   0.00012403 \\
5   8.7659e-06 \\
6   4.6169e-07 \\
7   2.3202e-08 \\
8   9.8371e-10 \\
};

\addplot [color=myblue,solid,line width=0.75pt, mark size=3pt,mark=+,mark options={color=myblue}]
  table[row sep=crcr]{
1   0.068464 \\
2   0.00098503 \\
3   6.3894e-05 \\
4   2.0174e-06 \\
5   8.4352e-08 \\
6   1.785e-09 \\
7   5.9571e-11 \\
8   1.1552e-12 \\
};

\addplot [color=myblue,solid,line width=0.75pt, mark size=3pt,mark=o,mark options={color=myblue}]
  table[row sep=crcr]{
1   0.068464 \\
2   0.00098503 \\
3   6.3894e-05 \\
4   2.0174e-06 \\
5   8.4352e-08 \\
6   1.785e-09 \\
7   5.9575e-11 \\
8   1.3747e-12 \\
};

\addplot [color=myblue,solid,line width=0.75pt, mark size=3pt,mark=x,mark options={color=myblue}]
  table[row sep=crcr]{
1   0.068464 \\
2   0.00098503 \\
3   6.3894e-05 \\
4   2.0174e-06 \\
5   8.4352e-08 \\
6   1.785e-09 \\
7   6.294e-11 \\
8   2.1538e-11 \\
};

\addplot [color=mygreen,solid,line width=0.75pt, mark size=3pt,mark=+,mark options={color=mygreen}]
  table[row sep=crcr]{
1   0.072162 \\
2   0.00099261 \\
3   6.3906e-05 \\
4   2.018e-06 \\
5   8.4361e-08 \\
6   1.7849e-09 \\
7   5.9573e-11 \\
8   1.0198e-12 \\
};

\addplot [color=mygreen,solid,line width=0.75pt, mark size=3pt,mark=o,mark options={color=mygreen}]
  table[row sep=crcr]{
1   0.072162 \\
2   0.00099261 \\
3   6.3906e-05 \\
4   2.018e-06 \\
5   8.4361e-08 \\
6   1.7849e-09 \\
7   5.9571e-11 \\
8   1.0493e-12 \\
};

\addplot [color=mygreen,solid,line width=0.75pt, mark size=3pt,mark=x,mark options={color=mygreen}]
  table[row sep=crcr]{
1   0.072162 \\
2   0.00099261 \\
3   6.3906e-05 \\
4   2.018e-06 \\
5   8.4361e-08 \\
6   1.7849e-09 \\
7   6.0003e-11 \\
8   7.4501e-12 \\
};

\addplot [color=black,dashed,line width=0.75pt]
  table[row sep=crcr]{
1   0.024894 \\
2   0.0012394 \\
3   6.1705e-05 \\
4   3.0721e-06 \\
5   1.5295e-07 \\
6   7.615e-09 \\
7   3.7913e-10 \\
8   1.8876e-11 \\
};

\end{axis}
\end{tikzpicture}
\end{subfigure}
\begin{subfigure}[b]{0.59\textwidth}
\begin{tikzpicture}
\begin{axis}[%
width=5cm,
height=3.8cm,
at={(0\textwidth,0\textwidth)},
scale only axis,
xmin=0.001,
xmax=9,
xminorticks=true,
xlabel={\footnotesize $1/h$},
ymode=log,
ymin=3e-11,
ymax=6,
yminorticks=true,
ylabel={\footnotesize Energy-errors},
legend style={draw=none,fill=none,legend cell align=left},
legend pos=outer north east,
]
\addplot [color=myred,solid,line width=0.75pt, mark size=3pt,mark=+, mark options={color=myred}]
  table[row sep=crcr]{
1   5.4236 \\
2   1.9045 \\
3   0.32648 \\
4   0.041624 \\
5   0.0038393 \\
6   0.00028836 \\
7   1.8951e-05 \\
8   1.0282e-06 \\
};
\addlegendentry{\scriptsize $\mathbf{u}_h$(FHM)}

\addplot [color=myred,solid,line width=0.75pt, mark size=3pt,mark=o, mark options={color=myred}]
  table[row sep=crcr]{
1   5.4236 \\
2   1.9045 \\
3   0.32648 \\
4   0.041624 \\
5   0.0038393 \\
6   0.00028836 \\
7   1.8951e-05 \\
8   1.0282e-06 \\
};
\addlegendentry{\scriptsize $\mathbf{u}_h$(FM-H)}

\addplot [color=myred,solid,line width=0.75pt, mark size=3pt,mark=x, mark options={color=myred}]
  table[row sep=crcr]{
1   5.4236 \\
2   1.9045 \\
3   0.32648 \\
4   0.041624 \\
5   0.0038393 \\
6   0.00028836 \\
7   1.8951e-05 \\
8   1.0282e-06 \\
};
\addlegendentry{\scriptsize $\mathbf{u}_h$(F-H-M)}

\addplot [color=myblue,solid,line width=0.75pt, mark size=3pt,mark=+,mark options={color=myblue}]
  table[row sep=crcr]{
1   0.39487 \\
2   0.09804 \\
3   0.0097089 \\
4   0.00058358 \\
5   2.9751e-05 \\
6   9.9303e-07 \\
7   3.8304e-08 \\
8   8.6709e-10 \\
};
\addlegendentry{\scriptsize $p_h$(FHM)}

\addplot [color=myblue,solid,line width=0.75pt, mark size=3pt,mark=o,mark options={color=myblue}]
  table[row sep=crcr]{
1   0.39487 \\
2   0.09804 \\
3   0.0097089 \\
4   0.00058358 \\
5   2.9751e-05 \\
6   9.9303e-07 \\
7   3.8304e-08 \\
8   8.6695e-10 \\
};
\addlegendentry{\scriptsize $p_h$(FM-H)}

\addplot [color=myblue,solid,line width=0.75pt, mark size=3pt,mark=x,mark options={color=myblue}]
  table[row sep=crcr]{
1   0.39487 \\
2   0.09804 \\
3   0.0097089 \\
4   0.00058358 \\
5   2.9751e-05 \\
6   9.9303e-07 \\
7   3.8305e-08 \\
8   8.7181e-10 \\
};
\addlegendentry{\scriptsize $p_h$(F-H-M)}

\addplot [color=mygreen,solid,line width=0.75pt, mark size=3pt,mark=+,mark options={color=mygreen}]
  table[row sep=crcr]{
1   0.40988 \\
2   0.098059 \\
3   0.0097096 \\
4   0.00058359 \\
5   2.9752e-05 \\
6   9.9303e-07 \\
7   3.8305e-08 \\
8   8.6669e-10 \\
};
\addlegendentry{\scriptsize $T_h$(FHM)}

\addplot [color=mygreen,solid,line width=0.75pt, mark size=3pt,mark=o,mark options={color=mygreen}]
  table[row sep=crcr]{
1   0.40988 \\
2   0.098059 \\
3   0.0097096 \\
4   0.00058359 \\
5   2.9752e-05 \\
6   9.9303e-07 \\
7   3.8305e-08 \\
8   8.6723e-10 \\
};
\addlegendentry{\scriptsize $T_h$(FM-H)}

\addplot [color=mygreen,solid,line width=0.75pt, mark size=3pt,mark=x,mark options={color=mygreen}]
  table[row sep=crcr]{
1   0.40988 \\
2   0.098059 \\
3   0.0097096 \\
4   0.00058359 \\
5   2.9752e-05 \\
6   9.9303e-07 \\
7   3.8305e-08 \\
8   8.6741e-10 \\
};
\addlegendentry{\scriptsize $T_h$(F-H-M)}

\addplot [color=black,dashed,line width=0.75pt]
  table[row sep=crcr]{
1   0.049787 \\
2   0.0024788 \\
3   0.00012341 \\
4   6.1442e-06 \\
5   3.059e-07 \\
6   1.523e-08 \\
7   7.5826e-10 \\
8   3.7751e-11 \\
};
\addlegendentry{$\exp{(-3\ell)}$}

\end{axis}
\end{tikzpicture}
\end{subfigure}
\end{subfigure}

\caption{Convergence test vs $\ell$ in 2D: computed errors in $L^2$-norm (left column) and $dG$-norms (right column) versus $1/h$ (\textit{log-log} scale).}
\label{fig:ConvP}

\end{figure}
\end{center}

%% file: Fig/ConvH_3D.tikz
\begin{center}
\begin{figure}[ht]
\begin{subfigure}[b]{1\textwidth}
\begin{subfigure}[b]{0.4\textwidth}
\begin{tikzpicture}
\begin{axis}[%
width=5cm,
height=3.8cm,
at={(0\textwidth,0\textwidth)},
scale only axis,
xmode=log,
xmin=1,
xmax=23,
xminorticks=true,
xlabel={\footnotesize $1/h$},
ymode=log,
ymin=5e-4,
ymax=1.1,
yminorticks=true,
ylabel={\footnotesize $L^2$-errors},
legend style={draw=none,fill=none,legend cell align=left},
legend pos=outer north east,
]

\addplot [color=myred,solid,line width=0.75pt, mark size=3pt,mark=+, mark options={color=myred}]
  table[row sep=crcr]{
1.1547   0.93708435 \\
3.4641   0.27128733 \\
5.7735   0.1223474  \\
8.0829   0.06802111 \\
10.3923  0.04290072 \\
12.7017  0.0294019 \\
15.0111  0.02136218 \\
17.3205  0.01620398\\
};

\addplot [color=myred,solid,line width=0.75pt, mark size=3pt,mark=o, mark options={color=myred}]
  table[row sep=crcr]{
1.1547   0.93708435  \\
3.4641   0.27128738 \\
5.7735   0.12234739  \\
8.0829   0.06802114 \\
10.3923   0.04290076  \\
12.7017  0.02940185 \\
15.0111  0.021362 \\
17.3205  0.01620392\\
};

\addplot [color=myred,solid,line width=0.75pt, mark size=3pt,mark=x,mark options={color=myred}]
  table[row sep=crcr]{
1.1547  0.9370835   \\
3.4641  0.27128815  \\
5.7735  0.12234743   \\
8.0829  0.06802114 \\
10.3923 0.04290083   \\
12.7017 0.02940189 \\
15.0111 0.02136214  \\
17.3205 0.01620394\\
};

\addplot [color=myblue,solid,line width=0.75pt, mark size=3pt,mark=+,mark options={color=myblue}]
  table[row sep=crcr]{
1.1547  0.07785625  \\
3.4641 0.01556733  \\
5.7735 0.00638695\\
8.0829  0.00343067   \\
10.3923 0.00213139  \\
12.7017 0.00144995 \\
15.0111 0.00104915  \\
17.3205 0.00079393\\
};

\addplot [color=myblue,solid,line width=0.75pt, mark size=3pt,mark=o,mark options={color=myblue}]
  table[row sep=crcr]{
1.1547  0.07785133  \\
3.4641  0.01556743  \\
5.7735  0.00638695  \\
8.0829  0.00343067  \\
10.3923 0.00213139  \\
12.7017 0.00144995 \\
15.0111 0.00104977  \\
17.3205 0.00079428 \\
};

\addplot [color=myblue,solid,line width=0.75pt, mark size=3pt,mark=x,mark options={color=myblue}]
  table[row sep=crcr]{
1.1547  0.0778509  \\
3.4641  0.01556825   \\
5.7735  0.006387  \\
8.0829   0.00343063    \\
10.3923  0.0021313  \\
12.7017 0.00144979 \\
15.0111 0.00104866  \\
17.3205  0.00079337\\
};

\addplot [color=mygreen,solid,line width=0.75pt, mark size=3pt,mark=+,mark options={color=mygreen}]
  table[row sep=crcr]{
1.1547  0.07959291  \\
3.4641  0.01630031  \\
5.7735  0.00669192  \\
8.0829  0.00359277  \\
10.3923 0.00223107   \\
12.7017 0.00151721 \\
15.0111  0.00109744  \\
17.3205 0.00083044 \\
};

\addplot [color=mygreen,solid,line width=0.75pt, mark size=3pt,mark=o,mark options={color=mygreen}]
  table[row sep=crcr]{
1.1547  0.07958735 \\
3.4641  0.01630078   \\
5.7735  0.00669204   \\
8.0829  0.00359268  \\
10.3923 0.00223077  \\
12.7017 0.00151571 \\
15.0111 0.00109589  \\
17.3205 0.00082783 \\
};

\addplot [color=mygreen,solid,line width=0.75pt, mark size=3pt,mark=x,mark options={color=mygreen}]
  table[row sep=crcr]{
1.1547  0.07958729  \\
3.4641  0.0163008 \\
5.7735   0.00669203  \\
8.0829   0.00359268  \\
10.3923 0.00223079   \\
12.7017 0.00151573 \\
15.0111 0.00109586  \\
17.3205 0.0008278 \\
};

\addplot [color=black,solid,line width=0.5pt]
  table[row sep=crcr]{
 12.7017     0.0024 \\
 17.3205     0.0024 \\
 17.3205     0.0013 \\
 12.7017     0.0024 \\ 
};
\node[right, align=left, text=black, font=\footnotesize]
at (axis cs:17.3205,0.0018) {2}; 

\end{axis}
\end{tikzpicture}
\end{subfigure}
\begin{subfigure}[b]{0.59\textwidth}
\begin{tikzpicture}
\begin{axis}[%
width=5cm,
height=3.8cm,
at={(0\textwidth,0\textwidth)},
scale only axis,
xmode=log,
xmin=1,
xmax=23,
xminorticks=true,
xlabel={\footnotesize $1/h$},
ymode=log,
ymin=0.08,
ymax=15,
yminorticks=true,
ylabel={\footnotesize $dG$-errors},
legend style={draw=none,fill=none,legend cell align=left},
legend pos=outer north east,
]
\addplot [color=myred,solid,line width=0.75pt, mark size=3pt,mark=+, mark options={color=myred}]
  table[row sep=crcr]{
1.1547  13.11040634   \\
3.4641  7.44614904  \\
5.7735  4.74950903    \\
8.0829  3.44430297    \\
10.3923 2.69394752   \\
12.7017 2.21028575 \\
15.0111 1.87352041  \\
17.3205 1.62586148 \\
};
\addlegendentry{\scriptsize $\mathbf{u}_h$(FHM)}

\addplot [color=myred,solid,line width=0.75pt, mark size=3pt,mark=o, mark options={color=myred}]
  table[row sep=crcr]{
1.1547  13.11040895  \\
3.4641  7.44614871  \\
5.7735   4.74950906   \\
8.0829  3.44430295    \\
10.3923 2.69394753    \\
12.7017 2.21028577 \\
15.0111 1.87352041   \\
17.3205  1.62586148 \\
};
\addlegendentry{\scriptsize $\mathbf{u}_h$(FM-H)}

\addplot [color=myred,solid,line width=0.75pt, mark size=3pt,mark=x, mark options={color=myred}]
  table[row sep=crcr]{
1.1547  13.1104104   \\
3.4641  7.44614807 \\
5.7735  4.74950901    \\
8.0829   3.44430293  \\
10.3923  2.69394749\\
12.7017   2.21028575 \\
15.0111  1.8735204   \\
17.3205  1.62586146\\
};
\addlegendentry{\scriptsize $\mathbf{u}_h$(F-H-M)}

\addplot [color=myblue,solid,line width=0.75pt, mark size=3pt,mark=+,mark options={color=myblue}]
  table[row sep=crcr]{
1.1547  1.24187574  \\
3.4641 0.52034241 \\
5.7735  0.32042927\\
8.0829  0.23068902 \\
10.3923  0.18007404\\
12.7017 0.14763581 \\
15.0111  0.12508774 \\
17.3205 0.10850947  \\
};
\addlegendentry{\scriptsize $p_h$(FHM)}

\addplot [color=myblue,solid,line width=0.75pt, mark size=3pt,mark=o,mark options={color=myblue}]
  table[row sep=crcr]{
1.1547  1.24187542  \\
3.4641 0.52034241 \\
5.7735 0.32042927 \\
8.0829  0.23068902 \\
10.3923 0.18007404    \\
12.7017 0.14763581 \\
15.0111  0.12508778 \\
17.3205  0.10850949 \\
};
\addlegendentry{\scriptsize $p_h$(FM-H)}

\addplot [color=myblue,solid,line width=0.75pt, mark size=3pt,mark=x,mark options={color=myblue}]
  table[row sep=crcr]{
1.1547  1.24187535  \\
3.4641  0.52034267  \\
5.7735  0.32042945\\
8.0829   0.23068902  \\
10.3923  0.18007397  \\
12.7017 0.14763577 \\
15.0111  0.12508763   \\
17.3205 0.10850918 \\
};
\addlegendentry{\scriptsize $p_h$(F-H-M)}

\addplot [color=mygreen,solid,line width=0.75pt, mark size=3pt, mark=+, mark options={color=mygreen}]
  table[row sep=crcr]{
1.1547  1.24203201  \\
3.4641  0.52025329  \\
5.7735  0.3204051  \\
8.0829   0.23067878 \\
10.3923  0.18006841 \\
12.7017   0.14763221 \\
15.0111  0.1250852   \\
17.3205 0.10850758  \\
};
\addlegendentry{\scriptsize $T_h$(FHM)}

\addplot [color=mygreen,solid,line width=0.75pt, mark size=3pt, mark=o, mark options={color=mygreen}]
  table[row sep=crcr]{
1.1547  1.24203116  \\
3.4641 0.52025373 \\
5.7735  0.32040529    \\
8.0829 0.23067864 \\
10.3923 0.18006818  \\
12.7017 0.14763187\\
15.0111  0.12508504 \\
17.3205  0.10850736  \\
};
\addlegendentry{\scriptsize $T_h$(FM-H)}

\addplot [color=mygreen,solid,line width=0.75pt, mark size=3pt, mark=x, mark options={color=mygreen}]
  table[row sep=crcr]{
1.1547  1.24203118  \\
3.4641 0.52025374  \\
5.7735 0.32040529   \\
8.0829  0.23067864  \\
10.3923 0.18006818  \\
12.7017 0.14763187\\
15.0111  0.12508503  \\
17.3205 0.10850736 \\
};
\addlegendentry{\scriptsize $T_h$(F-H-M)}

\addplot [color=black,solid,line width=0.5pt]
  table[row sep=crcr]{
 12.7017     0.2045 \\
 17.3205     0.2045 \\
 17.3205     0.15 \\
 12.7017     0.2045 \\ 
};
\node[right, align=left, text=black, font=\footnotesize]
at (axis cs:17.3205,0.17) {1}; 

\end{axis}
\end{tikzpicture}
\end{subfigure}
\end{subfigure}

\caption{Convergence test vs $h$ in 3D: computed errors in $L^2$-norm (left column) and $dG$-norms (right column) versus $1/h$ (\textit{log-log} scale).}
\label{fig:ConvH_3D}

\end{figure}
\end{center}

%% file: Fig/Rob.tikz
\begin{center}
\begin{figure}[h!]
\begin{tabular}{c c c}
\hspace{1.5cm} \large\textbf{dG-Errors u} & \hspace{2.5cm}
\large\textbf{dG-Errors p} & \hspace{2.5cm}
\large\textbf{dG-Errors T} \\
\end{tabular}

\begin{subfigure}[b]{1\textwidth}
\begin{subfigure}[b]{0.32\textwidth}
\begin{tikzpicture}
\begin{axis}[%
width=4.3cm,
height=3.8cm,
at={(0\textwidth,0\textwidth)},
scale only axis,
xmode=log,
xmin=3,
xmax=250,
xminorticks=true,
ymode=log,
ymin=5e-4,
ymax=3,
yminorticks=true,
legend style={draw=none,fill=none,legend cell align=left},
legend pos=south west,
]
\addplot [color=myred,solid,line width=0.75pt, mark size=3pt,mark=+, mark options={color=myred}]
  table[row sep=crcr]{
4.0832   1.8576 \\
5.7043   0.93187 \\
8.7622   0.31183 \\
16.4893   0.097131 \\
28.7396   0.031688 \\
50.7883   0.009935 \\
90.6731   0.0031821 \\
153.862   0.00098219 \\
};
\addlegendentry{\scriptsize \textbf{FHM}}

\addplot [color=myblue,solid,line width=0.75pt, mark size=3pt,mark=o,mark options={color=myblue}]
  table[row sep=crcr]{
4.0832   1.8576 \\
5.7043   0.93187 \\
8.7622   0.31183 \\
16.4893   0.097131 \\
28.7396   0.031688 \\
50.7883   0.009935 \\
90.6731   0.0031821 \\
153.862   0.00098219 \\
};
\addlegendentry{\scriptsize \textbf{FM-H}}

\addplot [color=mygreen,solid,line width=0.75pt, mark size=3pt,mark=x,mark options={color=mygreen}]
  table[row sep=crcr]{
4.0832   1.8576 \\
5.7043   0.93187 \\
8.7622   0.31183 \\
16.4893   0.097131 \\
28.7396   0.031688 \\
50.7883   0.009935 \\
90.6731   0.0031821 \\
153.862   0.00098219 \\
};
\addlegendentry{\scriptsize \textbf{F-H-M}}

\addplot [color=black,solid,line width=0.5pt]
  table[row sep=crcr]{
 90.6731     0.0058 \\
 153.862     0.0058 \\
 153.862     2e-3 \\
 90.6731     0.0058 \\ 
};
\node[right, align=left, text=black, font=\footnotesize]
at (axis cs:153.862,0.0035) {2}; 

\node[anchor=north west] at (rel axis cs:0.75,0.98) {\textbf{(i)}};

\end{axis}
\end{tikzpicture}
\end{subfigure}
\begin{subfigure}[b]{0.32\textwidth}
\begin{tikzpicture}
\begin{axis}[%
width=4.3cm,
height=3.8cm,
at={(0\textwidth,0\textwidth)},
scale only axis,
xmode=log,
xmin=3,
xmax=250,
xminorticks=true,
ymode=log,
ymin=3e-5,
ymax=0.2,
yminorticks=true,
legend style={draw=none,fill=none,legend cell align=left},
]

\addplot [color=myred,solid,line width=0.75pt, mark size=3pt,mark=+, mark options={color=myred}]
  table[row sep=crcr]{
4.0832   0.098048 \\
5.7043   0.050867 \\
8.7622   0.017225 \\
16.4893   0.0052311 \\
28.7396   0.0017176 \\
50.7883   0.00053649 \\
90.6731   0.00017144 \\
153.862   5.2942e-05 \\
};

\addplot [color=myblue,solid,line width=0.75pt, mark size=3pt,mark=o,mark options={color=myblue}]
  table[row sep=crcr]{
4.0832   0.098048 \\
5.7043   0.050867 \\
8.7622   0.017225 \\
16.4893   0.0052311 \\
28.7396   0.0017176 \\
50.7883   0.00053649 \\
90.6731   0.00017144 \\
153.862   5.2942e-05 \\
};

\addplot [color=mygreen,solid,line width=0.75pt, mark size=3pt,mark=x,mark options={color=mygreen}]
  table[row sep=crcr]{
4.0832   0.098048 \\
5.7043   0.050867 \\
8.7622   0.017225 \\
16.4893   0.0052311 \\
28.7396   0.0017176 \\
50.7883   0.00053649 \\
90.6731   0.00017144 \\
153.862   5.2942e-05 \\
};

\addplot [color=black,solid,line width=0.5pt]
  table[row sep=crcr]{
 90.6731     2.5915e-04 \\
 153.862     2.5915e-04 \\
 153.862     9e-5 \\
 90.6731     2.5915e-04 \\ 
};
\node[right, align=left, text=black, font=\footnotesize]
at (axis cs:153.862,1.7e-4) {2};  

\node[anchor=north west] at (rel axis cs:0.75,0.98) {\textbf{(i)}};

\end{axis}
\end{tikzpicture}
\end{subfigure}
\begin{subfigure}[b]{0.32\textwidth}
\begin{tikzpicture}
\begin{axis}[%
width=4.3cm,
height=3.8cm,
at={(0\textwidth,0\textwidth)},
scale only axis,
xmode=log,
xmin=3,
xmax=250,
xminorticks=true,
ymode=log,
ymin=3e-5,
ymax=0.2,
yminorticks=true,
legend style={draw=none,fill=none,legend cell align=left},
]

\addplot [color=myred,solid,line width=0.75pt, mark size=3pt,mark=+, mark options={color=myred}]
  table[row sep=crcr]{
4.0832   0.098065 \\
5.7043   0.05087 \\
8.7622   0.017225 \\
16.4893   0.0052309 \\
28.7396   0.0017175 \\
50.7883   0.00053648 \\
90.6731   0.00017143 \\
153.862   5.2941e-05 \\
};

\addplot [color=myblue,solid,line width=0.75pt, mark size=3pt,mark=o,mark options={color=myblue}]
  table[row sep=crcr]{
4.0832   0.098065 \\
5.7043   0.05087 \\
8.7622   0.017225 \\
16.4893   0.0052309 \\
28.7396   0.0017175 \\
50.7883   0.00053648 \\
90.6731   0.00017143 \\
153.862   5.2941e-05 \\
};

\addplot [color=mygreen,solid,line width=0.75pt, mark size=3pt,mark=x,mark options={color=mygreen}]
  table[row sep=crcr]{
4.0832   0.098065 \\
5.7043   0.05087 \\
8.7622   0.017225 \\
16.4893   0.0052309 \\
28.7396   0.0017175 \\
50.7883   0.00053648 \\
90.6731   0.00017143 \\
153.862   5.2941e-05 \\
};

\addplot [color=black,solid,line width=0.5pt]
  table[row sep=crcr]{
 90.6731     2.5915e-04 \\
 153.862     2.5915e-04 \\
 153.862     9e-5 \\
 90.6731     2.5915e-04 \\ 
};
\node[right, align=left, text=black, font=\footnotesize]
at (axis cs:153.862,1.7e-4) {2};  

\node[anchor=north west] at (rel axis cs:0.75,0.98) {\textbf{(i)}};

\end{axis}
\end{tikzpicture}
\end{subfigure}
\end{subfigure}
\vspace{-0.1cm}

\begin{subfigure}[b]{1\textwidth}
\begin{subfigure}[b]{0.32\textwidth}
\begin{tikzpicture}
\begin{axis}[%
width=4.3cm,
height=3.8cm,
at={(0\textwidth,0\textwidth)},
scale only axis,
xmode=log,
xmin=3,
xmax=250,
xminorticks=true,
ymode=log,
ymin=5e-4,
ymax=3,
yminorticks=true,
legend style={draw=none,fill=none,legend cell align=left},
legend pos=outer north east,
]

\addplot [color=myred,solid,line width=0.75pt, mark size=3pt,mark=+, mark options={color=myred}]
  table[row sep=crcr]{
4.0832   1.8614 \\
5.7043   0.93286 \\
8.7622   0.31197 \\
16.4893   0.097144 \\
28.7396   0.031693 \\
50.7883   0.0099361 \\
90.6731   0.0031823 \\
153.862   0.00098225 \\
};

\addplot [color=myblue,solid,line width=0.75pt, mark size=3pt,mark=o,mark options={color=myblue}]
  table[row sep=crcr]{
4.0832   1.8614 \\
5.7043   0.93286 \\
8.7622   0.31197 \\
16.4893   0.097144 \\
28.7396   0.031693 \\
50.7883   0.0099361 \\
90.6731   0.0031823 \\
153.862   0.00098225 \\
};

\addplot [color=mygreen,solid,line width=0.75pt, mark size=3pt,mark=x,mark options={color=mygreen}]
  table[row sep=crcr]{
4.0832   1.8614 \\
5.7043   0.93286 \\
8.7622   0.31197 \\
16.4893   0.097144 \\
28.7396   0.031693 \\
50.7883   0.0099361 \\
90.6731   0.0031823 \\
153.862   0.00098225 \\
};

\addplot [color=black,solid,line width=0.5pt]
  table[row sep=crcr]{
 90.6731     0.0058 \\
 153.862     0.0058 \\
 153.862     2e-3 \\
 90.6731     0.0058 \\ 
};
\node[right, align=left, text=black, font=\footnotesize]
at (axis cs:153.862,0.004) {2};  

\node[anchor=north west] at (rel axis cs:0.75,0.98) {\textbf{(ii)}};

\end{axis}
\end{tikzpicture}
\end{subfigure}
\begin{subfigure}[b]{0.32\textwidth}
\begin{tikzpicture}
\begin{axis}[%
width=4.3cm,
height=3.8cm,
at={(0\textwidth,0\textwidth)},
scale only axis,
xmode=log,
xmin=3,
xmax=250,
xminorticks=true,
ymode=log,
ymin=2e-8,
ymax=7e-5,
yminorticks=true,
legend style={draw=none,fill=none,legend cell align=left},
]

\addplot [color=myred,solid,line width=0.75pt, mark size=3pt,mark=+, mark options={color=myred}]
  table[row sep=crcr]{
4.0832   4.724e-05 \\
5.7043   2.4202e-05 \\
8.7622   8.0241e-06 \\
16.4893   2.3022e-06 \\
28.7396   6.9805e-07 \\
50.7883   2.148e-07 \\
90.6731   9.2321e-08 \\
153.862   3.6491e-08 \\
};

\addplot [color=myblue,solid,line width=0.75pt, mark size=3pt,mark=o,mark options={color=myblue}]
  table[row sep=crcr]{
4.0832   4.724e-05 \\
5.7043   2.4202e-05 \\
8.7622   8.0241e-06 \\
16.4893   2.3022e-06 \\
28.7396   6.9805e-07 \\
50.7883   2.148e-07 \\
90.6731   9.2321e-08 \\
153.862   3.6491e-08 \\
};

\addplot [color=mygreen,solid,line width=0.75pt, mark size=3pt,mark=x,mark options={color=mygreen}]
  table[row sep=crcr]{
4.0832   4.724e-05 \\
5.7043   2.4202e-05 \\
8.7622   8.0241e-06 \\
16.4893   2.3022e-06 \\
28.7396   6.9805e-07 \\
50.7883   2.148e-07 \\
90.6731   9.2321e-08 \\
153.862   3.6491e-08 \\
};

\addplot [color=black,solid,line width=0.5pt]
  table[row sep=crcr]{
90.6731      1.7277e-07 \\
 153.862     1.7277e-07 \\
 153.862     6e-8 \\
 90.6731     1.7277e-07 \\ 
};
\node[right, align=left, text=black, font=\footnotesize]
at (axis cs:153.862,1e-7) {2};  

\node[anchor=north west] at (rel axis cs:0.75,0.98) {\textbf{(ii)}};

\end{axis}
\end{tikzpicture}
\end{subfigure}
\begin{subfigure}[b]{0.32\textwidth}
\begin{tikzpicture}
\begin{axis}[%
width=4.3cm,
height=3.8cm,
at={(0\textwidth,0\textwidth)},
scale only axis,
xmode=log,
xmin=3,
xmax=250,
xminorticks=true,
ymode=log,
ymin=3e-5,
ymax=0.2,
yminorticks=true,
legend style={draw=none,fill=none,legend cell align=left},
]

\addplot [color=myred,solid,line width=0.75pt, mark size=3pt,mark=+, mark options={color=myred}]
  table[row sep=crcr]{
4.0832   0.098043 \\
5.7043   0.050867 \\
8.7622   0.017225 \\
16.4893   0.0052311 \\
28.7396   0.0017175 \\
50.7883   0.00053649 \\
90.6731   0.00017144 \\
153.862   5.2942e-05 \\
};

\addplot [color=myblue,solid,line width=0.75pt, mark size=3pt,mark=o,mark options={color=myblue}]
  table[row sep=crcr]{
4.0832   0.098043 \\
5.7043   0.050867 \\
8.7622   0.017225 \\
16.4893   0.0052311 \\
28.7396   0.0017175 \\
50.7883   0.00053649 \\
90.6731   0.00017144 \\
153.862   5.2942e-05 \\
};

\addplot [color=mygreen,solid,line width=0.75pt, mark size=3pt,mark=x,mark options={color=mygreen}]
  table[row sep=crcr]{
4.0832   0.098043 \\
5.7043   0.050867 \\
8.7622   0.017225 \\
16.4893   0.0052311 \\
28.7396   0.0017175 \\
50.7883   0.00053649 \\
90.6731   0.00017144 \\
153.862   5.2942e-05 \\
};

\addplot [color=black,solid,line width=0.5pt]
  table[row sep=crcr]{
 90.6731     2.5915e-04 \\
 153.862     2.5915e-04 \\
 153.862     9e-5 \\
 90.6731     2.5915e-04 \\ 
};
\node[right, align=left, text=black, font=\footnotesize]
at (axis cs:153.862,1.7e-4) {2};   

\node[anchor=north west] at (rel axis cs:0.75,0.98) {\textbf{(ii)}};

\end{axis}
\end{tikzpicture}
\end{subfigure}
\end{subfigure}
\vspace{-0.1cm}

\begin{subfigure}[b]{1\textwidth}
\begin{subfigure}[b]{0.32\textwidth}
\begin{tikzpicture}
\begin{axis}[%
width=4.3cm,
height=3.8cm,
at={(0\textwidth,0\textwidth)},
scale only axis,
xmode=log,
xmin=3,
xmax=250,
xminorticks=true,
ymode=log,
ymin=5e-4,
ymax=3,
yminorticks=true,
legend style={draw=none,fill=none,legend cell align=left},
legend pos=outer north east,
]

\addplot [color=myred,solid,line width=0.75pt, mark size=3pt,mark=+, mark options={color=myred}]
  table[row sep=crcr]{
4.0832   1.8614 \\
5.7043   0.93286 \\
8.7622   0.31197 \\
16.4893   0.097144 \\
28.7396   0.031693 \\
50.7883   0.0099361 \\
90.6731   0.0031823 \\
153.862   0.00098225 \\
};

\addplot [color=myblue,solid,line width=0.75pt, mark size=3pt,mark=o,mark options={color=myblue}]
  table[row sep=crcr]{
4.0832   1.8614 \\
5.7043   0.93286 \\
8.7622   0.31197 \\
16.4893   0.097144 \\
28.7396   0.031693 \\
50.7883   0.0099361 \\
90.6731   0.0031823 \\
153.862   0.00098225 \\
};

\addplot [color=mygreen,solid,line width=0.75pt, mark size=3pt,mark=x,mark options={color=mygreen}]
  table[row sep=crcr]{
4.0832   1.8614 \\
5.7043   0.93286 \\
8.7622   0.31197 \\
16.4893   0.097144 \\
28.7396   0.031693 \\
50.7883   0.0099361 \\
90.6731   0.0031823 \\
153.862   0.00098225 \\
};

\addplot [color=black,solid,line width=0.5pt]
  table[row sep=crcr]{
 90.6731     0.0058 \\
 153.862     0.0058 \\
 153.862     2e-3 \\
 90.6731     0.0058 \\ 
};
\node[right, align=left, text=black, font=\footnotesize]
at (axis cs:153.862,0.004) {2};  

\node[anchor=north west] at (rel axis cs:0.75,0.98) {\textbf{(iii)}};

\end{axis}
\end{tikzpicture}
\end{subfigure}
\begin{subfigure}[b]{0.32\textwidth}
\begin{tikzpicture}
\begin{axis}[%
width=4.3cm,
height=3.8cm,
at={(0\textwidth,0\textwidth)},
scale only axis,
xmode=log,
xmin=3,
xmax=250,
xminorticks=true,
ymode=log,
ymin=3e-5,
ymax=0.2,
yminorticks=true,
legend style={draw=none,fill=none,legend cell align=left},
]

\addplot [color=myred,solid,line width=0.75pt, mark size=3pt,mark=+, mark options={color=myred}]
  table[row sep=crcr]{
4.0832   0.098175 \\
5.7043   0.050858 \\
8.7622   0.017226 \\
16.4893   0.0052309 \\
28.7396   0.0017176 \\
50.7883   0.0005365 \\
90.6731   0.00017143 \\
153.862   5.2941e-05 \\
};

\addplot [color=myblue,solid,line width=0.75pt, mark size=3pt,mark=o,mark options={color=myblue}]
  table[row sep=crcr]{
4.0832   0.098175 \\
5.7043   0.050858 \\
8.7622   0.017226 \\
16.4893   0.0052309 \\
28.7396   0.0017176 \\
50.7883   0.0005365 \\
90.6731   0.00017143 \\
153.862   5.2941e-05 \\
};

\addplot [color=mygreen,solid,line width=0.75pt, mark size=3pt,mark=x,mark options={color=mygreen}]
  table[row sep=crcr]{
4.0832   0.098175 \\
5.7043   0.050858 \\
8.7622   0.017226 \\
16.4893   0.0052309 \\
28.7396   0.0017176 \\
50.7883   0.0005365 \\
90.6731   0.00017143 \\
153.862   5.2941e-05 \\
};

\addplot [color=black,solid,line width=0.5pt]
  table[row sep=crcr]{
 90.6731     2.5915e-04 \\
 153.862     2.5915e-04 \\
 153.862     9e-5 \\
 90.6731     2.5915e-04 \\ 
};
\node[right, align=left, text=black, font=\footnotesize]
at (axis cs:153.862,1.7e-4) {2};  

\node[anchor=north west] at (rel axis cs:0.75,0.98) {\textbf{(iii)}};

\end{axis}
\end{tikzpicture}
\end{subfigure}
\begin{subfigure}[b]{0.32\textwidth}
\begin{tikzpicture}
\begin{axis}[%
width=4.3cm,
height=3.8cm,
at={(0\textwidth,0\textwidth)},
scale only axis,
xmode=log,
xmin=3,
xmax=250,
xminorticks=true,
ymode=log,
ymin=2e-9,
ymax=8e-5,
yminorticks=true,
legend style={draw=none,fill=none,legend cell align=left},
]

\addplot [color=myred,solid,line width=0.75pt, mark size=3pt,mark=+, mark options={color=myred}]
  table[row sep=crcr]{
4.0832   4.629e-05 \\
5.7043   6.0256e-06 \\
8.7622   1.3076e-06 \\
16.4893   3.6202e-07 \\
28.7396   1.0278e-07 \\
50.7883   3.4444e-08 \\
90.6731   1.5346e-08 \\
153.862   4.3779e-09 \\
};

\addplot [color=myblue,solid,line width=0.75pt, mark size=3pt,mark=o,mark options={color=myblue}]
  table[row sep=crcr]{
4.0832   4.629e-05 \\
5.7043   6.0256e-06 \\
8.7622   1.3076e-06 \\
16.4893   3.6202e-07 \\
28.7396   1.0278e-07 \\
50.7883   3.4444e-08 \\
90.6731   1.5346e-08 \\
153.862   4.3778e-09 \\
};

\addplot [color=mygreen,solid,line width=0.75pt, mark size=3pt,mark=x,mark options={color=mygreen}]
  table[row sep=crcr]{
4.0832   4.629e-05 \\
5.7043   6.0256e-06 \\
8.7622   1.3076e-06 \\
16.4893   3.6202e-07 \\
28.7396   1.0278e-07 \\
50.7883   3.4444e-08 \\
90.6731   1.5346e-08 \\
153.862   4.3776e-09 \\
};

\addplot [color=black,solid,line width=0.5pt]
  table[row sep=crcr]{
 90.6731     2.5915e-08 \\
 153.862     2.5915e-08 \\
 153.862     9e-9 \\
 90.6731     2.5915e-08 \\ 
};
\node[right, align=left, text=black, font=\footnotesize]
at (axis cs:153.862,1.9e-8) {2};  

\node[anchor=north west] at (rel axis cs:0.75,0.98) {\textbf{(iii)}};

\end{axis}
\end{tikzpicture}
\end{subfigure}
\end{subfigure}
\vspace{-0.1cm}

\begin{subfigure}[b]{1\textwidth}
\begin{subfigure}[b]{0.32\textwidth}
\begin{tikzpicture}
\begin{axis}[%
width=4.3cm,
height=3.8cm,
at={(0\textwidth,0\textwidth)},
scale only axis,
xmode=log,
xmin=3,
xmax=250,
xminorticks=true,
ymode=log,
ymin=5e-4,
ymax=3,
yminorticks=true,
legend style={draw=none,fill=none,legend cell align=left},
legend pos=outer north east,
]

\addplot [color=myred,solid,line width=0.75pt, mark size=3pt,mark=+, mark options={color=myred}]
  table[row sep=crcr]{
4.0832   1.8614 \\
5.7043   0.93285 \\
8.7622   0.31197 \\
16.4893   0.097144 \\
28.7396   0.031693 \\
50.7883   0.0099361 \\
90.6731   0.0031823 \\
};

\addplot [color=myblue,solid,line width=0.75pt, mark size=3pt,mark=o,mark options={color=myblue}]
  table[row sep=crcr]{
4.0832   1.8614 \\
5.7043   0.93285 \\
8.7622   0.31197 \\
16.4893   0.097144 \\
28.7396   0.031693 \\
50.7883   0.0099361 \\
90.6731   0.0031823 \\
153.862   0.00098225 \\
};

\addplot [color=mygreen,solid,line width=0.75pt, mark size=3pt,mark=x,mark options={color=mygreen}]
  table[row sep=crcr]{
4.0832   1.8614 \\
5.7043   0.93285 \\
8.7622   0.31197 \\
16.4893   0.097144 \\
28.7396   0.031693 \\
50.7883   0.0099361 \\
90.6731   0.0031823 \\
153.862   0.00098225 \\
};

\addplot [color=black,solid,line width=0.5pt]
  table[row sep=crcr]{
 90.6731     0.0058 \\
 153.862     0.0058 \\
 153.862     2e-3 \\
 90.6731     0.0058 \\ 
};
\node[right, align=left, text=black, font=\footnotesize]
at (axis cs:153.862,0.004) {2};  

\node[anchor=north west] at (rel axis cs:0.75,0.98) {\textbf{(iv)}};

\end{axis}
\end{tikzpicture}
\end{subfigure}
\begin{subfigure}[b]{0.32\textwidth}
\begin{tikzpicture}
\begin{axis}[%
width=4.3cm,
height=3.8cm,
at={(0\textwidth,0\textwidth)},
scale only axis,
xmode=log,
xmin=3,
xmax=250,
xminorticks=true,
ymode=log,
ymin=4e-8,
ymax=2e-4,
yminorticks=true,
legend style={draw=none,fill=none,legend cell align=left},
]

\addplot [color=myred,solid,line width=0.75pt, mark size=3pt,mark=+, mark options={color=myred}]
  table[row sep=crcr]{
4.0832   8.4679e-05 \\
5.7043   4.3343e-05 \\
8.7622   1.4367e-05 \\
16.4893   4.1158e-06 \\
28.7396   1.2429e-06 \\
50.7883   3.7804e-07 \\
90.6731   1.5855e-07 \\
};

\addplot [color=myblue,solid,line width=0.75pt, mark size=3pt,mark=o,mark options={color=myblue}]
  table[row sep=crcr]{
4.0832   8.4679e-05 \\
5.7043   4.3343e-05 \\
8.7622   1.4367e-05 \\
16.4893   4.1158e-06 \\
28.7396   1.2429e-06 \\
50.7883   3.7804e-07 \\
90.6731   1.5855e-07 \\
153.862   6.0741e-08 \\
};

\addplot [color=mygreen,solid,line width=0.75pt, mark size=3pt,mark=x,mark options={color=mygreen}]
  table[row sep=crcr]{
4.0832   8.4679e-05 \\
5.7043   4.3343e-05 \\
8.7622   1.4367e-05 \\
16.4893   4.1158e-06 \\
28.7396   1.2429e-06 \\
50.7883   3.7804e-07 \\
90.6731   1.5855e-07 \\
153.862   6.0741e-08 \\
};

\addplot [color=black,solid,line width=0.5pt]
  table[row sep=crcr]{
 90.6731     2.8794e-07 \\
 153.862     2.8794e-07 \\
 153.862     1e-7 \\
 90.6731     2.8794e-07 \\ 
};
\node[right, align=left, text=black, font=\footnotesize]
at (axis cs:153.862,2e-7) {2};  

\node[anchor=north west] at (rel axis cs:0.75,0.98) {\textbf{(iv)}};

\end{axis}
\end{tikzpicture}
\end{subfigure}
\begin{subfigure}[b]{0.32\textwidth}
\begin{tikzpicture}
\begin{axis}[%
width=4.3cm,
height=3.8cm,
at={(0\textwidth,0\textwidth)},
scale only axis,
xmode=log,
xmin=3,
xmax=250,
xminorticks=true,
ymode=log,
ymin=5e-8,
ymax=2e-4,
yminorticks=true,
legend style={draw=none,fill=none,legend cell align=left},
]

\addplot [color=myred,solid,line width=0.75pt, mark size=3pt,mark=+, mark options={color=myred}]
  table[row sep=crcr]{
4.0832   0.0001129 \\
5.7043   5.7788e-05 \\
8.7622   1.9154e-05 \\
16.4893   5.4849e-06 \\
28.7396   1.6547e-06 \\
50.7883   5.0179e-07 \\
90.6731   2.091e-07 \\
};

\addplot [color=myblue,solid,line width=0.75pt, mark size=3pt,mark=o,mark options={color=myblue}]
  table[row sep=crcr]{
4.0832   0.0001129 \\
5.7043   5.7788e-05 \\
8.7622   1.9154e-05 \\
16.4893   5.4849e-06 \\
28.7396   1.6547e-06 \\
50.7883   5.0178e-07 \\
90.6731   2.091e-07 \\
153.862   7.9449e-08 \\
};

\addplot [color=mygreen,solid,line width=0.75pt, mark size=3pt,mark=x,mark options={color=mygreen}]
  table[row sep=crcr]{
4.0832   0.0001129 \\
5.7043   5.7788e-05 \\
8.7622   1.9154e-05 \\
16.4893   5.4849e-06 \\
28.7396   1.6547e-06 \\
50.7883   5.0178e-07 \\
90.6731   2.091e-07 \\
153.862   7.9449e-08 \\
};

\addplot [color=black,solid,line width=0.5pt]
  table[row sep=crcr]{
 90.6731     4.3191e-07 \\
 153.862     4.3191e-07 \\
 153.862     1.5e-7 \\
 90.6731     4.3191e-07 \\ 
};
\node[right, align=left, text=black, font=\footnotesize]
at (axis cs:153.862,3.5e-7) {2};  

\node[anchor=north west] at (rel axis cs:0.75,0.98) {\textbf{(iv)}};
\end{axis}

\end{tikzpicture}
\end{subfigure}
\end{subfigure}
\caption{Robustness test: computed errors in $dG$-norms versus $1/h$ (\textit{log-log} scale) for the displacement field (left column), pressure field (middle column), and temperature field (right column) for test case $\mathbf{(i)}$ (first row), test case $\mathbf{(ii)}$ (second row), test case $\mathbf{(iii)}$ (third row), and test case $\mathbf{(iv)}$ (fourth row).}
\label{fig:Rob}
\end{figure}
\end{center}